\documentclass{article}
\usepackage[numbers,sort&compress,super,comma]{natbib} % this must be before neurips_2026
\usepackage[preprint]{neurips_2026}   % arXiv build (make_arxiv.sh) swaps this to [preprint]

\usepackage[utf8]{inputenc}
\usepackage[english]{babel}
\usepackage{amssymb, amsmath, amsfonts}
\usepackage{mathtools, mathrsfs, dsfont}
\usepackage{bbm}
\usepackage[rightcaption,raggedright]{sidecap}
\sidecaptionvpos{figure}{c} % vertically align scfigure in sidecap
\usepackage{caption}
\usepackage{enumitem}
\usepackage{xcolor}
\usepackage{hyperref}
\usepackage{tikz}
\usetikzlibrary{arrows.meta, patterns, decorations.markings, calc}
\usepackage{times}
\usepackage{booktabs}
\usepackage{tabularx}
\usepackage{pdflscape}

\usepackage{amsthm}
\theoremstyle{plain}
\newtheorem{theorem}{Theorem}
\newtheorem{lemma}[theorem]{Lemma}
\newtheorem{corollary}[theorem]{Corollary}
\newtheorem{proposition}[theorem]{Proposition}

\theoremstyle{definition}
\newtheorem{definition}{Definition}

\newtheorem{example}{Example}

\theoremstyle{remark}
\newtheorem{remark}{Remark}
\newcommand{\reals}{\mathbb{R}}
\newcommand{\vol}{\operatorname{vol}}
\newcommand{\boa}{\operatorname{BoA}}
\newcommand{\mcX}{\mathcal{X}}
\newcommand{\mcF}{\mathcal{F}}
\newcommand{\mcM}{\mathcal{M}}
\newcommand{\mcV}{\mathcal{V}}
\newcommand{\mcE}{\mathcal{E}}
\newcommand{\mcB}{\mathcal{B}}
\DeclareMathOperator{\Lip}{Lip}
\DeclareMathOperator{\MSE}{MSE}
\DeclareMathOperator{\dist}{dist}

\author{
  \'Abel S\'agodi and Il Memming Park \\
  Champalimaud Foundation, Lisbon, Portugal \\
  \texttt{\{abel.sagodi, memming.park\}@research.fchampalimaud.org}
}

\title{\makebox[\textwidth][c]{%
  \begin{tabular}{@{}c@{}}
    Universal Approximation Theorems for Dynamical Systems \\
    with Infinite-Time Horizon Guarantees
  \end{tabular}}
}

\begin{document}

\maketitle

%==============================================================================
\begin{abstract}
Universal approximation theorems establish the expressive capacity of neural network architectures.
A trained dynamical model can pass every finite-time test yet fail in the long run in three ways: it can settle to the wrong steady state, drift out of phase, or collapse a continuum of memory states into isolated ones.
We consider models whose vector field is parameterized expressively enough to approximate any target vector field together with its Jacobian, e.g., Neural ODEs.
We prove that such models can avoid all three failures, matching a target system's trajectories over the infinite-time horizon $[0,\infty)$ on all but a small set of initial conditions ($\varepsilon$-$\delta$ closeness: trajectory error within $\varepsilon$ except for initial conditions of measure $< \delta$).
We cover three target classes, the dynamical primitives of decision-making, rhythm generation, and analog memory: 
structurally stable systems with a finite number of hyperbolic periodic orbits (fixed points or limit cycles via exact period matching), and continuous attractors that are normally hyperbolic manifolds of fixed points or of isochronous limit cycles (via discretization).
A temporal generalization bound, $\varepsilon$-$\delta$ closeness implies time-averaged $L^p$ error $\leq \varepsilon^p + \delta \cdot D^p$ for all $t \geq 0$, connects these guarantees to training metrics.
This is the first universal approximation framework for multistable dynamics over an unbounded horizon.
\end{abstract}
%\begin{keywords}
%Nonlinear dynamical systems, universal approximation, recurrent neural network, continuous attractor, stable limit cycle
%\end{keywords}

%==============================================================================
\section{Introduction}\label{sec:intro}

Universal approximation results provide rigorous guarantees about the expressive capacity of neural networks.
For dynamical systems, Recurrent Neural Networks (RNNs) are widely cited as universal approximators~\citep{funahashi1993approximation}, justifying their deployment in models of neural computation~\citep{durstewitz2023reconstructing,vyas2020ctd}.
However, a fundamental gap exists between this theoretical promise and biological reality.
Existing guarantees are strictly limited to \emph{finite} time horizons or systems with a globally stable equilibrium (the fading memory property). 
This restriction explicitly excludes \textbf{multistability}, i.e., the coexistence of multiple attractors, which is the dynamical basis of essential cognitive functions. 
Decision-making relies on selecting among distinct basins of attraction; working memory requires self-sustaining persistent activity; and neural oscillations (limit cycles) drive rhythmic motor control~\citep{townley2000existence, kag2020rnns, chang2019antisymmetricrnn, rapp1987periodic}. 
Consequently, current theories fail to address the very dynamical regimes required for computation.
Extending guarantees to infinite time is crucial for \textbf{temporal generalization}: ensuring models coherently replicate dynamics over indefinite durations rather than finite windows.
Three fundamental failure modes prevent a naive extension of finite-time results:
\begin{itemize}[leftmargin=*,itemsep=2pt]
\item \textbf{B-type error} (Basin mismatch): Small approximation errors near separatrices push trajectories into incorrect basins of attraction.
\item \textbf{P-type error} (Phase drift): For limit cycles, minute period mismatches cause unbounded phase divergence as $t \to \infty$.
\item \textbf{D-type error} (Discretization): Continuous attractors are not structurally stable; generic perturbations destroy their continua of marginally stable fixed points.
\end{itemize}
These topological obstructions call for more than standard Gr\"onwall-based error analysis, which yields exponentially growing bounds unusable for infinite time.

This work establishes universal approximation results for multistable dynamical systems over infinite time horizons.
Our approach exploits structural stability theory: Morse-Smale systems (a significant subset of all structurally stable systems~\citep{bonatti2011survey}) are robust to small $C^1$ perturbations, enabling infinite-time bounds.
For limit cycles, we additionally require \emph{exact period matching} via a localized correction procedure.
Continuous attractors are not structurally stable, so we \emph{tile} them with hyperbolic attractors spaced closer than $\varepsilon$ and approximate the resulting Morse-Smale system.

Our analysis adopts a learning-theoretic framework: we define \textbf{target classes} $\mcF$ of dynamical systems to be approximated, a \textbf{hypothesis class} $\hat{\mcF}$ of Neural ODEs (Definition~\ref{def:hypothesis_class}), and an \textbf{approximation criterion} ($\varepsilon$-$\delta$ closeness, Definition~\ref{def:eps_delta_close}).
Universal approximation means: for any $f \in \mcF$ and any $\varepsilon, \delta > 0$, there exists $\hat{f} \in \hat{\mcF}$ achieving $\varepsilon$-$\delta$ closeness over infinite time.

\subsection{Main Results (Informal)}

We establish universal approximation for three target classes. % with increasing generality.
Throughout, $\varepsilon$-$\delta$ closeness means: the volume of initial conditions with trajectory error exceeding $\varepsilon$ is less than $\delta$ (Definition~\ref{def:eps_delta_close}).

\begin{quote}
\textbf{Theorem FP} (Fixed Points, Informal; see Theorem~\ref{thm:uap_fp}).
\textit{For any Morse-Smale system with hyperbolic fixed points and any $\varepsilon, \delta > 0$, there exists a finite-size Neural ODE that is $\varepsilon$-$\delta$ close to the target for all $t \in [0, \infty)$.}
\end{quote}

\begin{quote}
\textbf{Theorem LC} (Limit Cycles, Informal; see Theorem~\ref{thm:uap_lc}).
\textit{For Morse-Smale systems with hyperbolic limit cycles, Neural ODEs achieve $\varepsilon$-$\delta$ closeness via exact period matching through localized vector field scaling.}
\end{quote}

\begin{quote}
\textbf{Theorem CA} (Continuous Attractors, Informal; see Theorem~\ref{thm:uap_ca}).
\textit{For systems with normally hyperbolic continuous attractors (line attractors, ring attractors, isochronous cylinders), Neural ODEs achieve $\varepsilon$-$\delta$ closeness via tiling: approximation by a dense grid of discrete attractors with spacing $< \varepsilon$.}
\end{quote}

Finally, we establish a temporal generalization bound to bridge these  guarantees to  training metrics:

\begin{quote}
\textbf{Theorem TG} (Temporal Generalization, Informal; see Theorem~\ref{thm:lp_bound}).
\textit{If $\hat{\varphi}$ is $\varepsilon$-$\delta$ close to $\varphi$, then the time-averaged $L^p$ error satisfies $\mathcal{E}_{p,\infty} \leq \varepsilon^p + \delta \cdot D^p$, where $D$ is the domain diameter.
This bridges topological guarantees to practical training metrics like Mean Squared Error (MSE).}
\end{quote}

\subsection{Contributions}

\begin{enumerate}[leftmargin=*,itemsep=2pt]
\item We establish \textbf{universal approximation results for multistable dynamical systems over infinite time horizons} (Theorems~\ref{thm:uap_fp}--\ref{thm:uap_ca}).
\item We prove that for each $(\varepsilon, \delta)$ a \textbf{Neural ODE of finite width} suffices (Section~\ref{sec:main}), while no single width suffices for all pairs.
\item We introduce \textbf{exact period matching} via localized vector field scaling to eliminate P-type error for limit cycles (Theorem~\ref{thm:uap_lc}).
\item We derive a \textbf{temporal generalization bound} linking $\varepsilon$-$\delta$ closeness to $L^p$ error (Theorem~\ref{thm:lp_bound}).
\end{enumerate}

Figure~\ref{fig:landscape} illustrates the landscape of universal approximation results.
Prior infinite-time theories were limited to fading memory systems (a single global attractor).
Our results extend to Morse-Smale systems and to the normally hyperbolic continuous attractors of $\mcF_{\mathrm{CA}}$~\citep{Sagodi2024a}.

\begin{SCfigure}[15][bthp]
\centering
\includegraphics[width=0.4\textwidth]{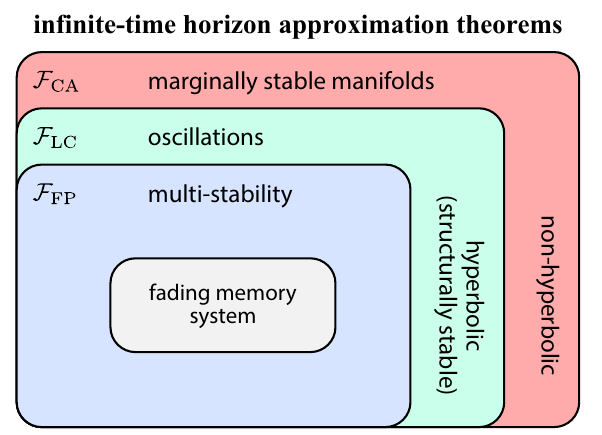}
\caption{The landscape of universal approximation for dynamical systems.
Existing infinite-time results require the fading memory property (FMP), excluding multistability; FMP systems with a hyperbolic equilibrium lie in $\mcF_{\mathrm{FP}}$.
We prove Neural ODEs are dense in Morse-Smale systems ($\mcF_{\mathrm{FP}}$, $\mcF_{\mathrm{LC}}$) and normally hyperbolic continuous attractors ($\mcF_{\mathrm{CA}}$), subject to isochrony for oscillatory manifolds.
}
\label{fig:landscape}
\end{SCfigure}

\paragraph{Scope.}
Our theorems cover Morse-Smale systems (finitely many hyperbolic fixed points and periodic orbits) and continuous attractors that are normally hyperbolic manifolds of fixed points or of isochronous limit cycles. They exclude chaotic targets, for which a positive Lyapunov exponent makes the infinite-horizon bound vacuous. We do not claim Morse-Smale systems are generic; they are open and dense in dimension two~\citep{peixoto1959ss}, but not dense in dimension three and higher~\citep{smale1966structurally}. Our proofs rely on openness rather than density: a learned field that is Morse-Smale keeps its qualitative dynamics under weight perturbation, which is what makes the guarantee meaningful. Our third target class, continuous attractors, is structurally unstable and a genericity argument would exclude it; we include it because it is a computational primitive, and Theorem~\ref{thm:uap_ca} shows it remains approximable.

\subsection{Related Work}

The theoretical capability of neural networks to model dynamical systems is well-established, but existing results largely bifurcate into two regimes: finite-time approximation for general systems, and infinite-time approximation for restricted classes of stable systems.

Universal approximation for RNNs on finite time horizons was established by~\citet{funahashi1993approximation} and \citet{li1992approximation}, however, these guarantees rely on Grönwall inequalities where error bounds grow exponentially with time, rendering them vacuous for long-term behavior~\citep{sontag1998learning}.
This lineage was expanded by \citet{kimura1998learning} and proven for various architectures by \citet{chow2000modeling}, extending to time-variant systems with fixed initial states \citep{li2005approximation}. 
This limitation is universal across architectures: it applies to discrete-time recurrence~\citep{jin1995universal, aguiar2023universal}, time-variant systems~\citep{li2005approximation}, and Neural ODEs~\citep{chen2018neural}.
While Neural ODEs are universal approximators for homeomorphisms~\citep{zhang2020approximation}, flow maps~\citep{li2022deep}, and functions on manifolds via geometric controllability~\citep{elamvazhuthi2023learning}, these results only approximate the diffeomorphism $\phi_t$ for a fixed $t$, providing no guarantee of topological equivalence or attractor preservation as $t \to \infty$. Recent work enforces exact fixed-point constraints for provable universality~\citep{pacifico2026exact}, but does not address the infinite-horizon trajectory guarantees or the limit-cycle and continuous-attractor cases we treat.

For infinite time, existing results require the \emph{fading memory property}~\citep{boyd1985fading}, which implies convergence to a unique equilibrium and excludes multistability.
Echo State Networks~\citep{jaeger2001echo} and reservoir computing~\citep{grigoryeva2018echo} achieve infinite-time universality only for fading memory systems.
The fading memory property ensures that the influence of past inputs decays asymptotically, allowing for infinite-time approximation of filters \citep{gonon2019reservoir} and State Space Models \citep{wang2024state}.
Similarly, \citet{nakamura2009approximation} extended approximation to infinite horizons only for systems with a globally asymptotically stable equilibrium. 
\citet{hanson2020universal} provided infinite-time results based on uniform asymptotic incremental stability, a condition that forces all trajectories to converge to each other.
Finally, while \citet{hart2020embedding} showed that ESNs can form attracting submanifolds topologically conjugate to structurally stable diffeomorphisms, their result is restricted to discrete-time maps and requires embedding the dynamics into a high-dimensional reservoir state space.
Recent Neural ODE architectures instead build in convergence by construction~\citep{mei2024controlsynth}, but the guaranteed convergence is again to a single stable equilibrium, precluding the multistability our target classes require.
Explicit universal approximation guarantees for multistable continuous-time flows in their native state space are lacking.

Alternatively, operator-theoretic approaches attempt to globalize dynamics by lifting the state to an infinite-dimensional space where evolution becomes linear. 
Koopman operator theory~\citep{koopman1931hamiltonian, mezic2005spectral} lifts dynamics to a space of observables, while Carleman linearization embeds polynomial nonlinearities into an infinite system of moments~\citep{carleman1932application, forets2017explicit}.
While finite-dimensional truncations like Extended Dynamic Mode Decomposition (EDMD)~\citep{williams2015data} can capture global features, they theoretically require infinite width to eliminate closure error and generally lack rigorous uniform trajectory guarantees~\citep{brunton2021modern}. Koopman eigenfunctions have also been used to identify the separatrices themselves~\citep{dabholkar2025finding}, the measure-zero loci where our B-type error concentrates.

Our work extends these results to infinite time for structurally stable systems.
For a comprehensive review of the literature on universal approximation for dynamical systems, including finite-time results, fading memory systems, and reservoir computing, see App.~\ref{app:literature}.

\begin{remark}[Direct vs.\ embedding approximation]
Finite-dimensional RNNs face a structural limitation: they cannot express an arbitrary number of attractors without increasing dimension~\citep{funahashi1993approximation,hwang2019lc,hwang2020lc,yi2003multistability}.
Consequently, RNN universality requires \emph{embedding} the target into a higher-dimensional latent space.
In contrast, Neural ODEs enable \emph{direct} approximation: the learned vector field operates in the same state space as the target, preserving the original coordinates and attractor geometry.
This distinction is crucial for mechanistic interpretability of learned dynamics.
\end{remark}

%==============================================================================
\section{Preliminaries}\label{sec:prelim}

Let $\mcX \subset \reals^n$ be a bounded open set.
We denote by $\mathfrak{X}^1(\mcX)$ the space of $C^1$ vector fields $f: \mcX \to \reals^n$ whose trajectories remain in $\mcX$ for all $t \geq 0$.
The \textbf{flow} $\varphi: \reals_{\geq 0} \times \mcX \to \mcX$ satisfies $\frac{d}{dt}\varphi(t, x_0) = f(\varphi(t, x_0))$ with $\varphi(0, x_0) = x_0$.
We identify each system $\dot{x} = f(x)$ with its vector field $f$.
Standard definitions from dynamical systems theory ($\omega$-limit sets, non-wandering sets, transversal intersections, separatrices, tubular neighborhoods) are given in App.~\ref{app:definitions}.

\begin{definition}[Hypothesis class with $C^1$ UAP]\label{def:hypothesis_class}
A class $\hat{\mcF}$ of $C^1$ functions $\hat{f}: \mcX \to \reals^n$ has the \textbf{$C^1$ universal approximation property} if for every $f \in C^1(\mcX, \reals^n)$ and every $\eta > 0$, there exists $\hat{f} \in \hat{\mcF}$ with $\|f - \hat{f}\|_{C^1} < \eta$, where $\|g\|_{C^1} \coloneqq \sup_x \|g(x)\| + \sup_x \|Dg(x)\|_{\mathrm{op}}$ controls both function values and derivatives.
\end{definition}

\begin{remark}
Standard feedforward networks with smooth activations (e.g., tanh, sigmoid) achieve $C^r$ approximation for any $r \geq 0$ on compact domains~\citep{hornik1991approximation}, with explicit width and depth bounds for the joint function-and-derivative ($C^1$) error in the tanh case~\citep{deryck2021approximation}.
The property is not tied to one activation: periodic-activation networks~\citep{sitzmann2020implicit}, Kolmogorov-Arnold spline networks~\citep{liu2024kan}, and piecewise-polynomial (ReLU$^k$) networks that approximate a function together with its derivatives~\citep{belomestny2023simultaneous} all satisfy Definition~\ref{def:hypothesis_class}; even plain ReLU networks qualify after mollification.
\end{remark}

\begin{definition}[Neural ODE]\label{def:node}
Let $\{f_\theta\}_{\theta \in \Theta}$ be a family of $C^1$ vector fields $f_\theta: \mcX \to \reals^n$ computed by a neural network architecture, with $\theta$ ranging over the weights and biases of all widths.
A \textbf{Neural ODE} is the dynamical system $\dot{x} = f_\theta(x)$ for some $\theta \in \Theta$, and $\hat{\mcF} = \{f_\theta\}_{\theta \in \Theta}$ is its hypothesis class.
\end{definition}

\begin{definition}[Hyperbolic fixed point]\label{def:hyperbolic_fp}
A fixed point $x^*$ of $\dot{x} = f(x)$ is \textbf{hyperbolic} if all eigenvalues of the Jacobian $Df(x^*)$ have nonzero real part.
\end{definition}

\begin{definition}[Hyperbolic periodic orbit]\label{def:hyperbolic_lc}
A periodic orbit $\gamma$ with period $T > 0$ is \textbf{hyperbolic} if no eigenvalue of the Poincar\'e return map (Definition~\ref{def:poincare_map}) on a transverse section has modulus 1.
\end{definition}

\begin{definition}[Morse-Smale system]\label{def:morse_smale}
A $C^1$ vector field $f$ is \textbf{Morse-Smale} if:
(1) the non-wandering set $\Omega(f)$ is a finite collection of hyperbolic fixed points and hyperbolic periodic orbits, where $x$ is \emph{non-wandering} if for every neighborhood $U$ of $x$ and every $T > 0$ there exists $t > T$ with $\varphi(t, U) \cap U \neq \emptyset$; and
(2) all stable and unstable manifolds intersect \emph{transversally}: at every intersection point $p$, the tangent spaces satisfy $T_p W^s + T_p W^u = \mathbb{R}^n$.
\end{definition}

\begin{theorem}[Palis-Smale~{\citep{palis1970ss}}]\label{thm:palis_smale}
Every Morse-Smale system is structurally stable: small $C^1$ perturbations yield topologically equivalent flows.
\end{theorem}

\begin{remark}[Morse-Smale systems]
Morse-Smale systems are ``generic well-behaved'' dynamics: all trajectories converge to finitely many stable fixed points or limit cycles, with no degenerate connections (homoclinic tangencies, heteroclinic cycles) that would cause bifurcations.
Condition (1) rules out chaotic attractors; condition (2) ensures structural stability.
Morse-Smale systems are not the only structurally stable class (Axiom A systems~\citep{smale1967differentiable} with strange attractors are also structurally stable), but our trajectory-tracking analysis does not extend to chaotic dynamics, where sensitive dependence makes trajectory-level identifiability fundamentally limited~\citep{gallo2026attractor} and data-driven Neural ODEs instead recover the attractor's statistics rather than individual trajectories~\citep{linot2022data,chesebro2026scientific}.
\end{remark}

\begin{theorem}[Stable Manifold Theorem~{\citep{chicone2006ode}}]\label{thm:stable_manifold}
Let $x^*$ be a hyperbolic fixed point of a $C^1$ vector field $f$ with $k$ stable eigenvalues (negative real part) and $n-k$ unstable eigenvalues.
Then:
\begin{enumerate}[nosep]
\item The \textbf{stable manifold} $W^s(x^*) = \{x : \varphi(t,x) \to x^* \text{ as } t \to +\infty\}$ is a $C^1$ embedded submanifold of dimension $k$.
\item The \textbf{unstable manifold} $W^u(x^*)$ is a $C^1$ embedded submanifold of dimension $n-k$.
\item Trajectories in $W^s(x^*)$ converge to $x^*$ exponentially.
\end{enumerate}
\end{theorem}

\begin{remark}
For saddle points ($0 < k < n$), the stable manifold has dimension $< n$, hence Lebesgue measure zero.
Since basin boundaries in Morse-Smale systems are finite unions of stable manifolds of saddles, they have measure zero, a key fact for controlling B-type error.
\end{remark}

\begin{definition}[Normally hyperbolic invariant manifold]\label{def:nhim}
A compact invariant manifold $\mcM$ is \textbf{normally hyperbolic} if the linearized flow contracts/expands transverse to $\mcM$ at exponential rates dominating any expansion/contraction tangent to $\mcM$~\citep{fenichel1971persistence}.
\end{definition}

\begin{definition}[Strictly inward-pointing]\label{def:inward_pointing}
Let $\mcX \subset \reals^n$ have smooth boundary $\partial \mcX$ with outward unit normal $\nu(x)$.
A vector field $f$ is \textbf{strictly inward-pointing} at $\partial \mcX$ if there exists $\gamma > 0$ such that $f(x) \cdot \nu(x) < -\gamma$ for all $x \in \partial \mcX$.
\end{definition}

\begin{remark}
The strictly inward-pointing condition ensures that any $C^1$-close approximation $\hat{f}$ with $\|f - \hat{f}\|_{C^0} < \gamma$ automatically satisfies $\hat{f}(x) \cdot \nu(x) < 0$ on $\partial \mcX$, so trajectories of $\hat{f}$ remain in $\mcX$ (Lemma~\ref{lem:forward_invariance}).
This can always be achieved by choosing $\mcX$ appropriately (e.g., a sublevel set of a Lyapunov-like function).
\end{remark}

%==============================================================================
\section{Approximation Metric}\label{sec:metric}

\begin{SCfigure}[10][t!bhp]
\centering
\includegraphics[width=0.3\textwidth]{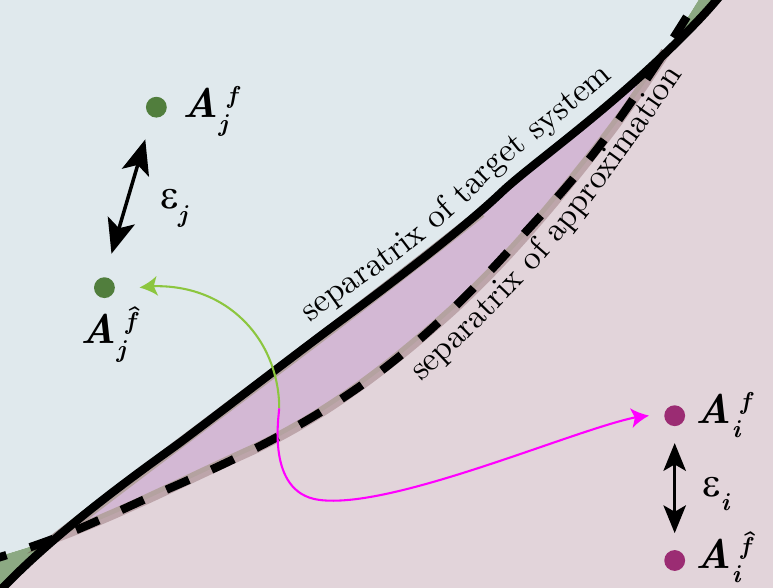}
\caption{Basin mismatch (B-type error) in a bistable system.
The target has stable fixed points $A_i^f$ and $A_j^f$ separated by a separatrix (solid).
The approximation has nearby fixed points with a shifted separatrix (dashed).
Initial conditions in the shaded region converge to different attractors in target vs.\ approximation, controlled by the $\delta$ parameter.
Near separatrices, perfect trajectory tracking is impossible; the $\delta$ term captures this unavoidable failure region.
}
\label{fig:separatrices}
\end{SCfigure}

Standard approximation metrics (uniform norm, $L^p$ norms) are unsuitable for infinite-time dynamical systems: near separatrices, trajectories can diverge to different attractors regardless of vector field closeness.
We adopt an $\varepsilon$-$\delta$ framework that decouples precision from reliability.

A distinct class of approximation errors, which we term \textit{basin errors}, arises when the target and the approximation assign the same initial condition $x_0$ to different attractors (i.e., $x_0 \in \boa(A_i^f)$ but $x_0 \in \boa(A_j^{\hat{f}})$ for $i \neq j$). 
For such points, the trajectory error does not remain bounded by a small $\varepsilon$ but converges to the distance between the two distinct attractors $\|A_i - A_j\|$, see Fig.~\ref{fig:separatrices}.
Our metric accepts this inherent limitation via the $\delta$ parameter.

\begin{definition}[$\varepsilon$-volume error]\label{def:eps_volume_error}
For flows $\varphi, \hat{\varphi}: \reals_{\geq 0} \times \mcX \to \mcX$, the \textbf{$\varepsilon$-volume error} is:
\begin{equation}\label{eq:eps_volume_error}
\|\varphi - \hat{\varphi}\|_\varepsilon \coloneqq \frac{1}{\vol(\mcX)} \int_{\mcX} \mathds{1}\left[ \sup_{t \geq 0} \|\varphi(t, x_0) - \hat{\varphi}(t, x_0)\| \geq \varepsilon \right] dx_0.
\end{equation}
\end{definition}

\begin{definition}[$\varepsilon$-$\delta$ closeness]\label{def:eps_delta_close}
Flows $\varphi$ and $\hat{\varphi}$ are \textbf{$\varepsilon$-$\delta$ close} if $\|\varphi - \hat{\varphi}\|_\varepsilon < \delta$.
\end{definition}

Viewed as a function of $\varepsilon$, $\delta(\varepsilon) \coloneqq \|\varphi - \hat\varphi\|_\varepsilon$ is the tail of the worst-case error $S(x_0) = \sup_{t \ge 0}\|\varphi(t,x_0) - \hat\varphi(t,x_0)\|$ under the uniform measure on initial conditions, and the flows are $\varepsilon$-$\delta$ close when this curve passes below $(\varepsilon, \delta)$.
Its crossing with the diagonal is the Ky Fan distance (Remark~\ref{rem:ky_fan}), and the pair $(\varepsilon,\delta)$ plays the role of accuracy and confidence in PAC learning~\citep{valiant1984theory}, here over initial conditions; for large $\varepsilon$, $\delta$ measures basin mismatch, the model-versus-target analogue of basin stability~\citep{menck2013basinstability}.

This $(\varepsilon, \delta)$-closeness is more than a practical heuristic; it induces a rigorous topology on the space of dynamical systems. 
Specifically, it corresponds to the topology of \textit{convergence in measure} with respect to the trajectory supremum norm, which is metrizable via the Ky Fan metric. 
We provide the topological proofs and formal basis properties in App.~\ref{sec:approximation_topology}.

%==============================================================================
\section{Main Results}\label{sec:main}

\subsection{Target Classes}

\begin{definition}[Target class $\mcF_{\mathrm{FP}}$]\label{def:target_fp}
$\mcF_{\mathrm{FP}} \subset \mathfrak{X}^1(\mcX)$ is the set of Morse-Smale vector fields whose $\omega$-limit sets (Definition~\ref{def:omega_limit}) consist solely of hyperbolic fixed points.
\end{definition}

\begin{definition}[Target class $\mcF_{\mathrm{LC}}$]\label{def:target_lc}
$\mcF_{\mathrm{LC}} \subset \mathfrak{X}^1(\mcX)$ is the set of Morse-Smale vector fields with at least one hyperbolic limit cycle attractor (Definition~\ref{def:attractor}).
\end{definition}

\begin{definition}[Isochronous manifold]\label{def:isochronous}
A normally hyperbolic invariant manifold $\mcM$ foliated by periodic orbits $\{\gamma_\theta\}_{\theta \in \Theta}$ is \textbf{isochronous} if all orbits have the same period: $T(\gamma_\theta) = T$ for all $\theta \in \Theta$.
\end{definition}

\begin{definition}[Target class $\mcF_{\mathrm{CA}}$]\label{def:target_ca}
$\mcF_{\mathrm{CA}} \subset \mathfrak{X}^1(\mcX)$ is the set of vector fields whose $\omega$-limit sets are: (i) hyperbolic fixed points, (ii) hyperbolic periodic orbits, (iii) normally hyperbolic manifolds of fixed points, or (iv) isochronous normally hyperbolic manifolds of limit cycles (Definition~\ref{def:isochronous}).
\end{definition}

\begin{remark}[Isochrony requirement]
Condition (iv) prevents P-type error from compounding D-type error. Tiling an oscillatory manifold by discrete limit cycles makes each tile inherit its local period, so a frequency gradient (periods varying along the attractor) leaves neighboring tiles mismatched and drives unbounded phase drift as $t \to \infty$; isochrony lets every tile be retuned to a common period through the correction of Theorem~\ref{thm:uap_lc}. For example, a cylinder $S^1 \times [0,1]$ whose orbit at height $z$ has period $T(z)$ is isochronous only when $T(z)$ is constant. %, and a nonconstant $T(z)$ defeats the tiling.
\end{remark}

\subsection{Universal Approximation for Fixed Points}

\begin{theorem}[UAP for $\mcF_{\mathrm{FP}}$]\label{thm:uap_fp}
Let $f \in \mcF_{\mathrm{FP}}$ with flow $\varphi: \reals_{\geq 0} \times \mcX \to \mcX$, and let $\hat{\mcF}$ be a hypothesis class with the $C^1$ universal approximation property (Definition~\ref{def:hypothesis_class}).
Assume $f$ is strictly inward-pointing at $\partial \mcX$ (Definition~\ref{def:inward_pointing}).

For all $\varepsilon > 0$ and all $\delta > 0$, there exists $\hat{f} \in \hat{\mcF}$ with flow $\hat{\varphi}$ such that:
\(
\|\varphi - \hat{\varphi}\|_\varepsilon < \delta.
\)
\end{theorem}

\begin{proof}[Proof Sketch]
The proof proceeds in three steps (full proof in App.~\ref{app:proof_fp}):

\textbf{Step 1: Structural stability radius.}
By Theorem~\ref{thm:palis_smale}, there exists $\eta_0 > 0$ such that $\|f - \hat{f}\|_{C^1} < \eta_0$ implies topological equivalence via homeomorphism $h$ with $\|h - \mathrm{id}\|_{C^0} \to 0$ as $\eta_0 \to 0$.

\textbf{Step 2: Basin error control.}
Separatrices $S$ have measure zero (codim-$\geq 1$ saddle stable manifolds), so $\mu(N_\eta(S)) \to 0$ as $\eta \to 0$. Pick $\eta > 0$ with $\mu(N_\eta(S)) < \delta$; by Lemma~\ref{lem:basin_error}, $\mcE_{\mathrm{basin}} \subset N_\eta(S)$ for $\|f-\hat f\|_{C^1}$ small, so $\mu(\mcE_{\mathrm{basin}}) < \delta$.

\textbf{Step 3: Trajectory bounds in valid set.}
Define $\mcV \coloneqq \mcX \setminus N_\eta(S)$, depending on $f, \eta$ alone. By Lemma~\ref{lem:traj_bound}, $T^*(\eta) < \infty$ and the constants $K_1, K$ are determined a priori. For $x_0 \in \mcV$, with $D\varphi_{s,t}$ the flow Jacobian:
\[
\|\varphi(t, x_0) - \hat{\varphi}(t, x_0)\| \leq \int_0^t \|D\varphi_{s,t}\| \cdot \|f - \hat{f}\|_{C^0} \, ds \leq \frac{K\nu}{\lambda},
\]
where $\lambda > 0$ is the slowest contraction rate (smallest $|\mathrm{Re}(\mu)|$ for stable eigenvalues $\mu$ at attractors) and $K \geq 1$ bounds transient growth before trajectories enter the linearization regime.
Exponential contraction near hyperbolic attractors makes the integral converge, giving a bound uniform in $t \geq 0$.

By UAP, choose $\hat{f} \in \hat{\mcF}$ with $\|f - \hat{f}\|_{C^1} < \min(\eta, \lambda\varepsilon/K)$ to achieve $\|\varphi - \hat{\varphi}\|_\varepsilon < \delta$.
\end{proof}

\begin{remark}[Interpretation]
Theorem~\ref{thm:uap_fp} states that multistable systems with multiple stable equilibria can be approximated with arbitrary precision ($\varepsilon$) and reliability ($1 - \delta$) over infinite time.
The key insight is that structural stability confines basin errors to a thin layer near separatrices, and exponential contraction near attractors bounds trajectory error uniformly in time.
\end{remark}

\begin{remark}[Basin error scaling]\label{rem:basin_scaling}
For Morse-Smale systems with codimension-1 separatrices, the basin error volume scales linearly with the $C^1$ perturbation: $\mu(\mcE_{\mathrm{basin}}) \leq C_S \cdot \|f - \hat{f}\|_{C^1}$, where $C_S$ depends on the surface area of the separatrices and the transversality of the flow.
This allows choosing $\eta_{\mathrm{vol}} = \delta / C_S$ to satisfy the reliability constraint.
\end{remark}

\begin{remark}[Quantitative scaling]
The required $C^1$ approximation accuracy $\eta_0$ scales as:
\(
\eta_0 = O\!\left(\min\left\{\delta, {\lambda / \varepsilon}{K}\right\}\right),
\)
where $\lambda$ is the slowest contraction rate at attractors, and $K \geq 1$ bounds transient growth.
\end{remark}

\begin{example}[1D bistable system]\label{ex:bistable}
Consider the double-well system $\dot{x} = f(x) = x - x^3$ on $\mcX = [-2, 2]$ (Figure~\ref{fig:bistable}).
This has stable fixed points at $x_\pm = \pm 1$ and an unstable fixed point (separatrix) at $x_0 = 0$.
The basins are $\boa(x_-) = [-2, 0)$ and $\boa(x_+) = (0, 2]$.
\textbf{Parameters:}
The Jacobian $f'(x) = 1 - 3x^2$ gives $f'(\pm 1) = -2$, so $\lambda = 2$.
The Lipschitz constant on $\mcX$ is $L = \max|f'| = \max|1 - 3x^2| = 11$ (at $x = \pm 2$).

\begin{figure}[t]
\centering
% Panel (a) is boxed, then all three panels are laid out in a single tikzpicture
% so their top edges share one line; each label is tucked inside the top-left.
\sbox0{\begin{tikzpicture}[scale=0.5]
  % Vector field plot
  \begin{scope}
    \draw[->] (-2.3,0) -- (2.5,0) node[above] {$x$};
    \draw[->] (0,-1.2) -- (0,1.9) node[above] {$\dot{x}$};
    % Perturbed vector field (transparent)
    \draw[thick, orange, opacity=0.6, domain=-1.5:1.5, samples=50] plot (\x, {\x - \x*\x*\x - 0.15*cos(2*\x r)});
    % True vector field
    \draw[thick, blue, domain=-1.5:1.5, samples=50] plot (\x, {\x - \x*\x*\x});
    \fill[red] (-1,0) circle (2pt);
    \fill[red] (1,0) circle (2pt);
    \draw[red, fill=white, thick] (0,0) circle (2pt);
    \node[below] at (-1,-0.25) {\scriptsize $x_-$};
    \node[below] at (1.3,0.5) {\scriptsize $x_+$};
    \node[below] at (0,-0.25) {\scriptsize $x_0$};
  \end{scope}
  % Phase line below
  \begin{scope}[yshift=-2.0cm]
    \draw[thick] (-2.2,0) -- (2.2,0);
    \fill[blue!15] (-2.2,-0.12) rectangle (0,0.12);
    \fill[red!15] (0,-0.12) rectangle (2.2,0.12);
    \foreach \x in {-1.5,0.5} {
      \draw[-{Stealth[length=1.5mm]}, thick] (\x,0) -- (\x+0.2,0);
    }
    \foreach \x in {-0.5,1.5} {
      \draw[-{Stealth[length=1.5mm]}, thick] (\x,0) -- (\x-0.2,0);
    }
    \fill[red] (-1,0) circle (2.5pt);
    \fill[red] (1,0) circle (2.5pt);
    \draw[red, fill=white, thick] (0,0) circle (2.5pt);
    \draw[dashed, orange, thick] (0.15,-0.25) -- (0.15,0.25);
  \end{scope}
\end{tikzpicture}}%
\begin{tikzpicture}[inner sep=0]
  \node[anchor=north west] (a) at (0,0) {\makebox[0.19\textwidth]{\usebox0}};
  \node[anchor=north west, xshift=5.5mm] (b) at (a.north east) {\includegraphics[width=0.5\textwidth]{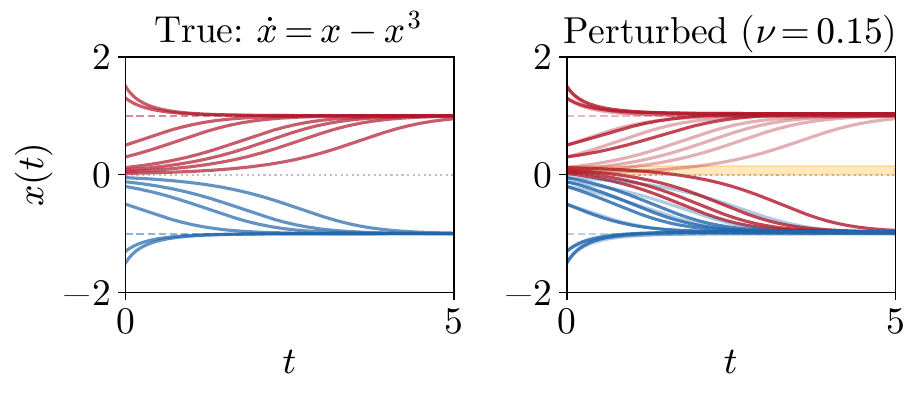}};
  \node[anchor=north west, xshift=5.5mm] (c) at (b.north east) {\includegraphics[width=0.23\textwidth]{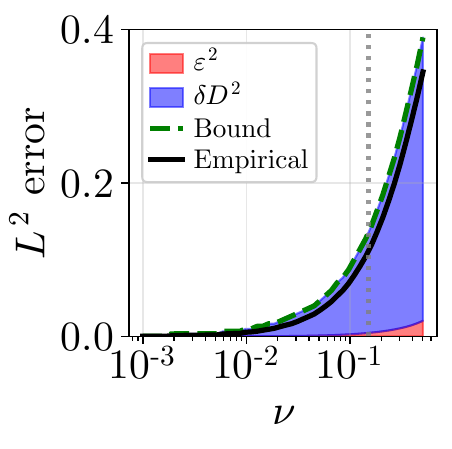}};
  % labels in the top-left corner; (C) nudged left so it clears the y-ticks
  \foreach \n/\L/\dx in {a/A/0pt, b/B/0pt, c/C/-2.2mm}
    \node[anchor=north west, font=\small\bfseries, inner sep=1.5pt,
          fill=white, fill opacity=0.72, text opacity=1, rounded corners=0.5pt]
      at ([xshift=\dx]\n.north west) {(\L)};
\end{tikzpicture}
\caption{The 1D bistable system (Example~\ref{ex:bistable}).
\textbf{(A)} Vector field $f(x) = x - x^3$ with stable $x_\pm = \pm 1$ (filled) and unstable $x_0 = 0$ (open); phase line shows basins with perturbed separatrix (dashed orange).
\textbf{(B)} Trajectories from various initial conditions; orange region marks ICs that switch basins under perturbation $\hat{f} = f - \nu\cos(2x)$.
\textbf{(C)} Time-averaged $L^2$ error vs.\ perturbation size $\nu$: empirical error (black) stays below bound $\varepsilon^2 + \delta D^2$ (dashed, Theorem~\ref{thm:lp_bound}), decomposed into trajectory error (red) and B-type basin error (blue).}
\label{fig:bistable}
% Experimental details:
% - Domain: x0 in [-1.8, 1.8], 1000 initial conditions
% - Simulation time: T=10 (20 time constants, tau=0.5 from f'(±1)=-2)
% - Perturbation: nu in [0.001, 0.5] log-spaced, 50 values
% - Basin diameter D=1.8 for L^p bound
% - Reference nu=0.15 used in trajectory panel (b)
\end{figure}

\textbf{Approximation:}
Let $\hat{f}(x) = x - x^3 + \nu g(x)$ for some bounded perturbation $g$ with $\|g\|_{C^1} \leq 1$.
The perturbed separatrix shifts by $O(\nu)$: if $\hat{x}_0$ satisfies $\hat{f}(\hat{x}_0) = 0$, then $|\hat{x}_0| \leq C\nu$ for small $\nu$.

\textbf{Error bounds:}
For the valid set $\mcV = \mcX \setminus (-\eta, \eta)$ with $\eta > C\nu$: 1. basin error: $\mu(\mcE_{\mathrm{basin}}) \leq 2\eta$ (the strip around $x=0$) and  2. trajectory error: $\sup_t \|\varphi - \hat{\varphi}\| \leq \nu \left( \frac{e^{LT^*} - 1}{L} + \frac{C}{\lambda} \right)$ (cf.\ Lemma~\ref{lem:traj_bound}, with $C \approx 1$ in 1D).
Choosing $\eta = \delta/2$ and $\nu = \min(\eta/C, \lambda\varepsilon/K)$ achieves $\varepsilon$-$\delta$ closeness.
\end{example}

\subsection{Universal Approximation for Limit Cycles}

\begin{theorem}[UAP for $\mcF_{\mathrm{LC}}$]\label{thm:uap_lc}
Let $f \in \mcF_{\mathrm{LC}}$ with hyperbolic limit cycles $\gamma_1, \ldots, \gamma_k$ having periods $T_1, \ldots, T_k > 0$.
Assume $f$ is strictly inward-pointing at $\partial \mcX$.
Let $\hat{\mcF}$ be a hypothesis class with the $C^1$ universal approximation property satisfying:
\begin{itemize}[nosep]
\item For $k = 1$: $\hat{\mcF}$ is closed under scalar multiplication (always satisfied by neural networks).
\item For $k \in \mathbb{N}$: the cycles are well-separated ($\min_{i \neq j} \dist(\gamma_i, \gamma_j) > 0$). The localized period correction is then realized under the $C^1$ UAP alone (Remark~\ref{rem:bump_support}).
\end{itemize}

For all $\varepsilon > 0$ and all $\delta > 0$, there exists $\hat{f} \in \hat{\mcF}$ with flow $\hat{\varphi}$ such that:
\(
\|\varphi - \hat{\varphi}\|_\varepsilon < \delta.
\)
\end{theorem}

\begin{proof}[Proof Sketch]
The key challenge is P-type error: even with $\|f - \hat{f}\|_{C^1}$ small, period mismatch causes unbounded phase drift.
We eliminate this via exact period matching (full proof in App.~\ref{app:proof_lc}).

\textbf{Step 1: Base approximation.}
Choose $\tilde{f} \in \hat{\mcF}$ with $\|f - \tilde{f}\|_{C^1} < \eta$.
By Fenichel's persistence theorem, $\tilde{f}$ has limit cycles $\tilde{\gamma}_i$ with periods $\tilde{T}_i$ satisfying $|\tilde{T}_i - T_i| \leq L_T \eta$.

\textbf{Step 2: Period correction.}
For a single cycle, global scaling $\hat{f} = c^* \tilde{f}$ with $c^* = \tilde{T}/T$ achieves exact period matching: the cycle $\hat{\gamma} = \tilde{\gamma}$ has period exactly $T$.

For multiple cycles with different correction factors, use localized bump functions $\psi_i$ supported in disjoint tubular neighborhoods $N_i$ of radius $r < \frac{1}{2}\min_{i \neq j}\dist(\gamma_i, \gamma_j)$, with $\psi_i = 1$ on $\tilde{\gamma}_i$.
Define $\hat{f} = (1 + \sum_i \alpha_i^* \psi_i) \tilde{f}$ where $\alpha_i^* = (\tilde{T}_i - T_i)/T_i$.
The separation condition ensures the supports are disjoint, so each cycle is corrected independently.

\textbf{Step 3: Infinite-time bound.}
With exact period matching, the transient phase offset accumulated during $[0, T^*]$ is \emph{frozen} for all $t > T^*$.
The asymptotic error is bounded by geometric proximity $d_H(\gamma, \hat{\gamma}) + O(\|f - \hat{f}\|_{C^1})$, which can be made $< \varepsilon$.
\end{proof}

\begin{remark}[Interpretation]
Theorem~\ref{thm:uap_lc} shows that limit cycle systems require a two-stage approximation: first $C^1$-approximate the vector field, then correct the period exactly.
Without period correction, even perfect $C^1$ approximation yields unbounded phase drift as $t \to \infty$ (P-type error).
The bump function construction localizes corrections to avoid interference between cycles.
\end{remark}

\begin{remark}[Localized period correction: realizability]\label{rem:bump_support}
For $k > 1$ limit cycles (or isochronous manifolds in Theorem~\ref{thm:uap_ca}), the period correction is localized to disjoint tubular neighborhoods $N_i$ of the cycles. It suffices that $\hat{\mcF}$ contain localized corrections $\hat{\mathbf{\Phi}}_i$ that are $C^1$-close to the ideal fields (equal to $\tilde{f}$ on $N_i$ and vanishing outside a slightly larger neighborhood): by Lemma~\ref{lem:bump_realizability} (Bump Function Realizability), the $C^1$ UAP provides such $\hat{\mathbf{\Phi}}_i \in \hat{\mcF}$ to any tolerance $\min(\nu, \zeta)$ in $C^1$ norm, giving on-cycle alignment below $\nu$ on $\tilde{\gamma}_i$ and off-cycle leakage below $\zeta$ on $\tilde{\gamma}_j$ for $j \neq i$. Whenever these tolerances satisfy
\(
  \nu + (N-1)\,\zeta < \frac{1}{C_Z}
\) 
with
\(
 C_Z = \max_i \sup_{t \in [0, \tilde{T}_i]} \|Z_i(t)\|
\)
(with $Z_i$ the adjoint, i.e.\ phase-response, solutions), the period Jacobian is strictly diagonally dominant, hence invertible by Lemma~\ref{lem:robust_inv} (Robust Invertibility), and the inverse function theorem yields the exact period correction. Smooth-activation networks meet the condition with sufficient width; compactly supported activations (windowed ReLU, RBF, B-spline) meet it exactly \citep{clevert2016elu,barron2017celu,elfwing2018sigmoid,hasan2023talu,duch1999survey,dubey2022activation,jagtap2023activation,ramachandran2017activation,hayou2019activation}.
\end{remark}

%\begin{remark}[Practical limitation: existence vs.\ learnability]\label{rem:existence_vs_learning}
%The period correction construction is an \textbf{existence proof}, not a training algorithm.
%It requires explicit computation of correction factors $\alpha_i^* = (\tilde{T}_i - T_i)/T_i$ from measured period errors, which cannot be achieved through standard gradient-based training on trajectory data.
%For fixed-point attractors ($\mcF_{\mathrm{FP}}$, $\mcF_{\mathrm{CA}}$ with fixed points), no such fine-tuning is required, and standard training suffices.
%For oscillatory systems ($\mcF_{\mathrm{LC}}$, $\mcF_{\mathrm{CA}}$ with limit cycles), infinite-time guarantees are likely unattainable via gradient-based training: any residual period error causes unbounded phase drift as $t \to \infty$.
%\end{remark}

\subsection{Universal Approximation for Continuous Attractors}

Continuous attractors are not structurally stable: generic perturbations collapse the manifold into discrete points, causing \textbf{D-type error} (Discretization)~\citep{Sagodi2024a}.
We address this via a tiling strategy that accepts bounded D-type error.

\begin{theorem}[UAP for $\mcF_{\mathrm{CA}}$]\label{thm:uap_ca}
Let $f \in \mcF_{\mathrm{CA}}$ with normally hyperbolic continuous attractor $\mcM$ (Definition~\ref{def:nhim}).
Assume $f$ is strictly inward-pointing at $\partial \mcX$.
Let $\hat{\mcF}$ be a hypothesis class with the $C^1$ universal approximation property satisfying: (1) if $\mcM$ consists of fixed points: UAP suffices; (2) if $\mcM$ is foliated by isochronous limit cycles: the localized period correction is realized under the $C^1$ UAP alone (Remark~\ref{rem:bump_support}); no separate condition on $\hat{\mcF}$ is needed.
Then, for all $\varepsilon > 0$ and all $\delta > 0$, there exists $\hat{f} \in \hat{\mcF}$ with flow $\hat{\varphi}$ such that:
\(
\|\varphi - \hat{\varphi}\|_\varepsilon < \delta.
\)
\end{theorem}

\begin{proof}[Proof Sketch]
Continuous attractors are not structurally stable: generic perturbations collapse the continuum into discrete attractors (D-type error).
We use a \textbf{tiling strategy} (full proof in App.~\ref{app:proof_ca}).

\textbf{Manifolds of fixed points:}
Construct an intermediate system $f_{\mathrm{tile}}$ with discrete stable fixed points $\{x_1, \ldots, x_k\} \subset \mcM$ separated by saddles, with spacing $< \varepsilon/2$.
Since $f_{\mathrm{tile}} \in \mcF_{\mathrm{FP}}$, Theorem~\ref{thm:uap_fp} applies.
The maximum error is bounded by the tiling resolution.

\textbf{Isochronous manifolds of limit cycles:}
Tile the continuous attractor with a grid of discrete stable limit cycles $\{\gamma_1, \ldots, \gamma_N\}$ covering $\mcM$ with spacing $< \varepsilon/2$.
Apply Theorem~\ref{thm:uap_lc} with period correction. % to the common period $T$.
The isochrony condition ensures P-type error does not compound D-type error.
\end{proof}

%==============================================================================
\section{Temporal Generalization}\label{sec:temporal}

A key practical implication is that $\varepsilon$-$\delta$ closeness implies bounded $L^p$ error over infinite time, connecting topological guarantees to training metrics.

\begin{definition}[Time-averaged $L^p$ error]\label{def:lp_error}
For flows $\varphi, \hat{\varphi}$ on bounded domain $\mcX$:
\[
\mathcal{E}_{p,\infty}(\varphi, \hat{\varphi}) \coloneqq \frac{1}{\vol(\mcX)} \int_{\mcX} \left( \limsup_{T \to \infty} \frac{1}{T} \int_0^T \|\varphi(t, x_0) - \hat{\varphi}(t, x_0)\|^p \, dt \right) dx_0.
\]
\end{definition}

\begin{theorem}[$L^p$ error bound]\label{thm:lp_bound}
Let $\varphi, \hat{\varphi}: \reals_{\geq 0} \times \mcX \to \mcX$ be continuous flows on bounded domain $\mcX \subset \reals^n$ with diameter $D \coloneqq \sup_{x,y \in \mcX} \|x - y\| < \infty$.
If $\hat{\varphi}$ is $\varepsilon$-$\delta$ close to $\varphi$ (Definition~\ref{def:eps_delta_close}), then for all $p \geq 1$:
\(
\mathcal{E}_{p,\infty}(\varphi, \hat{\varphi}) \leq \varepsilon^p + \delta \cdot D^p.
\)
\end{theorem}

\begin{proof}
Partition $\mcX$ into the ``good'' set $G = \{x_0 : \sup_{t \geq 0} \|\varphi(t, x_0) - \hat{\varphi}(t, x_0)\| \leq \varepsilon\}$ and ``bad'' set $B = \mcX \setminus G$.
By $\varepsilon$-$\delta$ closeness, $\vol(B) < \delta \cdot \vol(\mcX)$.
On $G$: error bounded by $\varepsilon$ for all time, so time-average $\leq \varepsilon^p$.
On $B$: error bounded by diameter $D$, so time-average $\leq D^p$.
Summing: $\mathcal{E}_{p,\infty} \leq \frac{\vol(G)}{\vol(\mcX)} \varepsilon^p + \frac{\vol(B)}{\vol(\mcX)} D^p \leq \varepsilon^p + \delta \cdot D^p$.
\end{proof}

\begin{corollary}[MSE bound]\label{cor:mse}
Under the assumptions of Theorem~\ref{thm:lp_bound}:
\(
\MSE_\infty(\varphi, \hat{\varphi}) \leq \varepsilon^2 + \delta \cdot D^2.
\)
\end{corollary}

\begin{remark}
Theorem~\ref{thm:lp_bound} establishes that topological guarantees (via $\varepsilon$-$\delta$ closeness) imply bounds on practical metrics.
The converse is false: low MSE does not guarantee topological correctness, as a model can achieve low average error while failing to capture asymptotic stability or periodicity.
\end{remark}

%==============================================================================
\section{Numerical experiments}\label{sec:experiments}

We validate the three guarantees and the basin / phase / discretization taxonomy on trained Neural ODEs and on real task-trained networks. App.~\ref{app:experiments} gives full details, additional architectures, and the budget and width scaling.\footnote{Code reproducing all experiments is available at \url{https://github.com/Asagodi/uap-infinite-time-code}.}

\paragraph{$\varepsilon$-$\delta$ closeness and the failure modes.}
We illustrate the metric on the bistable Duffing system (Figure~\ref{fig:exp_bpd}A,B). A width-$64$ Neural ODE reproduces both basins with $99.96\%$ agreement; the trajectory sup-error is of order $10^{-3}$ across both basins and rises to the basin diameter only on a thin collar around the separatrix (B-type), so $\delta(\varepsilon)$ stays small ($\delta(0.1) = 0.007$). App.~\ref{app:experiments} illustrates the other two modes: on the Van der Pol oscillator a residual period error of $1.25 \times 10^{-3}$ produces phase drift that reaches the cycle diameter near $t \approx 6 \times 10^{3}$, invisible to a $45$-cycle validation window (P-type) and removed by the exact period matching of Theorem~\ref{thm:uap_lc}; and on a ring attractor the continuum discretizes into $10$ alternating sinks and saddles while the ring itself persists (D-type).

\begin{figure}[t]
\centering
\begin{tikzpicture}[inner sep=0]
  % widths solved so all four panels render at equal height (A's square+colorbar
  % shape needs a wider inclusion; its base font is dropped to keep on-page size)
  \node[anchor=north west] (a) at (0,0) {\includegraphics[width=0.272\textwidth]{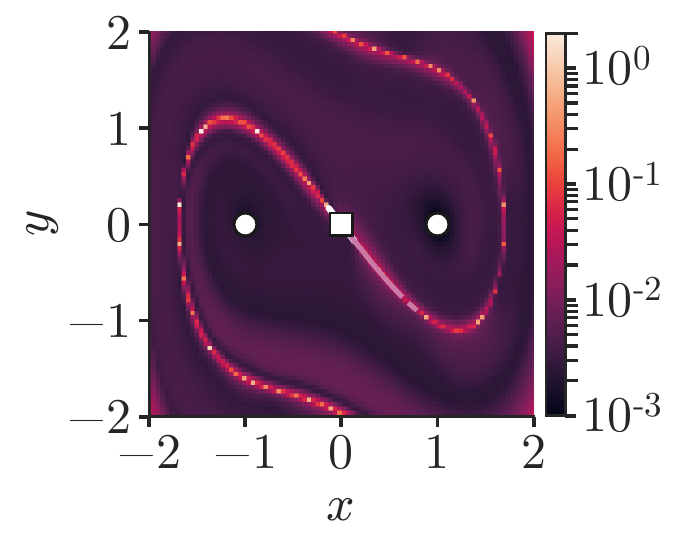}};
  \node[anchor=north west, xshift=1mm] (b) at (a.north east) {\includegraphics[width=0.239\textwidth]{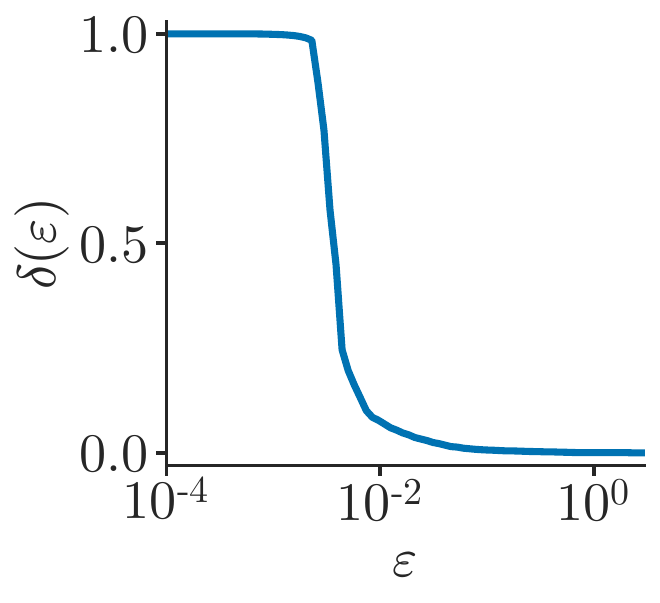}};
  \node[anchor=north west, xshift=1mm] (c) at (b.north east) {\includegraphics[width=0.229\textwidth]{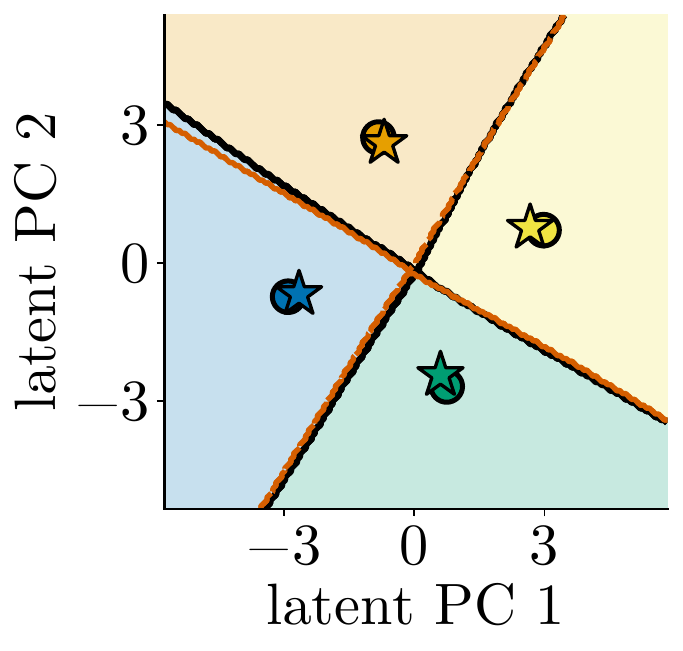}};
  \node[anchor=north west, xshift=1mm] (d) at (c.north east) {\includegraphics[width=0.238\textwidth]{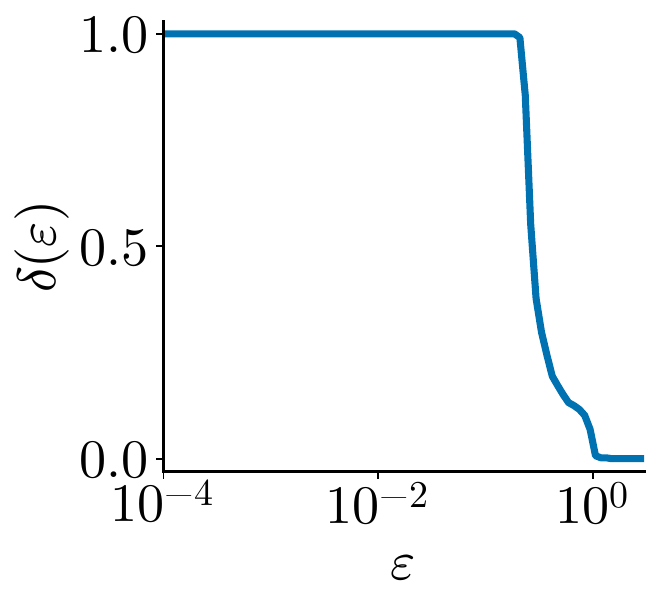}};
  % labels in the top-left corner; (B) and (D) nudged left to clear the y-ticks
  \foreach \n/\L/\dx in {a/A/0pt, b/B/-1.5mm, c/C/0pt, d/D/-1.5mm}
    \node[anchor=north west, font=\small\bfseries, inner sep=1.5pt,
          fill=white, fill opacity=0.72, text opacity=1, rounded corners=0.5pt]
      at ([xshift=\dx]\n.north west) {(\L)};
\end{tikzpicture}
\caption{$\varepsilon$-$\delta$ closeness on a bistable system and on a real trained network.
\textbf{(A)} Sup-error map for a Neural ODE fit to the damped Duffing system: the error concentrates on a thin collar around the separatrix (B-type); basin agreement $99.96\%$.
\textbf{(B)} Its $\delta(\varepsilon)$ curve stays small ($\delta(0.1) = 0.007$), dominated by the separatrix collar.
\textbf{(C)} Real versus fitted separatrices and fixed points on a 2D slice through four memory states: the trained network's basin boundaries (solid) and the Neural ODE's (dashed) coincide over $98\%$ of the slice, and the network's fixed points (circles) sit close to the Neural ODE's (stars), so the approximation preserves the basin partition ($\mathcal{F}_{\mathrm{FP}}$), as the true and learned separatrices coincide in (A).
\textbf{(D)} That Neural ODE's $\delta(\varepsilon)$ curve against the trained network, on the network's own readout ($99.2\%$ per-bit accuracy).}
\label{fig:exp_bpd}
\end{figure}

\paragraph{Architecture generality and real networks.}
On a three-dimensional target with eight fixed points, a smooth-activation MLP, a $C^1$ RNN field, and a radial-basis-function field each recover all eight attractors with $\hat\delta(0.3) \le 6 \times 10^{-3}$, as Theorem~\ref{thm:uap_fp} predicts for any $C^1$-dense class. Two real task-trained networks fall in the target classes: a gated recurrent network trained on the $3$-bit flip-flop task~\citep{sussillo2013blackbox, versteeg2025computation} has stable fixed points at all eight memory states ($\mathcal{F}_{\mathrm{FP}}$), and a Neural ODE inherits all eight (autonomous basin agreement $99.8\%$) and reproduces the readout at $99.2\%$ per-bit accuracy (Figure~\ref{fig:exp_bpd}C,D); a rectified-tanh network trained to integrate angular velocity~\citep{Sagodi2024a} holds heading on a ring that the finite network tiles into $12$ hyperbolic sinks alternating with $12$ saddles, a structurally stable approximation of the continuum whose fixed points a Neural ODE inherits ($\mathcal{F}_{\mathrm{CA}}$, D-type).

%==============================================================================
\section{Discussion}\label{sec:discussion}

\paragraph{Summary.}
We established universal approximation theorems for multistable dynamical systems over infinite time horizons.
Our results apply to any Neural ODE architecture with the universal approximation property and provide explicit $\varepsilon$-$\delta$ guarantees linking vector field approximation to trajectory bounds.
For systems with stable fixed points ($\mcF_\mathrm{FP}$), trajectory errors decay to zero on the good set; for limit cycles ($\mcF_\mathrm{LC}$), errors stabilize at a small residual set by the cycle's geometric distance and phase mismatch.
Finally, for continuous attractors ($\mcF_\mathrm{CA}$), we show that arbitrary precision is achievable via dense tiling of discrete attractors, overcoming their structural instability.
Furthermore, we established a rigorous link between our topological $\varepsilon$-$\delta$ condition and the squared error loss used in training.

For finite-time approximation, almost any architecture suffices, including linear state-space models.
However, their temporal generalization behavior can vary wildly, rendering finite-time guarantees inadequate for accurate, reliable long-time modeling of a system~\citep{sussillo2013blackbox,koppe2019identifying}.
If we require the causal relations in a mechanism~\citep{kaplan2011explanatory,cao2021explanatory1,cao2021explanatory2} to be captured by a vector field, then infinite-time horizon results provide for mechanistic interpretability guarantees at the state space level.
This distinguishes Neural ODEs from finite-dimensional RNNs~\citep{niu2019recurrent,habiba2020neural}, which require embedding-type approximations that transform the state space and introduce extra state variables \citep{versteeg2025computation}; Neural ODEs instead enable direct approximation of the vector field in the original coordinates. % in the original coordinates. 
This also makes XFADS~\citep{Dowling2024b} more expressive than LFADS~\citep{sussillo2016lfads,pandarinath2018inferring}, enabling topological reconstruction of multistable attractors.

Beyond existence proofs, our framework provides a practical taxonomy for analyzing model failure in computational neuroscience. 
By classifying errors into basin-type, period-type, and trajectory-type, researchers can diagnose specifically \emph{why} a model fails to match a true biological implementation.
For example, a Basin-type error corresponds directly to a failure in reliable memory storage, whereas Period-type errors reflect the fragility of oscillatory binding.
Furthermore, this framework introduces a novel metric for the complexity of learning a computation, quantified by geometric properties such as the total size of separatrices required to define decision boundaries. 
This enables a more comprehensive and unified approach to comparing models of neural computation, replacing simple goodness-of-fit metrics with geometric evaluations of dynamic fidelity.

\paragraph{Limitations.}
\textit{Learnability gap.}
Our theorems guarantee existence of approximations, not their learnability.
The loss landscape for infinite-horizon systems is prone to exploding gradients~\citep{ribeiro2020beyond}, and practical training requires specialized algorithms such as multiple shooting, teacher forcing~\citep{hess2023teacherforcing} or homotopy-based training~\citep{ko2023homotopy}, along with strong inductive biases.

\textit{Fragility of period matching.}
Exact isochrony is not structurally stable: in the space of vector fields, those with a prescribed period $T$ form a measure-zero set.
%Perturbation to network weights generically change the period, reintroducing phase drift.
Thus, while Neural ODEs are universal approximators for limit cycles, they are not robust implementations for tasks requiring infinite-time phase locking.
To mitigate this, one can parametrize the period explicitly, separating attractor geometry from the flow speed on it~\citep{Sagodi2025b}. % to decouple the learning of the invariant set from the dynamics on the manifold.

\textit{Finite-time sensitivity.}
Our guarantee is asymptotic, constraining the basin and phase structure as $t \to \infty$; the finite-time error is instead set by how strongly the flow amplifies perturbations before settling.
Even in Morse-Smale systems this amplification is largest on the B-type collar around separatrices, where finite-time Lyapunov exponents are large~\citep{shadden2005definition}, and non-normal linearizations can add substantial transient growth, quantified by pseudospectra~\citep{trefethen2005spectra}, without changing the asymptotic guarantee.
Near saddles, the same sensitivity lets small inputs steer trajectories between basins~\citep{Sagodi2026committor}.

\textit{Structurally unstable and chaotic systems.}
Our uniform convergence metric is ill-suited for chaotic systems due to sensitive dependence on initial conditions~\citep{hirsch1995computing}.
We term this \textbf{C-type error}: trajectories diverge exponentially even with accurate local approximation.
Validating chaotic models therefore requires invariant set reconstruction: matching Lyapunov spectra~\citep{hart2024attractor}, minimizing attractor Hausdorff distance~\citep{hess2023teacherforcing}, or shadowing~\citep{hoppensteadt2013analysis}.
\textit{Quasiperiodic flows on tori} similarly suffer frequency locking under generic perturbations~\citep{guckenheimer1983nonlinear,guckenheimer1988structurally,Park2023a}, while \textit{center manifolds} and \textit{homoclinic cycles} are determined by higher-order terms or precise manifold intersections.
%Infinite-time guarantees for these classes likely require architectural constraints (e.g., symmetry) rather than general-purpose approximation.

\textit{Stochastic dynamics.}
Real-world systems are invariably subject to noise. 
Recent work has begun to establish universal approximation properties for stochastic reservoirs and filters \citep{gonon2019reservoir, bishop2023recurrent, ehlers2025stochastic}, approximation guarantees for the distance between path measures \citep{backhoff2022adapted, chen2022universal} and flow matching techniques \citep{lipman2024flow, mishne2024elucidating}.

\textit{Non-autonomous systems.} 
Extending infinite-time guarantees to input-driven dynamics remains open and requires pullback attractors rather than static invariant sets~\citep{kloeden2011nonautonomous}.
%While infinite-horizon proofs are pending, it is established that non-autonomous nonlinear systems can be approximated on compact intervals \citep[Lemma 5.3.2]{garces2012strategies, smale1974differential}, broadening RNN applicability to signal-influenced systems. 
A continuous-time recurrent network can approximately realize any trajectory on a finite interval, and on an infinite one for periodic input~\citep{nakamura2009approximation}.
Finite-time approximation of input-driven systems has been established for Neural ODEs \citep{li2022deep, ko2023homotopy, zakwan2023universal} and neural flows~\citep{bilos2021neural}.

%==============================================================================
\newpage   % fill goes to the page bottom, not into the Discussion's paragraph gaps (\flushbottom)
\begin{ack}
This work was supported by NIH RF1-DA056404 and the Portuguese Recovery and Resilience Plan (PPR), through project number 62, Center for Responsible AI, and the Portuguese national funds, through FCT (Funda\c{c}\~{a}o para a Ci\^{e}ncia e a Tecnologia) in the context of the project UIDB/04443/2020.
The authors declare no competing interests.
\end{ack}

\newpage
\bibliographystyle{plainnat}
\bibliography{../all_ref,../catniplab}

%==============================================================================
\newpage
\appendix

\paragraph{Classical results versus our contributions.}
We separate what we cite from what we prove. The classical results we use without reproving are Palis-Smale structural stability, the stable manifold theorem, Fenichel persistence of normally hyperbolic manifolds, Kupka-Smale density, Malkin's first variation of the period, Gr\"onwall's inequality, and the $C^1$ universal approximation property. These supply the combinatorics of the invariant sets, but none of them controls the matched-time error $\sup_{t}\|\varphi - \hat\varphi\|$, because topological equivalence matches orbits only up to a time reparametrization; the gap is sharpest for a limit cycle, where a nearby orbit of period $\hat T \neq T$ lets $\sup_{t}\|\varphi - \hat\varphi\|$ grow to the cycle diameter. Our contributions are: (i) simultaneous multi-cycle period matching, via a period Jacobian that is strictly diagonally dominant and hence invertible (Theorem~\ref{thm:uap_lc}), which recovers the time compatibility structural stability discards; (ii) realizability of the localized corrections from the $C^1$ universal approximation property alone (Lemma~\ref{lem:bump_realizability}), so Definition~\ref{def:hypothesis_class} is the only architectural assumption needed; (iii) a uniform transient time together with a Lyapunov-trapping construction that make an infinite-horizon Gr\"onwall bound possible at all (Lemma~\ref{lem:traj_bound}); (iv) the explicit tiling of the structurally unstable continuous-attractor case (Theorem~\ref{thm:uap_ca}); (v) the isochrony obstruction and a phase-coherence bound across tiles; and (vi) the bridge from $\varepsilon$-$\delta$ closeness to training metrics (Theorem~\ref{thm:lp_bound}).

\section{Definitions and Background}\label{app:definitions}

\subsection{Norms and Function Spaces}

\begin{definition}[$C^1$ norm]\label{def:c1_norm}
For $f: \mcX \to \reals^n$ continuously differentiable:
\[
\|f\|_{C^1} \coloneqq \sup_{x \in \mcX} \|f(x)\| + \sup_{x \in \mcX} \|Df(x)\|_{\mathrm{op}}.
\]
\end{definition}

\begin{definition}[Hausdorff distance]\label{def:hausdorff}
For compact sets $A, B \subset \mcX$:
\[
d_H(A, B) \coloneqq \max\left\{ \sup_{a \in A} \inf_{b \in B} \|a - b\|, \sup_{b \in B} \inf_{a \in A} \|a - b\| \right\}.
\]
\end{definition}

\subsection{Dynamical Systems}

\begin{definition}[$\omega$-limit set]\label{def:omega_limit}
For a trajectory $\varphi(t, x_0)$, the \textbf{$\omega$-limit set} is:
\[
\omega(x_0) \coloneqq \bigcap_{T > 0} \overline{\{\varphi(t, x_0) : t \geq T\}}.
\]
Equivalently, $y \in \omega(x_0)$ if and only if there exists a sequence $t_n \to \infty$ with $\varphi(t_n, x_0) \to y$.
\end{definition}

\begin{definition}[Attractor]\label{def:attractor}
A compact invariant set $\mathcal{A} \subset \mcX$ is an \textbf{attractor} if it has an open neighborhood $U \supset \mathcal{A}$ such that $\varphi(t, U) \subset U$ for all $t > 0$ and $\bigcap_{t > 0} \varphi(t, U) = \mathcal{A}$.
\end{definition}

\begin{definition}[Basin of attraction]\label{def:basin}
For attractor $\mathcal{A}$:
\[
\boa(\mathcal{A}) \coloneqq \{x_0 \in \mcX : \lim_{t \to \infty} \dist(\varphi(t, x_0), \mathcal{A}) = 0\}.
\]
\end{definition}

\begin{definition}[Separatrix]\label{def:separatrix}
A \textbf{separatrix} is the boundary between two basins of attraction.
For Morse-Smale systems, separatrices are (unions of) stable manifolds of saddle-type equilibria or periodic orbits.
\end{definition}

\begin{definition}[Non-wandering set]\label{def:non_wandering}
A point $x \in \mcX$ is \textbf{non-wandering} if for every neighborhood $U$ of $x$ and every $T > 0$, there exists $t > T$ such that $\varphi(t, U) \cap U \neq \emptyset$.
The \textbf{non-wandering set} $\Omega(f)$ is the set of all non-wandering points.
Fixed points and periodic orbits are always non-wandering.
\end{definition}

\begin{definition}[Transversal intersection]\label{def:transversal}
Two submanifolds $M, N \subset \mcX$ intersect \textbf{transversally} at $p \in M \cap N$ if $T_p M + T_p N = T_p \mcX$ (the tangent spaces span the ambient space).
The intersection is transversal if this holds at every point of $M \cap N$.
\end{definition}

\begin{definition}[Tubular neighborhood]\label{def:tubular}
For a submanifold $M \subset \mcX$, a \textbf{tubular neighborhood} of radius $r > 0$ is the set $N_r(M) = \{x \in \mcX : \dist(x, M) < r\}$.
\end{definition}

\begin{definition}[Poincar\'e return map]\label{def:poincare_map}
Let $\gamma$ be a periodic orbit and $\Sigma$ a local cross-section (codi\-men\-sion-1 submanifold transverse to the flow) intersecting $\gamma$ at a point $p$.
The \textbf{Poincar\'e return map} $P: U \to \Sigma$ is defined on a neighborhood $U \subset \Sigma$ of $p$ by $P(x) = \varphi(\tau(x), x)$, where $\tau(x) > 0$ is the first return time to $\Sigma$.
The periodic orbit $\gamma$ is hyperbolic if and only if $DP(p)$ has no eigenvalues of modulus 1.
\end{definition}

\begin{definition}[Topological equivalence]\label{def:top_equiv}
Flows $\varphi$ and $\psi$ are \textbf{topologically equivalent} if there exists a homeomorphism $h: \mcX \to \mcX$ mapping orbits of $\varphi$ to orbits of $\psi$ preserving orientation (i.e., the direction of time along orbits).
\end{definition}

Structural stability is a fundamental concept in the study of dynamical systems \citep{peixoto1959ss, mane1987ss, hu1994ss, hayashi1997invariant, robbin1971ss, robinson1974ss, palis1970ss}. 

\begin{definition}[Structural stability]\label{def:structural_stability}
%A vector field $f$ is \textbf{structurally stable} if there exists $\rho > 0$ such that $\|f - \tilde{f}\|_{C^1} < \rho$ implies the flows are topologically equivalent (Definition~\ref{def:top_equiv}).
Let $G$ be an open domain in $\mathbb{R}^n$ with compact closure and smooth $(n-1)$-dimensional boundary. 
Consider the space $X^1(G)$ consisting of restrictions to $G$ of $C^1$ vector fields on $\mathbb{R}^n$ that are transversal to the boundary of $G$ and are inward oriented.
This space is endowed with the $C^1$ metric.
A vector field $F \in X^1(G)$ is \textit{weakly structurally stable} if for any sufficiently small perturbation $F_1$, the corresponding flows are \textit{topologically equivalent} (Definition~\ref{def:top_equiv}) on $G$: there exists a homeomorphism $h: G \to G$ which transforms the oriented trajectories of $F$ into the oriented trajectories of $F_1$.
If, moreover, for any $\varepsilon > 0$ the homeomorphism $h$ may be chosen to be $C^0$ $\varepsilon$-close to the identity map when $F_1$ belongs to a suitable neighborhood of $F$ depending on $\varepsilon$, then $F$ is called (strongly) \textit{structurally stable}.
\end{definition}

\begin{remark}
Structural stability ensures that the qualitative features of the dynamics (number and type of fixed points, periodic orbits, connections between them) persist under small perturbations. 
This is essential for robust modeling, as approximation errors act as perturbations to the true system.
\end{remark}

\subsection{Universal Approximation}

\begin{theorem}[$C^r$ UAP for vector fields~{\citep{hornik1991approximation}}]\label{thm:uap_vf}
Let $\mcX \subset \reals^n$ be compact.
For any $f \in C^r(\mcX, \reals^n)$ with $r \geq 0$ and any $\varepsilon > 0$, there exists a feedforward neural network $\hat{f}$ with smooth activation (e.g., sigmoid, tanh) such that $\|f - \hat{f}\|_{C^r} < \varepsilon$.
\end{theorem}

\begin{remark}
The $C^r$ approximation follows from the $C^0$ UAP combined with smoothness of neural network outputs: smooth activations produce smooth outputs, and derivatives can be approximated by approximating the function~\citep{hornik1991approximation}.
For our theorems, $r = 1$ suffices.
\end{remark}

\begin{lemma}[Levy-Desplanques]\label{lem:levy_desplanques}
A square matrix $M \in \reals^{k \times k}$ is \textbf{strictly diagonally dominant} if $|M_{ii}| > \sum_{j \neq i} |M_{ij}|$ for all $i$.
Every strictly diagonally dominant matrix is invertible.
\end{lemma}

\begin{proof}
Suppose $Mx = 0$ for some $x \neq 0$.
Let $i = \arg\max_j |x_j|$, so $|x_i| > 0$.
Row $i$ gives $M_{ii} x_i = -\sum_{j \neq i} M_{ij} x_j$, hence $|M_{ii}| |x_i| \leq \sum_{j \neq i} |M_{ij}| |x_j| \leq |x_i| \sum_{j \neq i} |M_{ij}|$.
This contradicts diagonal dominance.
\end{proof}

%==============================================================================
\section{Topological structure of the space of dynamical systems}\label{sec:approximation_topology}
In this section, we rigorously formalize the notion of approximation used in our main result. 
We demonstrate that our $(\varepsilon, \delta)$-closeness condition (Def.~\ref{def:eps_delta_close}) generates a valid topology on the space of dynamical systems, specifically the topology of \textit{convergence in measure} regarding the supremum norm of trajectories.

Throughout this appendix, we work with three objects.
The state space $\mcX \subset \reals^n$ is the bounded open domain fixed in Section~\ref{sec:prelim}.
The space of dynamical systems under consideration is denoted $\mathfrak{S} \subseteq \mathfrak{X}^1(\mcX)$; concretely, $\mathfrak{S}$ may be taken as any of the target classes $\mcF_{\mathrm{FP}}, \mcF_{\mathrm{LC}}, \mcF_{\mathrm{CA}}$ of Section~\ref{sec:main} together with the hypothesis class $\hat{\mcF}$, or any larger ambient family.
The measure $\mu$ is the normalized Lebesgue measure on $\mcX$, i.e., $\mu(A) \coloneqq \vol(A)/\vol(\mcX)$ for every Borel set $A \subseteq \mcX$, so that $\mu(\mcX) = 1$ and $\mu$ agrees with the integrand of Eq.~\eqref{eq:eps_volume_error}.
We define the trajectory difference between two systems $f, g \in \mathfrak{S}$ at an initial condition $x_0 \in \mcX$ as:
\begin{equation}
    \Delta(f, g, x_0) \coloneqq \sup_{t \ge 0} \|\varphi_f(t,x_0) - \varphi_g(t,x_0)\|.
\end{equation}

For any reference system $f \in \mathfrak{S}$ and constants $\varepsilon > 0, \delta > 0$, we define the neighborhood basis element $B(f, \varepsilon, \delta)$ as the set of systems that deviate by at most $\varepsilon$ on all but a small volume of initial conditions:
\begin{equation}
    B(f, \varepsilon, \delta) \coloneqq \left\{ g \in \mathfrak{S} \;\middle|\; \mu\left(\{ x_0 \in \mcX : \Delta(f, g, x_0) \ge \varepsilon \}\right) < \delta \right\}.
\end{equation}

%Let $\mathfrak{S}$ be the space of considered dynamical systems on a compact domain $\mcX$ with normalized Lebesgue measure $\mu$.
%For any reference system $f \in \mathfrak{S}$ and constants $\varepsilon > 0, \delta > 0$, we define the neighborhood basis element:
%\[ B(f, \varepsilon, \delta) \coloneqq \left\{ g \in \mathfrak{S} \;\middle|\; \exists E \subset \mcX, \mu(E) < \delta \text{ s.t. } \forall x \notin E, \sup_{t \ge 0} \|\varphi_f(t,x) - \varphi_g(t,x)\| < \varepsilon \right\}. \]

\begin{proposition}%[Basis Properties]
The collection $\mathcal{B} = \{ B(f, \varepsilon, \delta) \}$ satisfies the conditions to be a basis for a topology on $\mathfrak{S}$.
\end{proposition}

\begin{proof}
We verify the two fundamental properties of a topological basis.
First, we verify the covering property. 
For any $f \in \mathfrak{S}$, it holds that $f \in B(f, \varepsilon, \delta)$ for any $\varepsilon, \delta > 0$. 
Since $\Delta(f, f, x_0) = 0$ for all $x_0 \in \mcX$, the set where the error exceeds $\varepsilon$ is empty, and $\mu(\emptyset) = 0 < \delta$. 
Thus, $\mathfrak{S} = \bigcup_{B \in \mathcal{B}} B$.

Next, we verify the intersection property. Let $h \in B(f, \varepsilon_1, \delta_1) \cap B(g, \varepsilon_2, \delta_2)$. 
We show there exists a neighborhood $B(h, \alpha, \beta)$ such that $B(h, \alpha, \beta) \subseteq B(f, \varepsilon_1, \delta_1)$. 
The inclusion for $g$ follows by symmetry, and taking the intersection of the two resulting neighborhoods yields the required basis element.

Since $h \in B(f, \varepsilon_1, \delta_1)$, let $E_h = \{ x_0 \in \mcX : \Delta(f, h, x_0) \ge \varepsilon_1 \}$. By definition, $\mu(E_h) < \delta_1$. We utilize the continuity of measure to establish a safety margin. 
Consider the sets $A_n = \{ x_0 \in \mcX : \Delta(f, h, x_0) \ge \varepsilon_1 - 1/n \}$. Note that $\bigcap_{n=1}^\infty A_n = E_h$. Since $\mu(E_h) < \delta_1$, there exists sufficiently large $N$ such that $\mu(A_N) < \delta_1$. We define the slack parameters $\alpha \coloneqq 1/N$ and $\beta \coloneqq \delta_1 - \mu(A_N)$.

Consider any $k \in B(h, \alpha, \beta)$. Let $E_k = \{ x_0 \in \mcX : \Delta(h, k, x_0) \ge \alpha \}$ be the exception set for $k$, noting that $\mu(E_k) < \beta$. 
We apply the triangle inequality: $\Delta(f, k, x_0) \le \Delta(f, h, x_0) + \Delta(h, k, x_0)$. For any point $x_0 \notin (A_N \cup E_k)$, we have
\[
    \Delta(f, k, x_0) < \left(\varepsilon_1 - \frac{1}{N}\right) + \frac{1}{N} = \varepsilon_1.
\]
Thus, the new exception set for $k$ relative to $f$ is contained in $A_N \cup E_k$. By the sub-additivity of the measure,
\[
    \mu( \{ x_0 \in \mcX : \Delta(f, k, x_0) \ge \varepsilon_1 \} ) \le \mu(A_N) + \mu(E_k) < \mu(A_N) + (\delta_1 - \mu(A_N)) = \delta_1.
\]
Therefore, $k \in B(f, \varepsilon_1, \delta_1)$, which completes the proof.
\end{proof}

\begin{remark}[Connection to the Ky Fan metric]\label{rem:ky_fan}
Our definition of $\varepsilon$-$\delta$ closeness (Eq.~\ref{eq:eps_volume_error}) is topologically equivalent to the topology induced by the \textbf{Ky Fan metric} \citep{dudley2002real} on the space of trajectories. 
Specifically, the distance between two flows $\varphi$ and $\hat{\varphi}$ can be defined as:
\[ 
d_{KF}(\varphi, \hat{\varphi}) = \inf \big\{ \varepsilon > 0 \;\big|\; \mu(\{x \in \mcX : \sup_{t \ge 0} \|\varphi_t(x) - \hat{\varphi}_t(x)\| \ge \varepsilon \}) < \varepsilon \big\}. 
\]
Our guarantee essentially states that $d_{KF}(\varphi, \hat{\varphi})$ can be made arbitrarily small.
\end{remark}

\begin{remark}[Relation to standard error measures]\label{rem:metric_relations}
$\varepsilon$-$\delta$ closeness takes the supremum of the trajectory error over time and controls it in measure over initial conditions.
Relaxing either choice gives a standard measure, and both relaxations are strict.
The uniform error $\sup_{x_0}\sup_{t \geq 0}\|\varphi_t(x_0) - \hat\varphi_t(x_0)\|$ is stronger and cannot be made small in a multistable system: a $C^1$-small perturbation shifts the separatrix, so a collar of initial conditions converges to the other attractor and has error of order $D$ whatever the fit.
$\varepsilon$-$\delta$ closeness concedes exactly this collar, as the set of measure below $\delta$.
The time-averaged $L^p$ error of Theorem~\ref{thm:lp_bound} is weaker: a transient excursion that recontracts has supremum error of order one while its time average tends to zero, so the average cannot constrain long-run behaviour.
Topological equivalence is incomparable with all three, since it matches orbits only up to a time reparametrization.
The damped Duffing fit of Section~\ref{sec:experiments} shows the first gap: its uniform error is about $2.2$, of the order of the distance between the two attractors, while $\delta(0.3) = 0.003$ and basin agreement is $99.96\%$.
\end{remark}

%==============================================================================
\section{Proof of Theorem~\ref{thm:uap_fp}}\label{app:proof_fp}

\begin{lemma}[Trajectory bounds inside basins]\label{lem:traj_bound}
Let $\mathcal{A}$ be a hyperbolic attractor (stable fixed point) with spectral gap $\lambda > 0$ (slowest eigenvalue decay rate, i.e., $\lambda = \min_i |\mathrm{Re}(\mu_i)|$ for eigenvalues $\mu_i$ of $Df(\mathcal{A})$).
Let $r > 0$ be a radius such that linearization is valid in $B_r(\mathcal{A})$.
Let $\mcV \subset \boa(\mathcal{A})$ with $\inf_{x \in \mcV} \dist(x, \partial \boa(\mathcal{A})) \geq \eta > 0$.

There exists $\nu_0 = \nu_0(\eta, \lambda, L, r) > 0$ such that for all $\nu < \nu_0$:
if $\|f - \hat{f}\|_{C^1} < \nu$, then $\hat{f}$ has a hyperbolic attractor $\hat{\mathcal{A}}$ with $\mcV \subset \boa(\hat{\mathcal{A}})$, and for all $x_0 \in \mcV$:
\[
\sup_{t \geq 0} \|\varphi(t, x_0) - \hat{\varphi}(t, x_0)\| \leq \frac{K\nu}{\lambda},
\]
where $K = K(\eta, f) = \lambda C K_1 + C$ with $K_1 = (e^{LT^*} - 1)/L$, $T^* = T^*(\eta)$ is the uniform transient time, $L = \Lip(f)$, and $C$ bounds the flow Jacobian (fundamental matrix) in the linearization regime.
\end{lemma}

\begin{proof}
The proof proceeds in two stages.

\textbf{Stage 1: Transient phase $[0, T^*]$.}
Since $\mcV$ is bounded away from the separatrix by $\eta > 0$, trajectories from $\mcV$ have uniform convergence time.
To see this, note that $\bar{\mcV}$ (the closure of $\mcV$) is compact, and the convergence time function $\tau(x) = \inf\{t : \dist(\varphi(t,x), \mathcal{A}) < r/2\}$ is upper semicontinuous on the basin of attraction.
Since $\bar{\mcV} \subset \boa(\mathcal{A})$ is compact and bounded away from $\partial \boa(\mathcal{A})$, the supremum $T^* = \sup_{x \in \bar{\mcV}} \tau(x)$ is finite~\citep[Thm.~2.1]{chicone2006ode}.

For $t \in [0, T^*]$, by Gr\"onwall's inequality:
\[
\|\varphi(t, x_0) - \hat{\varphi}(t, x_0)\| \leq \nu \cdot \frac{e^{Lt} - 1}{L} \leq K_1 \nu,
\]
where $L = \Lip(f)$ and $K_1 = (e^{LT^*} - 1)/L$.
Choosing $\nu_0 \leq r/(2K_1)$ ensures the perturbed trajectory $\hat{\varphi}(T^*, x_0)$ is within distance $r$ of $\hat{\mathcal{A}}$ (since $\|\mathcal{A} - \hat{\mathcal{A}}\| \leq C_1 \nu$ for some $C_1 > 0$ by structural stability).

\textbf{Stage 2: Asymptotic phase $[T^*, \infty)$.}
Both trajectories start within the linearization neighborhood at time $T^*$.
For $x \in B_r(\mathcal{A})$, the exponential contraction satisfies $\|D\varphi_{s,t}(x)\| \leq C e^{-\lambda(t-s)}$ for $t > s$.

\textbf{Claim:} 
Trajectories remain in the linearization neighborhood for all $t \ge T^*$. 
Because the Jacobian $A = Df(\mathcal{A})$ has eigenvalues with strictly negative real parts, there exists a symmetric positive definite matrix $P$ solving the continuous Lyapunov equation $A^\top P + P A = -I$. 
We define the adapted Lyapunov function $V(x) = (x-\mathcal{A})^\top P (x-\mathcal{A})$, whose sublevel sets form invariant ellipsoids around $\mathcal{A}$. 
Along the unperturbed flow $f$, the derivative satisfies $\dot{V}_f(x) \le -c_1 \|x-\mathcal{A}\|^2$ inside $B_r$ for some constant $c_1 > 0$.
 For the approximated flow $\hat{f}$, the derivative is:
$$ \dot{V}_{\hat{f}}(x) = \nabla V(x) \cdot \hat{f}(x) = \dot{V}_f(x) + \nabla V(x) \cdot (\hat{f}(x) - f(x)). $$
Since $\|\nabla V(x)\| \le 2\|P\|\|x-\mathcal{A}\|$ and the uniform perturbation is bounded by $\|f - \hat{f}\|_{C^0} < \nu$, we obtain:
$$ \dot{V}_{\hat{f}}(x) \le -c_1 \|x-\mathcal{A}\|^2 + 2\|P\|\|x-\mathcal{A}\| \nu. $$
This derivative is strictly negative for $\|x-\mathcal{A}\| > 2\|P\|\nu / c_1$. 
Thus, for $\nu$ sufficiently small, $\dot{V}_{\hat{f}} < 0$ on the boundary of an ellipsoidal sublevel set of $V$ contained entirely within $B_r$, ensuring that any trajectory entering this set is permanently trapped.

By the variation of constants formula~\citep{vanhandel2007filtering}, applied with the linearized flow's fundamental matrix bounded in operator norm by $C$:
\[
\|\varphi(t, x_0) - \hat{\varphi}(t, x_0)\| \leq C K_1 \nu \cdot e^{-\lambda(t - T^*)} + \int_{T^*}^t C e^{-\lambda(t-s)} \nu \, ds \leq C K_1 \nu + \frac{C\nu}{\lambda}.
\]
Setting $K = \lambda C K_1 + C$ gives the bound $K\nu/\lambda$.

\textit{Explicit $\nu_0$:} Set $\nu_0 = \min(r/(2K_1), \lambda r/(2C), \delta_{\mathrm{ss}})$ where $\delta_{\mathrm{ss}}$ is the structural stability radius.
\end{proof}

\begin{lemma}[Basin error volume]\label{lem:basin_error}
For any $\delta > 0$, there exists $\eta_{\mathrm{vol}} > 0$ such that $\|f - \hat{f}\|_{C^1} < \eta_{\mathrm{vol}}$ implies:
\(
\mu(\mcE_{\mathrm{basin}}) < \delta,
\)
where $\mcE_{\mathrm{basin}} = \bigcup_i (\boa(\mathcal{A}_i^f) \setminus \boa(\mathcal{A}_i^{\hat{f}}))$.
\end{lemma}

\begin{proof}
By structural stability, $\|f - \hat{f}\|_{C^1} < \eta$ implies topological equivalence via homeomorphism $h$ with $\|h - \mathrm{id}\|_{C^0} \to 0$ as $\eta \to 0$.

The basin error set is contained in the $\|h - \mathrm{id}\|_{C^0}$-neighborhood of the separatrices $S = \bigcup_i \partial \boa(\mathcal{A}_i)$.
Since separatrices are stable manifolds of saddles with codimension $\geq 1$, they have Lebesgue measure zero.
By continuity of measure, $\mu(\mcE_{\mathrm{basin}}) \to 0$ as $\eta \to 0$.
\end{proof}

\begin{lemma}[Forward invariance under approximation]\label{lem:forward_invariance}
Let $f \in \mathfrak{X}^1(\mcX)$ be strictly inward-pointing at $\partial \mcX$ with constant $\gamma > 0$ (Definition~\ref{def:inward_pointing}).
If $\hat{f}: \mcX \to \reals^n$ satisfies $\|f - \hat{f}\|_{C^0} < \gamma$, then $\hat{f} \in \mathfrak{X}^1(\mcX)$, i.e., trajectories of $\hat{f}$ remain in $\mcX$ for all $t \geq 0$.
\end{lemma}

\begin{proof}
Let $\nu(x)$ denote the outward unit normal at $x \in \partial \mcX$.
By the strictly inward-pointing condition, $f(x) \cdot \nu(x) < -\gamma$ for all $x \in \partial \mcX$.

For $\hat{f}$ with $\|f - \hat{f}\|_{C^0} < \gamma$:
\[
\hat{f}(x) \cdot \nu(x) = f(x) \cdot \nu(x) + (\hat{f}(x) - f(x)) \cdot \nu(x) < -\gamma + \|\hat{f} - f\|_{C^0} < 0.
\]
Thus $\hat{f}$ also points strictly inward at $\partial \mcX$.
By the flow box theorem, trajectories starting in $\mcX$ cannot exit through $\partial \mcX$, so $\hat{\varphi}(t, x_0) \in \mcX$ for all $t \geq 0$ and $x_0 \in \mcX$.
\end{proof}

\begin{proof}[Full proof of Theorem~\ref{thm:uap_fp}]
Given $\varepsilon, \delta > 0$:

\textbf{Step 1 (spatial radius):} Choose a spatial radius $\eta > 0$ small enough that $\mu(N_\eta(S)) < \delta/2$. This is possible because $S$ has measure zero (Theorem~\ref{thm:stable_manifold}) and $\mu$ is continuous from above.

\textbf{Step 2 (perturbation budget for basin error):} By Lemma~\ref{lem:basin_error}, there exists a perturbation budget $\eta_{\mathrm{vol}} = \eta_{\mathrm{vol}}(\eta) > 0$ such that $\|f - \hat{f}\|_{C^1} < \eta_{\mathrm{vol}}$ implies $\mcE_{\mathrm{basin}} \subset N_\eta(S)$.

\textbf{Step 3 (perturbation budget for trajectory error):} Define $\mcV \coloneqq \mcX \setminus N_\eta(S)$, depending on $f, \eta$ alone. By Lemma~\ref{lem:traj_bound}, the constants $K_1 = (e^{LT^*}-1)/L$ and $K = K(\eta, f)$ are determined a priori; choose $\eta_{\mathrm{traj}} = \lambda\varepsilon/K$ so the trajectory error is $< \varepsilon$ on $\mcV$.

\textbf{Step 4 (combined budget):} Let $\gamma > 0$ be the strict inward-pointing constant. Set
\[
\eta_0 \coloneqq \min(\eta_{\mathrm{vol}}, \eta_{\mathrm{traj}}, \gamma).
\]
All three quantities are now perturbation budgets in the $C^1$ norm.

\textbf{Step 5 (apply UAP):} By Theorem~\ref{thm:uap_vf}, choose $\hat{f} \in \hat{\mcF}$ with $\|f - \hat{f}\|_{C^1} < \eta_0$.

Since $\|f - \hat{f}\|_{C^0} < \gamma$, forward invariance holds by Lemma~\ref{lem:forward_invariance}.
By Step~2, $\mcE_{\mathrm{basin}} \subset N_\eta(S)$, so for the error set $\mu\bigl(\{x_0 : \sup_t \|\varphi - \hat{\varphi}\| > \varepsilon\}\bigr) \le \mu(N_\eta(S)) < \delta/2 < \delta$.
\end{proof}

%==============================================================================
\section{Proof of Theorem~\ref{thm:uap_lc}}\label{app:proof_lc}

% we prove it for multiplicative and addiitive architectures

\begin{lemma}[Structural stability of hyperbolic limit cycles]\label{lem:lc_stability}
Let $\gamma$ be a hyperbolic periodic orbit with period $T$.
There exist $L_H, L_T, \delta_0 > 0$ such that $\|f - \hat{f}\|_{C^1} < \delta_0$ implies:
\begin{enumerate}[label=(\roman*)]
\item $\hat{f}$ has a unique hyperbolic periodic orbit $\hat{\gamma}$ near $\gamma$
\item $d_H(\gamma, \hat{\gamma}) \leq L_H \|f - \hat{f}\|_{C^1}$
\item $|\hat{T} - T| \leq L_T \|f - \hat{f}\|_{C^1}$
\end{enumerate}
\end{lemma}

\begin{proof}
By Fenichel's persistence theorem for normally hyperbolic invariant manifolds~\citep{fenichel1971persistence}.
The period functional is Fr\'echet differentiable for hyperbolic cycles~\citep{chicone2006ode}, hence Lipschitz continuous.
\end{proof}

\subsection{With multiplicative correction}

\begin{lemma}[Period correction]\label{lem:period_correction}
Let $\hat{f}$ have hyperbolic limit cycle $\hat{\gamma}$ with period $\hat{T}$.
For target period $T$ with $|\hat{T} - T|$ small, there exists $\alpha^* = (\hat{T} - T)/T$ such that:
\[
\hat{f}_{\alpha^*} \coloneqq (1 + \alpha^* \psi) \hat{f}
\]
has periodic orbit $\hat{\gamma}$ with period exactly $T$, where $\psi$ is a bump function with $\psi \equiv 1$ on $\hat{\gamma}$.
\end{lemma}

\begin{proof}
By time reparametrization, $(1 + \alpha \psi) \hat{f}$ preserves $\hat{\gamma}$ as a periodic orbit with period $P(\alpha) = \hat{T}/(1 + \alpha)$ (since $\psi = 1$ on $\hat{\gamma}$).
Setting $P(\alpha^*) = T$ yields $\alpha^* = (\hat{T} - T)/T$.
\end{proof}

\begin{lemma}[Simultaneous period correction]\label{lem:multi_period_correction}
Let the vector field $\hat{f}$ possess $N$ disjoint hyperbolic limit cycles $\{\hat{\gamma}_1, \dots, \hat{\gamma}_N\}$ with respective periods $\hat{\mathbf{T}} = (\hat{T}_1, \dots, \hat{T}_N)$.
Let $\{\psi_1, \dots, \psi_N\}$ be a set of smooth bump functions where each $\psi_i$ is equal to $1$ on $\hat{\gamma}_i$ and decays such that $\sup_{x \in \hat{\gamma}_k} |\psi_i(x)| < \xi$ for $i \neq k$.
Consider the parameterized family of vector fields:$$\hat{f}_{\boldsymbol{\alpha}}(x) \coloneqq \left( 1 + \sum_{i=1}^N \alpha_i \psi_i(x) \right) \hat{f}(x), \quad \boldsymbol{\alpha} \in \mathbb{R}^N.$$
For any target period vector $\mathbf{T} = (T_1, \dots, T_N)$ sufficiently close to $\hat{\mathbf{T}}$, provided the leakage $\xi$ is sufficiently small, there exists a unique parameter vector $\boldsymbol{\alpha}^*$ such that the limit cycles of $\hat{f}_{\boldsymbol{\alpha}^*}$ have periods exactly equal to $\mathbf{T}$.
\end{lemma}

\begin{proof}
Since the limit cycles are hyperbolic, structural stability implies they persist for small $\boldsymbol{\alpha}$. 
Let $\mathbf{T}(\boldsymbol{\alpha}) : \mathbb{R}^N \to \mathbb{R}^N$ be the map taking the parameters to the periods of the resulting cycles. We aim to solve $\mathbf{T}(\boldsymbol{\alpha}) = \mathbf{T}_{target}$.We analyze the Jacobian matrix of this map, $\mathbf{J} \in \mathbb{R}^{N \times N}$, evaluated at $\boldsymbol{\alpha} = \mathbf{0}$. 
The entry $J_{ki} = \frac{\partial T^k}{\partial \alpha_i}$ represents the sensitivity of the $k$-th cycle's period to the $i$-th bump function.

The period of a cycle $\gamma$ subject to a velocity scaling $v(x) \to (1 + \eta(x))v(x)$ is given by $T = \oint_{\gamma} \frac{ds}{\| (1+\eta) v \|} \approx \oint \frac{ds}{\|v\|} - \oint \frac{\eta}{\|v\|} ds$.
Thus, the partial derivatives are:$$ \frac{\partial T^k}{\partial \alpha_i} \bigg|_{\boldsymbol{\alpha}=0} = - \int_0^{\hat{T}_k} \psi_i(\hat{\gamma}_k(t)) \, dt. $$
\textbf{1. Diagonal terms ($i=k$):}
Since $\psi_k \equiv 1$ on $\hat{\gamma}_k$, the integral is exact:$$ J_{kk} = - \int_0^{\hat{T}_k} 1 \, dt = -\hat{T}_k. $$
\textbf{2. Off-diagonal terms ($i \neq k$):}
Since the support of $\psi_i$ is concentrated on cycle $i$, its value on cycle $k$ is bounded by the leakage parameter $\xi$:$$ |J_{ki}| = \left| - \int_0^{\hat{T}_k} \psi_i(\hat{\gamma}_k(t)) \, dt \right| \le \int_0^{\hat{T}_k} \xi \, dt = \xi \hat{T}_k. $$
\textbf{3. Invertibility via Diagonal Dominance:}
We construct the Jacobian to be strictly diagonally dominant. The condition $|J_{kk}| > \sum_{i \neq k} |J_{ki}|$ becomes:$$ \hat{T}_k > \sum_{i \neq k} \xi \hat{T}_k = (N-1)\xi \hat{T}_k \implies 1 > (N-1)\xi. $$
By choosing the bump functions such that the leakage $\xi < \frac{1}{N-1}$, the matrix $\mathbf{J}$ is strictly diagonally dominant. 
By the Lévy-Desplanques Theorem (Theorem~\ref{lem:levy_desplanques}), $\mathbf{J}$ is non-singular (invertible).

Since the Jacobian is invertible at $\mathbf{0}$, the Inverse Function Theorem guarantees that $\mathbf{T}(\boldsymbol{\alpha})$ is a local diffeomorphism mapping a neighborhood of $\mathbf{0}$ to a neighborhood of $\hat{\mathbf{T}}$. 
Thus, for any target periods $\mathbf{T}$ sufficiently close to $\hat{\mathbf{T}}$, there exists a unique $\boldsymbol{\alpha}^*$ such that $\mathbf{T}(\boldsymbol{\alpha}^*) = \mathbf{T}$.
\end{proof}

\begin{remark}[Hyperbolicity preservation]
The corrected vector field $\hat{f}_{\alpha^*}$ preserves the hyperbolicity of $\hat{\gamma}$.
Since $\psi$ is $C^\infty$ with support in a tubular neighborhood of $\hat{\gamma}$ and $|\alpha^*| \leq L_T\eta/T$ is small, the correction is a small $C^1$ perturbation.
By structural stability of hyperbolic periodic orbits, $\hat{\gamma}$ remains hyperbolic for the corrected system.
Moreover, $\|\hat{f} - \hat{f}_{\alpha^*}\|_{C^1} = O(|\alpha^*| \cdot \|\psi\|_{C^1} \cdot \|\hat{f}\|_{C^1})$, which can be made arbitrarily small by choosing $\eta$ small.
\end{remark}

\begin{lemma}[Trajectory bounds near limit cycles]\label{lem:lc_traj_bound}
Let $\mcV \subset \boa(\gamma)$ with $$\inf_{x \in \mcV} \dist(x, \partial \boa(\gamma)) \geq \eta > 0.$$
If periods match exactly ($T = \hat{T}$) and $d_H(\gamma, \hat{\gamma}) < \varepsilon_{\mathrm{geom}}$, then:
\[
\limsup_{t \to \infty} \|\varphi(t, x_0) - \hat{\varphi}(t, x_0)\| \leq \varepsilon_{\mathrm{geom}} + O(\|f - \hat{f}\|_{C^1}).
\]
\end{lemma}

\begin{proof}
The argument mirrors the two-stage analysis of Lemma~\ref{lem:traj_bound}, with the cycle's asymptotic phase function (Definition~\ref{def:asymptotic_phase}) replacing the role of the fixed point.

\textit{Transient phase $[0, T^*]$.}
Since $\mcV$ is bounded away from the basin boundary, there is a uniform transient time $T^* < \infty$ before trajectories from $\mcV$ enter a tubular neighborhood of $\gamma$ and $\hat{\gamma}$.
On this interval, Gr\"onwall's inequality applied to the difference $e(t) = \varphi(t, x_0) - \hat{\varphi}(t, x_0)$ yields
\[
\|e(t)\| \;\leq\; \|f - \hat{f}\|_{C^0} \cdot \frac{e^{Lt} - 1}{L}, \qquad t \in [0, T^*],
\]
where $L = \Lip(f)$.

\textit{Asymptotic phase $[T^*, \infty)$.}
After $T^*$, $\varphi(t, x_0)$ converges exponentially to $\gamma$ at asymptotic phase $\Psi^f(x_0)$, and $\hat{\varphi}(t, x_0)$ converges exponentially to $\hat{\gamma}$ at $\Psi^{\hat{f}}(x_0)$.
With exact period matching ($T = \hat{T}$), the phase difference $\Delta\Psi = \Psi^f(x_0) - \Psi^{\hat{f}}(x_0)$ is constant in $t$.
The asymptotic trajectory error decomposes as
\[
\limsup_{t \to \infty} \|e(t)\| \;\leq\; d_H(\gamma, \hat{\gamma}) + L_\gamma |\Delta\Psi|,
\]
where $L_\gamma = \sup \|\dot{\gamma}\|$.
By Floquet smoothness of the phase function on $C^1$-close vector fields (the same bound used in the proof of Lemma~\ref{lem:phase_coherence}), $|\Delta\Psi| \leq C_\Psi \|f - \hat{f}\|_{C^1}$, giving the claim.
\end{proof}

\begin{proof}[Proof of Theorem~\ref{thm:uap_lc} (multiplicative correction)]
\textbf{Step 1 (Base approximation):}
By Theorem~\ref{thm:uap_vf}, choose $\tilde{f} \in \hat{\mcF}$ with $\|f - \tilde{f}\|_{C^1} < \eta$.
By Lemma~\ref{lem:lc_stability}, $\tilde{f}$ has limit cycles $\tilde{\gamma}_i$ with periods $\tilde{T}_i$.

\textbf{Step 2 (Period correction):}
For single cycle: $\hat{f} = c^* \tilde{f}$ with $c^* = \tilde{T}/T$ (Lemma~\ref{lem:period_correction}).
For multiple cycles with disjoint tubular neighborhoods: $\hat{f} = (1 + \sum_i \alpha_i^* \psi_i) \tilde{f}$ (Lemma~\ref{lem:multi_period_correction}).

\textbf{Step 3 (Error control):}
By Lemma~\ref{lem:basin_error} (adapted for limit cycles), choose $\eta$ so $\mu(\mcE_{\mathrm{basin}}) < \delta/2$.
By Lemma~\ref{lem:lc_traj_bound}, exact period matching bounds asymptotic error.
Choose $\eta$ so geometric and transient errors sum to $< \varepsilon$.

The error set has measure $< \delta$, completing the proof.
\end{proof}

%%%%ALT

\subsection{With additive correction}

\begin{lemma}[Bump Function Realizability]\label{lem:bump_realizability}
Let $\tilde{f}$ be an approximation of $f$ with limit cycles $\tilde{\gamma}_i \subset N_i$. Let $\nu, \zeta > 0$ be tolerance parameters.
There exist functions $\{\hat{\mathbf{\Phi}}_i\}_{i=1}^N \subset \hat{\mcF}$ such that for each $i$:
\begin{enumerate}
\item Alignment: $\| \hat{\Phi}_i(x) - \tilde{f}(x) \| < \nu$ for all $x \in \tilde{\gamma}_i$.
\item Support Decay ($C^1$ Leakage): $\|\hat{\mathbf{\Phi}}_i(x)\|_{C^1} < \zeta$ for all $x \notin M_i$.
\item Smoothness: $\|\hat{\mathbf{\Phi}}_i\|_{C^1}$ is bounded.
\end{enumerate}
\end{lemma}

\begin{proof}
Let $\mathbf{\Phi}^*_i$ be a smooth ``ideal" bump function that equals $\tilde{f}$ on $N_i$ and vanishes outside $M_i$. 
Since $\hat{\mcF}$ possesses the $C^1$ UAP, there exists $\hat{\mathbf{\Phi}}_i \in \hat{\mcF}$ such that $\|\hat{\mathbf{\Phi}}_i - \mathbf{\Phi}^*_i\|_{C^1} < \min(\nu, \zeta)$.
The condition on $x \in \tilde{\gamma}_i$ is satisfied by $\nu$-closeness.
The condition on $x \notin M_i$ (where $\mathbf{\Phi}^*_i = 0$) is satisfied because the UAP yields $\zeta$-closeness in the $C^1$ norm.
\end{proof}

\begin{definition}[Adjoint Solution]\label{def:adjoint}
Let $\tilde{\varphi}_t$ denote the flow generated by the vector field $\tilde{f}$, and let $\tilde{\gamma}(t)$ be a hyperbolic periodic orbit with period $\tilde{T}$. The \textbf{linear variational equation} along the orbit is given by $\dot{u} = D\tilde{f}(\tilde{\gamma}(t))u$.
The corresponding \textbf{adjoint equation} is defined as:
\[
\dot{z} = -[D\tilde{f}(\tilde{\gamma}(t))]^\top z.
\]
There exists a unique $\tilde{T}$-periodic solution $Z(t)$ to this adjoint equation, normalized such that:
\[
\langle Z(t), \tilde{f}(\tilde{\gamma}(t)) \rangle \equiv 1 \quad \text{for all } t \in [0, \tilde{T}].
\]
The vector $Z(t)$ represents the sensitivity of the period to instantaneous perturbations of the vector field.
\end{definition}

\begin{lemma}[First Variation of the Period \citep{malkin1956some}]\label{lem:period_variation}
Consider the perturbed vector field $\hat{f}_{\alpha} = \tilde{f} + \alpha g$. The derivative of the period $T(\alpha)$ with respect to the parameter $\alpha$ at $\alpha=0$ is given exactly by:
\[
\frac{d T}{d \alpha}\bigg|_{\alpha=0} = - \int_0^{\tilde{T}} \langle Z(t), g(\tilde{\gamma}(t)) \rangle \, dt.
\]
\end{lemma}

See also \citet{guckenheimer1983nonlinear} and \citet{ermentrout2010mathematical}. 

\begin{lemma}[Exact Jacobian Form]\label{lem:exact_jacobian}
Consider the multi-cycle approximation $\hat{f}_{\boldsymbol{\alpha}} = \tilde{f} + \sum_{j=1}^N \alpha_j \hat{\mathbf{\Phi}}_j$.
Let $Z_i(t)$ be the normalized adjoint solution associated with the $i$-th limit cycle $\tilde{\gamma}_i$.
Applying Lemma~\ref{lem:period_variation} with perturbation $g = \hat{\mathbf{\Phi}}_j$, the entries of the period Jacobian matrix $J = D_{\boldsymbol{\alpha}} \mathbf{T}(\mathbf{0})$ are:
\[
J_{ij} = \frac{\partial T_i}{\partial \alpha_j}\bigg|_{\boldsymbol{\alpha}=\mathbf{0}} = - \int_0^{\tilde{T}_i} \langle Z_i(t), \hat{\mathbf{\Phi}}_j(\tilde{\gamma}_i(t)) \rangle \, dt.
\]
\end{lemma}

\begin{lemma}[Robust Invertibility of Period Jacobian]\label{lem:robust_inv}
Let $C_Z = \max_{i} \sup_{t \in [0, \tilde{T}_i]} \|Z_i(t)\|$ be the maximum norm of the adjoint solutions.
Assume the approximator bump functions $\hat{\mathbf{\Phi}}_j$ satisfy the following proximity conditions:
\begin{itemize}
    \item \textbf{Alignment (on cycle):} $\|\hat{\mathbf{\Phi}}_i(x) - \tilde{f}(x)\| < \nu$ for all $x \in \tilde{\gamma}_i$.
    \item \textbf{Leakage (off cycle):} $\|\hat{\mathbf{\Phi}}_j(x)\| < \zeta$ for all $x \in \tilde{\gamma}_i$ where $i \neq j$.
\end{itemize}
If the approximation parameters satisfy the condition:
\[
\nu + (N-1)\zeta < \frac{1}{C_Z},
\]
then the Jacobian matrix $J = D_{\boldsymbol{\alpha}} \mathbf{T}(\mathbf{0})$ is strictly diagonally dominant and therefore non-singular.
\end{lemma}

\begin{proof}
We analyze the entries of the Jacobian $J_{ij}$ defined in Lemma~\ref{lem:exact_jacobian}.

\textbf{1. Diagonal Terms ($i=j$):}
On the cycle $\tilde{\gamma}_i$, we write $\hat{\mathbf{\Phi}}_i = \tilde{f} + \boldsymbol{\Delta}_i$, where $\|\boldsymbol{\Delta}_i\| < \nu$.
Substituting this into the integral:
\[
J_{ii} = - \int_0^{\tilde{T}_i} \langle Z_i(t), \tilde{f}(\tilde{\gamma}_i(t)) + \boldsymbol{\Delta}_i(t) \rangle \, dt.
\]
Using linearity and the normalization property $\langle Z_i(t), \tilde{f}(\tilde{\gamma}_i(t)) \rangle \equiv 1$:
\[
J_{ii} = - \underbrace{\int_0^{\tilde{T}_i} 1 \, dt}_{= \tilde{T}_i} - \underbrace{\int_0^{\tilde{T}_i} \langle Z_i(t), \boldsymbol{\Delta}_i(t) \rangle \, dt}_{E_{ii}}.
\]
We bound the error term $E_{ii}$:
\[
|E_{ii}| \le \int_0^{\tilde{T}_i} \|Z_i(t)\| \|\boldsymbol{\Delta}_i(t)\| \, dt \le \tilde{T}_i C_Z \nu.
\]
Thus, the magnitude of the diagonal entry is bounded from below:
\[
|J_{ii}| \ge \tilde{T}_i (1 - C_Z \nu).
\]

\textbf{2. Off-Diagonal Terms ($i \neq j$):}
For $i \neq j$, the function $\hat{\mathbf{\Phi}}_j$ acts on the cycle $\tilde{\gamma}_i$. By the leakage assumption, $\|\hat{\mathbf{\Phi}}_j\| < \zeta$ on $\tilde{\gamma}_i$.
\[
|J_{ij}| = \left| - \int_0^{\tilde{T}_i} \langle Z_i(t), \hat{\mathbf{\Phi}}_j(\tilde{\gamma}_i(t)) \rangle \, dt \right| \le \int_0^{\tilde{T}_i} \|Z_i(t)\| \zeta \, dt \le \tilde{T}_i C_Z \zeta.
\]

\textbf{3. Diagonal Dominance:}
For strict diagonal dominance, we require $|J_{ii}| > \sum_{j \neq i} |J_{ij}|$.
Substituting our bounds:
\[
\tilde{T}_i (1 - C_Z \nu) > (N-1) \tilde{T}_i C_Z \zeta.
\]
Dividing by $\tilde{T}_i$ and rearranging:
\[
1 > C_Z \nu + C_Z (N-1) \zeta \implies \nu + (N-1)\zeta < \frac{1}{C_Z}.
\]
This condition holds by hypothesis, ensuring $J$ is invertible.
\end{proof}

\begin{proposition}[Correction Existence]\label{prop:correction_existence}
Let $\Delta T = \mathbf{T}_{\text{target}} - \tilde{\mathbf{T}}$.
Since $J(\mathbf{0})$ is non-singular (Lemma~\ref{lem:robust_inv}), by the Inverse Function Theorem, there exists a radius $R_{sol} > 0$ and a constant $\kappa > 0$ such that if $\|\Delta T\| < R_{sol}$, there exists a unique parameter vector $\boldsymbol{\alpha}^*$ satisfying $\mathbf{T}(\boldsymbol{\alpha}^*) = \mathbf{T}_{\text{target}}$ with the bound:
$$\|\boldsymbol{\alpha}^*\| \le \kappa \|\Delta T\|.$$
\end{proposition}

\begin{lemma}[Geometric Stability]\label{lem:geo_stability}
Let $\hat{f} = \tilde{f} + \sum \alpha^*_i \hat{\mathbf{\Phi}}_i$. 
Then:
\begin{enumerate}
\item \textbf{Cycle Shift:} $d_H(\tilde{\gamma}, \hat{\gamma}) \le L_H \, N B_\Phi \, \|\boldsymbol{\alpha}^*\|$.
\item \textbf{Basin Error:} $\mu(\mcB(\tilde{f}) \triangle \mcB(\hat{f})) \le K_{\text{sep}} \, N B_\Phi \, \|\boldsymbol{\alpha}^*\|$.
\end{enumerate}
where $L_H$ is the Lipschitz constant from Lemma~\ref{lem:lc_stability}(ii) and $K_{\text{sep}}$ is the basin-boundary Lipschitz constant.
\end{lemma}

\begin{proof}
(i) The cycle shift follows from Lemma~\ref{lem:lc_stability}(ii) applied to the perturbation $\hat{f} - \tilde{f} = \sum_i \alpha_i^* \hat{\mathbf{\Phi}}_i$, which has $C^1$ norm $\le N B_\Phi \|\boldsymbol{\alpha}^*\|$ by Lemma~\ref{lem:bump_realizability}.
(ii) Basin boundaries between adjacent attractors are stable manifolds of intermediate saddles. Their position depends $C^1$-Lipschitz on the vector field~\citep[Theorem~5.7.5]{chicone2006ode}, so the symmetric-difference measure is bounded by $K_{\text{sep}}$ times the $C^1$ perturbation magnitude $N B_\Phi \|\boldsymbol{\alpha}^*\|$.
\end{proof}

\begin{proof}[Proof of Theorem~\ref{thm:uap_lc} (additive correction)]
\textbf{Step 1: Configuration (Fixing Constants).}
Set bump parameters $\nu = \frac{1}{3 C_Z}$ and $\zeta = \frac{1}{3 N C_Z}$. By Lemma~\ref{lem:robust_inv}, this ensures the period Jacobian is invertible.
By Lemma~\ref{lem:bump_realizability}, such bumps exist in $\hat{\mcF}$ with bound $B_{\Phi}$.

\textbf{Step 2: Base Approximation.}
We choose a base approximation $\tilde{f} \in \hat{\mcF}$ such that $\|f - \tilde{f}\|_{C^1} < \eta_{\mathrm{base}}$.
By Lemma~\ref{lem:lc_stability}(ii), the initial period error is bounded by $\|\Delta T\| \leq L_T \eta_{\mathrm{base}}$.
From Proposition~\ref{prop:correction_existence}, if a solution exists, the correction magnitude is bounded by $\|\boldsymbol{\alpha}^*\| \le \kappa L_T \eta_{\mathrm{base}}$.
We now determine the required $\eta_{\mathrm{base}}$ by enforcing four simultaneous constraints:

\textbf{Constraint A: Solvability.}
To invoke Proposition~\ref{prop:correction_existence}, we require $\|\Delta T\| < R_{sol}$. Using the structural stability bound:
\[ \eta_{\mathrm{base}} < \frac{R_{sol}}{L_T}. \]

\textbf{Constraint B: Basin Stability ($\delta$).}
The total basin error comes from the base approximation (structural stability) plus the separatrix shift due to the bump-function correction.
Since separatrix shifts scale with the $C^1$ norm of the perturbation, the relevant bound on the corrective term is $\|\sum_j \alpha_j^* \hat{\mathbf{\Phi}}_j\|_{C^1} \leq N B_\Phi \|\boldsymbol{\alpha}^*\|$ (Lemma~\ref{lem:bump_realizability}, in conjunction with Lemma~\ref{lem:geo_stability}(ii)).
Combining,
\[ K_{\text{sep}} \eta_{\mathrm{base}} + K_{\text{sep}} N B_\Phi (\kappa L_T \eta_{\mathrm{base}}) < \delta \implies \eta_{\mathrm{base}} < \frac{\delta}{K_{\text{sep}}(1 + N B_\Phi \kappa L_T)}. \]

\textbf{Constraint C: Trajectory Accuracy ($\varepsilon$).}
We require the asymptotic trajectory error to satisfy $\limsup_{t \to \infty} \|\varphi(t, x) - \hat{\varphi}(t, x)\| < \varepsilon$.
By the triangle inequality, this error is bounded by the sum of three contributions: the geometric distance between the limit cycles, the residual phase offset that gets frozen once periods match, and the asymptotic phase mismatch coming from non-equal isochron foliations.

\begin{enumerate}
    \item \textbf{Geometric Error:} The Hausdorff distance between the true cycle $\gamma$ and the corrected cycle $\hat{\gamma}$ satisfies:
    \[
    d_H(\gamma, \hat{\gamma}) \le \underbrace{d_H(\gamma, \tilde{\gamma})}_{\text{Base Error}} + \underbrace{d_H(\tilde{\gamma}, \hat{\gamma})}_{\text{Correction Shift}} \le L_H \eta_{\mathrm{base}} + L_H N B_{\Phi} \|\boldsymbol{\alpha}^*\|.
    \]

    \item \textbf{Transient trajectory divergence (Gr\"onwall).}
    On the transient interval $[0, T_{\text{trans}}]$ before trajectories lock onto their respective cycles, the difference $e(t) \coloneqq \varphi(t, x_0) - \hat{\varphi}(t, x_0)$ satisfies $\dot{e} = f(\varphi) - \hat{f}(\hat{\varphi})$, which decomposes as
    \[
    \dot{e} = \bigl( f(\varphi) - f(\hat{\varphi}) \bigr) + \bigl( f(\hat{\varphi}) - \hat{f}(\hat{\varphi}) \bigr).
    \]
    The first term is bounded in norm by $L\|e\|$ where $L = \Lip(f)$; the second by $\|f - \hat{f}\|_{C^0}$.
    Gr\"onwall's inequality~\citep{chicone2006ode} then yields
    \begin{equation}\label{eq:gronwall_transient}
    \|\varphi(t, x_0) - \hat{\varphi}(t, x_0)\| \;\leq\; \|f - \hat{f}\|_{C^0} \cdot \frac{e^{Lt} - 1}{L}, \qquad t \in [0, T_{\text{trans}}],
    \end{equation}
    where $\|f - \hat{f}\|_{C^0} \leq \eta_{\mathrm{base}} + N B_\Phi \|\boldsymbol{\alpha}^*\|$.
    Once exact period matching takes effect (Lemma~\ref{lem:multi_period_correction} or Proposition~\ref{prop:correction_existence}), this transient offset is frozen (it neither amplifies nor decays), and it bounds the asymptotic phase contribution.
    The Gr\"onwall amplification $G_T \coloneqq (e^{L T_{\text{trans}}} - 1)/L$ is finite for each fixed transient horizon, so this term is $O(\eta_{\mathrm{base}} + \|\boldsymbol{\alpha}^*\|)$ but with a Lipschitz-controlled multiplicative constant.

    \item \textbf{Asymptotic phase mismatch (isochron Floquet bound).}
    After the transient, trajectories track their respective cycles at asymptotic phases $\Psi^f(x_0)$ and $\Psi^{\hat{f}}(x_0)$ (Definition~\ref{def:asymptotic_phase}).
    With exact period matching, the phase difference $\Delta\Psi \coloneqq \Psi^f(x_0) - \Psi^{\hat{f}}(x_0)$ is constant in $t$ and contributes to the trajectory error by $L_\gamma |\Delta\Psi|$, where $L_\gamma = \sup \|\dot{\gamma}\|$.
    By the same Floquet-smoothness argument used in Lemma~\ref{lem:phase_coherence} (Stage~2, applied directly to the isolated cycle pair $\gamma, \hat{\gamma}$),
    \[
    |\Delta\Psi| \leq C_\Psi \|f - \hat{f}\|_{C^1} \leq C_\Psi (\eta_{\mathrm{base}} + N B_\Phi \|\boldsymbol{\alpha}^*\|).
    \]
\end{enumerate}

Substituting the correction bound $\|\boldsymbol{\alpha}^*\| \le \kappa L_T \eta_{\mathrm{base}}$ and summing:
\[
\text{Total Error} \;\leq\; \eta_{\mathrm{base}} \cdot (1 + N B_\Phi \kappa L_T) \cdot \bigl( L_H + G_T + L_\gamma C_\Psi \bigr).
\]
To ensure the total error is strictly less than $\varepsilon$, we impose:
\begin{equation}\label{eq:constraint_C}
\eta_{\mathrm{base}} \;<\; \frac{\varepsilon}{(1 + N B_\Phi \kappa L_T)\bigl( L_H + G_T + L_\gamma C_\Psi \bigr)}.
\end{equation}
The Lipschitz amplification $G_T = (e^{L T_{\text{trans}}}-1)/L$ enters multiplicatively; for systems with large $L$ or long transients, this constant can dominate the budget, but it is finite and yields a positive lower bound on the achievable $\eta_{\mathrm{base}}$.

\textbf{Constraint D: Forward Invariance.}
To satisfy Lemma~\ref{lem:forward_invariance}, the total error must be less than $\gamma$:
\[ \|f - \hat{f}\| \le \|f - \tilde{f}\| + \|\hat{f} - \tilde{f}\| \le \eta_{\mathrm{base}}(1 + N B_{\Phi} \kappa L_T) < \gamma. \]

\textbf{Conclusion.}
Let $\eta^*$ be the minimum of the bounds derived in Constraints A--D.
By the $C^1$ UAP of $\hat{\mcF}$, there exists $\tilde{f}$ satisfying $\|f - \tilde{f}\|_{C^1} < \eta^*$.
This $\tilde{f}$ is sufficiently accurate that a period correction $\boldsymbol{\alpha}^*$ exists (A), maintains basin structure (B), ensures trajectory tracking (C), and preserves the domain (D).
\end{proof}

%==============================================================================
\section{Proof of Theorem~\ref{thm:uap_ca}}\label{app:proof_ca}

\begin{theorem}[Fenichel persistence~{\citep{fenichel1971persistence}}]\label{thm:fenichel}
Let $\mcM$ be a compact normally hyperbolic invariant manifold (NHIM) for $f$.
There exist $\delta_0 > 0$ and $L_F > 0$ (depending on $\mcM$) such that for $\|f - \hat{f}\|_{C^1} < \delta_0$:
\begin{enumerate}[label=(\roman*)]
\item $\hat{f}$ has a unique NHIM $\hat{\mcM}$ diffeomorphic to $\mcM$;
\item $d_H(\mcM, \hat{\mcM}) \leq L_F \|f - \hat{f}\|_{C^1}$.
\end{enumerate}
\end{theorem}

\begin{theorem}[Kupka-Smale Density Theorem~\citep{kupka1963contribution}]\label{thm:kupka_smale}
Let $\mcX$ be a compact smooth manifold and $\mathfrak{X}^r(\mcX)$ be the space of $C^r$ vector fields ($r \ge 1$) endowed with the $C^r$ topology.
The set of vector fields $f \in \mathfrak{X}^r(\mcX)$ whose periodic orbits are all hyperbolic and whose stable/unstable manifolds intersect transversally is residual (and thus dense) in $\mathfrak{X}^r(\mcX)$.
\end{theorem}

\begin{remark}[Intuition: Genericity and Discretization]
The Kupka-Smale theorem essentially states that dynamical features (such as non-hyperbolic cycles or coincidental manifold intersections) are rare and fragile. 
If you randomly perturb a dynamical system, these degenerate features break apart into robust, hyperbolic structures.
For our proof, this provides a rigorous justification for \textbf{discretization}. 
While the original system $f$ may possess a continuum of non-isolated periodic orbits (a degenerate feature common in isochronous systems), there exists an arbitrarily close system $f^\dagger$ where this continuum has broken down into a finite skeleton of isolated, stable limit cycles. 
This allows us to approximate the continuous manifold problem with a finite number of fixed points or limit cycle problem, knowing the error between them is negligible ($\eta_{\mathrm{model}}$).
\end{remark}

Before proving Case~2, we record two ingredients that the existing argument tacitly assumed but never bounded: the \emph{asymptotic phase function} of a hyperbolic limit cycle, and the closeness of the target's and approximator's phase functions on each tile.

\begin{definition}[Asymptotic phase function]\label{def:asymptotic_phase}
Let $g \in \mathfrak{X}^1(\mcX)$ have a hyperbolic limit cycle $\gamma$ of period $T$ with basin $\boa(\gamma)$.
Fix a base point $p_0 \in \gamma$ and write $\gamma(\tau)$ for the orbit through $p_0$ at $\tau = 0$.
The \textbf{asymptotic phase function} $\Psi^g: \boa(\gamma) \to \reals/T\reals$ is the unique continuous map satisfying
\[
\lim_{t \to \infty} \bigl\| \varphi^g(t, x_0) - \gamma(t + \Psi^g(x_0)) \bigr\| = 0
\qquad \text{for all } x_0 \in \boa(\gamma).
\]
The level sets of $\Psi^g$ are the \emph{isochrons} of $\gamma$.
For an isochronous NHIM $\mcM$ foliated by orbits of common period $T$, the phase function $\Psi^g$ extends continuously to all of $\mcM$ and to a tubular neighborhood thereof where the normal contraction is well-defined.
\end{definition}

\begin{lemma}[Phase coherence under tiling]\label{lem:phase_coherence}
Let $f \in \mcF_{\mathrm{CA}}$ have isochronous NHIM $\mcM$ with global phase function $\Psi^f$ (Definition~\ref{def:asymptotic_phase}), and let $\hat{f}$ be the corrected approximator constructed in Steps~1--3 of the proof of Theorem~\ref{thm:uap_ca} Case~2, with discrete cycles $\{\hat{\gamma}_i\}_{i=1}^N$ forming an $\epsilon_{\mathrm{tile}}$-net of $\mcM$ and per-cycle phase functions $\Psi^{\hat{f}}_i$ defined on each basin $\boa(\hat{\gamma}_i)$.

There exist constants $C_\Psi, \rho_0 > 0$ depending only on $f$ such that for every $x_0$ lying in some $\boa(\hat{\gamma}_i)$ at distance $\geq \rho_0$ from $\partial \boa(\hat{\gamma}_i)$,
\[
\bigl| \Psi^f(x_0) - \Psi^{\hat{f}}_i(x_0) \bigr|
\;\leq\; C_\Psi \cdot \bigl( \|\hat{f} - f\|_{C^1} + \epsilon_{\mathrm{tile}} \bigr).
\]
\end{lemma}

\begin{proof}
The phase function for a hyperbolic limit cycle depends $C^1$-Lipschitz on the generating vector field at points bounded away from the cycle's basin boundary: this is a standard consequence of Floquet theory and the smooth dependence of the stable foliation on parameters~\citep[\S~III.7]{guckenheimer1983nonlinear}.
Quantitatively, if $g_0, g_1$ are $C^1$-close vector fields with hyperbolic cycles $\gamma_0, \gamma_1$ sharing a common basin region $\mcV$ at distance $\geq \rho_0$ from any separatrix, and with phase functions $\Psi^{g_0}, \Psi^{g_1}$, then
\begin{equation}\label{eq:phase_lipschitz}
\sup_{x \in \mcV} |\Psi^{g_0}(x) - \Psi^{g_1}(x)|
\;\leq\; C(\rho_0) \cdot \|g_0 - g_1\|_{C^1}.
\end{equation}

We apply~\eqref{eq:phase_lipschitz} in two stages, comparing $f$ to $\hat{f}$ via the proxy $f^\dagger$ and the base approximator $\tilde{f}$.

\textbf{Stage 1: $f \to f^\dagger$.}
By Fenichel persistence (Theorem~\ref{thm:fenichel}) applied to $\mcM$, the perturbed manifold $\mcM^\dagger$ satisfies $d_H(\mcM, \mcM^\dagger) \leq L_F \eta_{\mathrm{model}}$.
Let $\theta_i^* \in \Theta$ index the original target cycle $\gamma_{\theta_i^*}$ chosen as the cycle on $\mcM$ closest to the projection of $\gamma_i^\dagger$ onto $\mcM$.
Since $\gamma_i^\dagger \subset \mcM^\dagger$, we have $d_H(\gamma_i^\dagger, \gamma_{\theta_i^*}) = O(\eta_{\mathrm{model}})$: the transverse component is bounded by $L_F \eta_{\mathrm{model}}$ via the Fenichel bound, and the in-manifold component is absorbed into the same $O(\eta_{\mathrm{model}})$ since $\gamma_{\theta_i^*}$ is chosen optimally.
Applying~\eqref{eq:phase_lipschitz} on the basin region $\mcV_i$,
\[
\sup_{x \in \mcV_i} \bigl| \Psi^f_{\theta_i^*}(x) - \Psi^{f^\dagger}_i(x) \bigr|
\;\leq\; C(\rho_0) \cdot \eta_{\mathrm{model}},
\]
where $\Psi^f_{\theta_i^*}$ denotes the restriction of the global $\Psi^f$ to the foliation basin of $\gamma_{\theta_i^*}$.
Since $\Psi^f$ is Lipschitz on $\mcM$ in the in-manifold direction,
\[
|\Psi^f(x) - \Psi^f_{\theta_i^*}(x)|
\;\leq\; \mathrm{Lip}_{\mcM}(\Psi^f) \cdot \dist_{\mcM}(x, \gamma_{\theta_i^*})
\;\leq\; C_1 \epsilon_{\mathrm{tile}}.
\]
Combining,
\[
\sup_{x \in \mcV_i} \bigl| \Psi^f(x) - \Psi^{f^\dagger}_i(x) \bigr|
\;\leq\; C(\rho_0) \eta_{\mathrm{model}} + C_1 \epsilon_{\mathrm{tile}}.
\]

\textbf{Stage 2: $f^\dagger \to \hat{f}$.}
The corrected approximator $\hat{f} = \tilde{f} + \sum_j \alpha_j^* \hat{\mathbf{\Phi}}_j$ satisfies $\|\hat{f} - f^\dagger\|_{C^1} \leq \eta_{\mathrm{base}} + N \|\boldsymbol{\alpha}^*\| B_\Phi$ by Lemma~\ref{lem:bump_realizability} and Step~3.
Applying~\eqref{eq:phase_lipschitz} once more on each basin,
\[
\sup_{x \in \mcV_i} \bigl| \Psi^{f^\dagger}_i(x) - \Psi^{\hat{f}}_i(x) \bigr|
\;\leq\; C(\rho_0) \cdot (\eta_{\mathrm{base}} + N\|\boldsymbol{\alpha}^*\| B_\Phi).
\]

Triangle inequality gives the claim with $C_\Psi = \max\bigl(C(\rho_0)(1 + \kappa N L_T B_\Phi),\, C_1\bigr)$, recalling $\|\boldsymbol{\alpha}^*\| \leq \kappa L_T \eta_{\mathrm{base}}$ from Proposition~\ref{prop:correction_existence}.
\end{proof}

The proof of Theorem~\ref{thm:uap_ca} splits into two cases.

\begin{proof}[Proof of Case 1: Manifold of Fixed Points]
Let $\mathcal{M}$ be a NHIM consisting entirely of fixed points, i.e., $f(x) = 0$ for all $x \in \mathcal{M}$.

\textbf{Step 1: Target Selection and Constructive Discretization.} 
The continuum of fixed points is structurally unstable. 
While the Kupka-Smale theorem guarantees the density of vector fields with hyperbolic critical elements and transversal invariant manifolds, it does not explicitly guarantee that a generic perturbation will yield a spatially dense $\epsilon$-net of attractors. Therefore, we explicitly construct a structurally stable proxy field $f^{\dagger}$.
For any compact continuous attractor $\mcM$, there exists a Morse function $V: \mcX \to \mathbb{R}$ whose local minima on $\mcM$ form an $\varepsilon/6$-net. 

We construct the proxy field by applying a global gradient perturbation: $f^{\dagger}(x) = f(x) - \eta_{\mathrm{tile}} \nabla V(x)$. Generating a dense $\varepsilon/6$-net requires $V$ to have high-frequency spatial oscillations, which increases the $C^1$ norm of $\nabla V$ proportionally to $1/\varepsilon^2$. 
However, because the spatial configuration is fixed by $\varepsilon$ independently of the scaling amplitude $\eta_{\mathrm{tile}}$, we can choose $\eta_{\mathrm{tile}} > 0$ to be arbitrarily small. This ensures that the global perturbation remains strictly bounded by $\|\eta_{\mathrm{tile}} \nabla V\|_{C^1} < \eta_{model}$.

Because this perturbation is globally $C^1$-small, it preserves the macroscopic dynamics: the transient approach to the manifold remains largely unperturbed, and global basin boundaries (separatrices) are structurally maintained. 
Thus, for any $\eta_{model}>0$ and $\varepsilon>0$, we construct $f^{\dagger}$ such that:
\begin{enumerate}
\item \textbf{Closeness:} $\|f-f^{\dagger}\|_{C^1} < \eta_{model}$.
\item \textbf{Persistence:} The manifold $\mcM$ persists as a perturbed invariant manifold $\mcM^{\dagger}$.
\item \textbf{Discretization:} The flow on $\mcM^{\dagger}$ possesses a finite set of hyperbolic fixed points $\{\hat{x}_i\}_{i=1}^N$ forming an $\varepsilon/6$-net of $\mcM$. 
The remaining $\epsilon/6$ in our spatial budget safely accounts for the structural shift of these equilibria when the final approximation $\tilde{f} \in \hat{\mcF}$ targets $f^{\dagger}$.\end{enumerate}

\textbf{Step 2: Robust Approximation.}
We target the stable proxy $f^\dagger$. By the $C^1$ UAP, choose $\tilde{f} \in \hat{\mathcal{F}}$ such that $\|\tilde{f} - f^\dagger\|_{C^1} < \eta_{\mathrm{base}}$.
Since the fixed points $\{\hat{x}_i\}$ of $f^\dagger$ are hyperbolic, they are structurally stable. 
By the Implicit Function Theorem, $\tilde{f}$ possesses exactly $N$ fixed points $\{\tilde{x}_i\}$ close to the proxy points.

\textbf{Step 3: Global Error Decomposition.}
For any initial condition $x_0$ in the basin of attraction, the trajectory error $\|\varphi(t) - \hat{\varphi}(t)\|$ is bounded by three components, allocated a budget of $\varepsilon/3$ each:\begin{enumerate}
\item \textbf{Manifold Error (Transverse):}
The physical distance between the true manifold $\mathcal{M}$ and the approximated manifold $\hat{\mathcal{M}}$. 
Controlled by Fenichel persistence (\ref{thm:fenichel}):
$$d_H(\mathcal{M}, \hat{\mathcal{M}}) \le L_{\mathrm{NHIM}} (\eta_{\mathrm{model}} + \eta_{\mathrm{base}}) < \frac{\varepsilon}{3}.$$

\item \textbf{Tiling Error (Longitudinal):} The error of approximating a stationary point $p \in \mcM$ with the nearest discrete attractor $\tilde{x}_i \in \hat{\mcM}$. 
Since the proxy points form an $\varepsilon/6$-net, and the approximation shift is bounded by $O(\eta_{base}) < \epsilon/6$, any trajectory on the manifold will converge to a sink within distance $\epsilon/3$:
$$ \lim_{t \to \infty} \|\varphi(t) - \hat{\varphi}(t)\| \le \max_i \text{diam}(\text{basin}(\tilde{x}_i) \cap \mcM) < \frac{\varepsilon}{3}. $$

\item \textbf{Transient and asymptotic trajectory error (Gr\"onwall + variation of constants).}
This component controls the deviation accrued while the trajectories \emph{approach} the invariant manifold transversally; the subsequent \emph{in-manifold} drift from the landing point to the nearest tile equilibrium is not treated here, since it is already charged to the Longitudinal component above, where it is bounded statically by the tile spacing ($< \varepsilon/3$) with no exponential factor. Separating the two is essential: the transverse approach is fast and amplitude-independent, whereas the in-manifold settling is slow, its rate being the $O(\eta_{\mathrm{tile}})$ tile drift.

\textit{Transient phase $[0, T^*_\perp]$.}
Let $T^*_\perp$ be the uniform time for trajectories from the valid set to enter a fixed-radius Fenichel neighborhood of $\hat{\mcM}$. This transverse approach is governed by the normal-hyperbolicity contraction rate $\lambda_\perp > 0$ of $\mcM$, which persists under the $O(\eta_{\mathrm{model}} + \eta_{\mathrm{base}})$ perturbation (Theorem~\ref{thm:fenichel}) and is bounded below by a constant depending only on $\mcM$, \emph{independent of the tiling amplitude $\eta_{\mathrm{tile}}$ and the resolution $\varepsilon$}. Hence $T^*_\perp$ is finite and $O(1)$, uniformly in $\eta_{\mathrm{tile}}$ and $\varepsilon$. On this phase, Gr\"onwall gives
\[
\|\varphi(t, x_0) - \hat{\varphi}(t, x_0)\| \;\leq\; \|f - \hat{f}\|_{C^0} \cdot \frac{e^{Lt} - 1}{L}, \qquad t \in [0, T^*_\perp],
\]
with $L = \Lip(f)$.

\textit{Asymptotic phase $[T^*_\perp, \infty)$.}
After $T^*_\perp$ the trajectory lies in the Fenichel neighborhood, where the normal contraction acts at the rate $\lambda_\perp$ (the normal-hyperbolicity rate, \emph{not} the $O(\eta_{\mathrm{tile}})$ in-manifold relaxation rate of an individual tile).
By the variation-of-constants formula~\citep{vanhandel2007filtering},
\[
\|\varphi(t, x_0) - \hat{\varphi}(t, x_0)\|
\;\leq\; C\|\varphi(T^*_\perp) - \hat{\varphi}(T^*_\perp)\| \cdot e^{-\lambda_\perp(t - T^*_\perp)} + \int_{T^*_\perp}^t C\, e^{-\lambda_\perp(t-s)} \|f - \hat{f}\|_{C^0}\, ds,
\]
where $C$ bounds the flow Jacobian in the linearization regime.
The second term is uniformly bounded in $t$ by $C \|f - \hat{f}\|_{C^0}/\lambda_\perp$, so
\[
\sup_{t \geq T^*_\perp} \|\varphi(t, x_0) - \hat{\varphi}(t, x_0)\| \;\leq\; C\|f - \hat{f}\|_{C^0} \cdot \frac{e^{L T^*_\perp} - 1}{L} + \frac{C \|f - \hat{f}\|_{C^0}}{\lambda_\perp}.
\]
Since $T^*_\perp$ and $\lambda_\perp$ are fixed $O(1)$ constants independent of $\eta_{\mathrm{tile}}$ and $\varepsilon$, and $\|f - \hat{f}\|_{C^0} \leq \eta_{\mathrm{model}} + \eta_{\mathrm{base}}$, this transverse contribution satisfies
\[
\sup_{t \geq 0} \|\varphi(t, x_0) - \hat{\varphi}(t, x_0)\|_{\perp}
\;\leq\; C(\eta_{\mathrm{model}} + \eta_{\mathrm{base}}) \cdot \left( \frac{e^{L T^*_\perp} - 1}{L} + \frac{1}{\lambda_\perp} \right) \;<\; \frac{\varepsilon}{3},
\]
attainable at every resolution $\varepsilon$ once $\eta_{\mathrm{model}} + \eta_{\mathrm{base}}$ is small enough, since the bracketed factor is now a fixed constant rather than a quantity that grows (as $e^{L/\eta_{\mathrm{tile}}}$) as the tiling is refined. Combining this transverse bound with the Manifold and Longitudinal components ($\varepsilon/3$ each) by the triangle inequality yields $\sup_{t \geq 0} \|\varphi - \hat{\varphi}\| < \varepsilon$.
\end{enumerate}\textbf{Conclusion.}
By choosing $\eta_{\mathrm{model}}$ and $\eta_{\mathrm{base}}$ sufficiently small, the total uniform error is bounded by $\varepsilon$.
The basin mismatch volume is bounded by $\delta$ due to the structural stability of the domain boundaries (separatrices) of the hyperbolic proxy $f^\dagger$.
\end{proof}

\begin{proof}[Proof of Case 2: Isochronous Limit Cycle Manifold]
Let $\mathcal{M}$ be a Normally Hyperbolic Invariant Manifold (NHIM) of $f$ foliated by a continuum of periodic orbits $\{\gamma_\theta\}_{\theta \in S^1}$, all having a common period $T$.

\textbf{Step 1: Target Selection and Constructive Discretization.} 
The continuum of orbits in $f$ is structurally unstable and cannot be approximated robustly directly. 
While the Kupka-Smale theorem guarantees that periodic hyperbolicity is dense, we must explicitly construct a perturbation to ensure the resulting isolated limit cycles form a dense $\epsilon$-net without destroying the global basin structure. 
Therefore, we explicitly construct a structurally stable proxy field $f^{\dagger}$. 

For any compact normally hyperbolic manifold $\mcM$ foliated by isochronous limit cycles, we introduce a weak transverse drift to break the continuum into discrete orbits. 
Let $V: \mcX \to \mathbb{R}$ be a smooth function that oscillates transversally to the limit cycles such that its local transverse minima form an $\epsilon/16$-net of $\mcM$. 
We construct the proxy field by applying a global perturbation: $f^{\dagger}(x) = f(x) - \eta_{\mathrm{tile}} \nabla V(x)$. 
Generating a dense $\epsilon/16$-net requires $V$ to have high-frequency spatial oscillations, scaling the $C^1$ norm of $\nabla V$ proportionally to $1/\epsilon^2$. 
Because this spatial configuration is fixed by $\epsilon$ independently of the amplitude $\eta_{\mathrm{tile}}$, we can choose $\eta_{\mathrm{tile}} > 0$ arbitrarily small. 
This ensures the global perturbation remains strictly bounded by $\|\eta_{\mathrm{tile}} \nabla V\|_{C^1} < \eta_{model}$, preserving the macroscopic transient approach and global separatrices.

Thus, for any $\eta_{model}>0$ and $\epsilon>0$, we construct $f^{\dagger}$ such that:
\begin{enumerate}
    \item \textbf{Closeness:} $\|f - f^{\dagger}\|_{C^1} < \eta_{model}$.
    \item \textbf{Persistence:} The manifold $\mcM$ persists as $\mcM^{\dagger}$.
    \item \textbf{Discretization:} The flow on $\mcM^{\dagger}$ possesses a finite number of hyperbolic limit cycles $\{\gamma_i^{\dagger}\}_{i=1}^N$ forming an $\epsilon/16$-net of $\mcM$. The remaining $\epsilon/16$ in our spatial budget safely accounts for the structural geometric shift of these cycles when the final approximation $\tilde{f} \in \hat{\mcF}$ targets $f^{\dagger}$.
\end{enumerate}

\textbf{Step 2: Robust Base Approximation.}
We treat the stable proxy $f^\dagger$ as our target.
By the $C^1$ UAP, there exists $\tilde{f} \in \hat{\mathcal{F}}$ such that $\|\tilde{f} - f^\dagger\|_{C^1} < \eta_{\mathrm{base}}$.
Since the limit cycles $\{\gamma^\dagger_i\}$ of $f^\dagger$ are hyperbolic, they are structurally stable. By Lemma~\ref{lem:lc_stability}, for sufficiently small $\eta_{\mathrm{base}}$, $\tilde{f}$ is guaranteed to possess exactly $N$ limit cycles $\{\tilde{\gamma}_i\}$ close to the proxy cycles.

\textbf{Step 3: Additive Period Correction.}
We apply the additive correction strategy (Theorem~\ref{thm:uap_lc}) to $\tilde{f}$.
Since the original system $f$ is isochronous, we define the target period for all $N$ discrete cycles to be exactly $T$.
\[ \hat{f} = \tilde{f} + \sum_{i=1}^N \alpha_i^* \hat{\mathbf{\Phi}}_i. \]
This creates a phase-locked skeleton on the approximated manifold $\hat{\mathcal{M}}$, where every cycle has period $T$.

\textbf{Step 4: Global Error Decomposition.}
Define the \textbf{good set}
\[
G \coloneqq \mcX \setminus \bigl( S \cup N_{\rho_0}(\partial \mcB) \bigr),
\]
where $S$ is the union of separatrices of $\hat{f}$, $\partial \mcB$ denotes the basin boundaries between adjacent tiles $\hat{\gamma}_i, \hat{\gamma}_j$, and $N_{\rho_0}$ is the $\rho_0$-tubular neighborhood from Lemma~\ref{lem:phase_coherence}.
Since separatrices have codimension $\geq 1$, $\mu(\mcX \setminus G) \to 0$ as $\rho_0 \to 0$, and we choose $\rho_0$ so that $\mu(\mcX \setminus G) < \delta$.

For $x_0 \in G$, let $\theta = \theta(x_0) \in \Theta$ be the foliation parameter of $x_0$ in the target and $i = i(x_0)$ the basin index in the approximator.
The trajectory error decomposes as
\begin{align*}
\| \varphi(t, x_0) - \hat{\varphi}(t, x_0) \|
&\leq \underbrace{\bigl\| \varphi(t, x_0) - \gamma_\theta(t + \Psi^f(x_0)) \bigr\|}_{\text{(A) target attraction}}
  + \underbrace{\bigl\| \gamma_\theta(t + \Psi^f(x_0)) - \hat{\gamma}_i(t + \Psi^{\hat{f}}_i(x_0)) \bigr\|}_{\text{(B) cycle--cycle}} \\
&\quad + \underbrace{\bigl\| \hat{\gamma}_i(t + \Psi^{\hat{f}}_i(x_0)) - \hat{\varphi}(t, x_0) \bigr\|}_{\text{(C) approximator attraction}}.
\end{align*}
We allocate $\varepsilon/4$ to each of four contributions: the attraction terms (A) and (C) jointly; the geometric part of (B); the phase part of (B); and the transverse manifold error.

\textbf{(A) and (C): Exponential attraction to respective cycles.}
By normal hyperbolicity of $\mcM$ for $f$ and hyperbolicity of $\hat{\gamma}_i$ for $\hat{f}$, both attraction errors decay exponentially with rates $\lambda_f, \lambda_{\hat{f}} > 0$:
$\text{(A)} \leq K_f e^{-\lambda_f t}$ and $\text{(C)} \leq K_{\hat{f}} e^{-\lambda_{\hat{f}} t}$.
For $t \geq T_{\mathrm{trans}} \coloneqq \max(K_f, K_{\hat{f}}) / \min(\lambda_f, \lambda_{\hat{f}}) \cdot \log(8/\varepsilon)$, both terms are $\leq \varepsilon/16$.
For $t < T_{\mathrm{trans}}$, the standard Gr\"onwall bound yields $\|\varphi - \hat{\varphi}\| \leq (e^{L T_{\mathrm{trans}}} - 1) \cdot \|f - \hat{f}\|_{C^0} / L$, which can be made $\leq \varepsilon/4$ by choosing $\eta_{\mathrm{base}} + \eta_{\mathrm{model}} + \|\boldsymbol{\alpha}^*\| B_\Phi$ small.

\textbf{(B-geom) Geometric tile-to-target distance.}
By Lemma~\ref{lem:bump_realizability}, Lemma~\ref{lem:geo_stability}, and the $\epsilon_{\mathrm{tile}}$-net property,
\[
d_H(\gamma_\theta, \hat{\gamma}_i)
\;\leq\; \underbrace{d_H(\gamma_\theta, \gamma_{\theta_i^*})}_{\leq \,\epsilon_{\mathrm{tile}}}
+ \underbrace{d_H(\gamma_{\theta_i^*}, \gamma_i^\dagger)}_{\leq\, L_F \eta_{\mathrm{model}}}
+ \underbrace{d_H(\gamma_i^\dagger, \tilde{\gamma}_i)}_{\leq\, L_H \eta_{\mathrm{base}}}
+ \underbrace{d_H(\tilde{\gamma}_i, \hat{\gamma}_i)}_{\leq\, L_H \|\boldsymbol{\alpha}^*\| B_\Phi}.
\]
Choosing the budget so the right-hand side is $\leq \varepsilon/4$ controls the geometric part of (B).

\textbf{(B-phase) Phase mismatch.}
Equal periods alone do not imply that $\gamma_\theta(t + \Psi^f(x_0))$ and $\hat{\gamma}_i(t + \Psi^{\hat{f}}_i(x_0))$ are at corresponding positions on their respective cycles.
The phase difference $\Delta\Psi(x_0) \coloneqq \Psi^f(x_0) - \Psi^{\hat{f}}_i(x_0)$ contributes
\[
\bigl\| \gamma_\theta(t + \Psi^f(x_0)) - \hat{\gamma}_i(t + \Psi^{\hat{f}}_i(x_0)) \bigr\|
\;\leq\; d_H(\gamma_\theta, \hat{\gamma}_i) + L_\gamma \cdot |\Delta\Psi(x_0)|,
\]
where $L_\gamma = \sup \|\dot{\gamma}\|$ on the cycles.
By Lemma~\ref{lem:phase_coherence}, on the good set $G$,
\[
|\Delta\Psi(x_0)|
\;\leq\; C_\Psi \bigl( \|\hat{f} - f\|_{C^1} + \epsilon_{\mathrm{tile}} \bigr)
\;\leq\; C_\Psi \bigl( \eta_{\mathrm{model}} + \eta_{\mathrm{base}} + \|\boldsymbol{\alpha}^*\| B_\Phi + \epsilon_{\mathrm{tile}} \bigr).
\]
Choose the budget so $L_\gamma C_\Psi \cdot (\cdots) \leq \varepsilon/4$.

\textbf{Transverse manifold error.}
By Fenichel persistence (Theorem~\ref{thm:fenichel}), $d_H(\mcM, \hat{\mcM}) \leq L_{\mathrm{NHIM}} \cdot \|\hat{f} - f\|_{C^1} \leq \varepsilon/4$ for sufficiently small $\eta_{\mathrm{model}} + \eta_{\mathrm{base}}$.

\textbf{Total.}
Summing the four budgets, $\sup_{t \geq 0} \|\varphi(t, x_0) - \hat{\varphi}(t, x_0)\| < \varepsilon$ for every $x_0 \in G$, with $\mu(\mcX \setminus G) < \delta$.
This is precisely $\|\varphi - \hat{\varphi}\|_\varepsilon < \delta$.
\end{proof}

\begin{remark}[Consistency with Constraint~C of Theorem~\ref{thm:uap_lc}]
The phase-coherence term in (B-phase) and the Gr\"onwall amplification in (A)+(C) above are precisely the two contributions formalized in Constraint~C of the additive-correction proof of Theorem~\ref{thm:uap_lc} (Eq.~\ref{eq:constraint_C}).
The bound~\eqref{eq:constraint_C} on $\eta_{\mathrm{base}}$ derived there propagates directly into the budget cascade for Case~2, so no separate budget needs to be tracked here: a choice of $\eta_{\mathrm{base}}, \eta_{\mathrm{model}}, \epsilon_{\mathrm{tile}}$ satisfying~\eqref{eq:constraint_C} (with $\rho_0$ chosen small enough that $\mu(\mcX \setminus G) < \delta$) closes both proofs simultaneously.
\end{remark}

%==============================================================================
\section{Numerical Experiments}\label{app:experiments}

The following experiments illustrate the three theorems and the B/P/D taxonomy of Section~\ref{sec:intro} on trained Neural ODEs and on real task-trained recurrent networks. They demonstrate the mechanisms used in the proofs; they are not part of them. Unless noted otherwise, we fit a Neural ODE $\dot x = \hat f_\theta(x)$ with a smooth-activation MLP field (hence $C^1$) by regressing the vector field in $C^1$, minimizing $\mathrm{MSE}(\hat f, f) + \mathrm{MSE}(\nabla \hat f, \nabla f)$, integrate both flows with a fixed-step solver, and report the empirical $\varepsilon$-$\delta$ profile
\begin{equation}\label{eq:emp_eps_delta}
\hat\delta(\varepsilon) \;=\; \frac{1}{M}\,\#\Big\{\, x_0^{(i)} \;:\; \sup_{t}\, \big\|\varphi(t, x_0^{(i)}) - \hat\varphi(t, x_0^{(i)})\big\| \geq \varepsilon \,\Big\},
\end{equation}
the sample form of Definition~\ref{def:eps_delta_close} over $M$ sampled initial conditions. All Neural ODE fits use double precision; the task-trained flip-flop network is evaluated in its native single precision.

\subsection{The three failure modes on planar targets}\label{app:exp_bpd}

We fit one Neural ODE per error mode (Figure~\ref{fig:exp_bpd}).

\textbf{B-type (basin): damped Duffing.} For $\dot x = y,\ \dot y = x - x^3 - \tfrac{1}{2} y$ (two stable spirals at $(\pm 1, 0)$, a saddle at the origin, a curved separatrix), a width-$64$ field with $\|f - \hat f\|_{C^0} = 9.4 \times 10^{-3}$ reproduces both basins with $99.96\%$ agreement ($3$ of $8100$ initial conditions switch basin). The trajectory sup-error is of order $10^{-3}$ throughout both basins and rises to the basin diameter $D \approx 1.8$ only on a thin collar around the separatrix, precisely the localization of Lemma~\ref{lem:basin_error} and Remark~\ref{rem:basin_scaling}. At the scale of the wells ($|x| = 1$), $\hat\delta(0.05) = 0.014$, $\hat\delta(0.1) = 0.007$, $\hat\delta(0.2) = 0.004$, falling to the pure basin-flip fraction $3/8100 = 4 \times 10^{-4}$ as $\varepsilon$ approaches the basin diameter.

\textbf{P-type (phase): Van der Pol.} For $\dot x = y,\ \dot y = \mu(1 - x^2) y - x$ with $\mu = 1$~\citep{vanderpol1926relaxation}, the learned cycle matches the true one in shape, but a residual period error $|T - \hat T| = 1.25 \times 10^{-3}$ drives a phase-drift beat: the running sup-error grows from $0.038$ at $t = 5$ to the full cycle diameter $5.66$ near $t \approx 6.4 \times 10^{3}$. The detection horizon predicted by the $1/|T - \hat T|$ law, $t^\ast = 2657$, matches the empirical worst-case time $2641$. A $45$-cycle validation window ($T = 300$) detects essentially nothing, so finite-horizon validation is blind to P-type error; this is the gap that exact period matching (Theorem~\ref{thm:uap_lc}) removes (Figure~\ref{fig:exp_vdp}).

\begin{figure}[t]
\centering
\includegraphics[width=\textwidth]{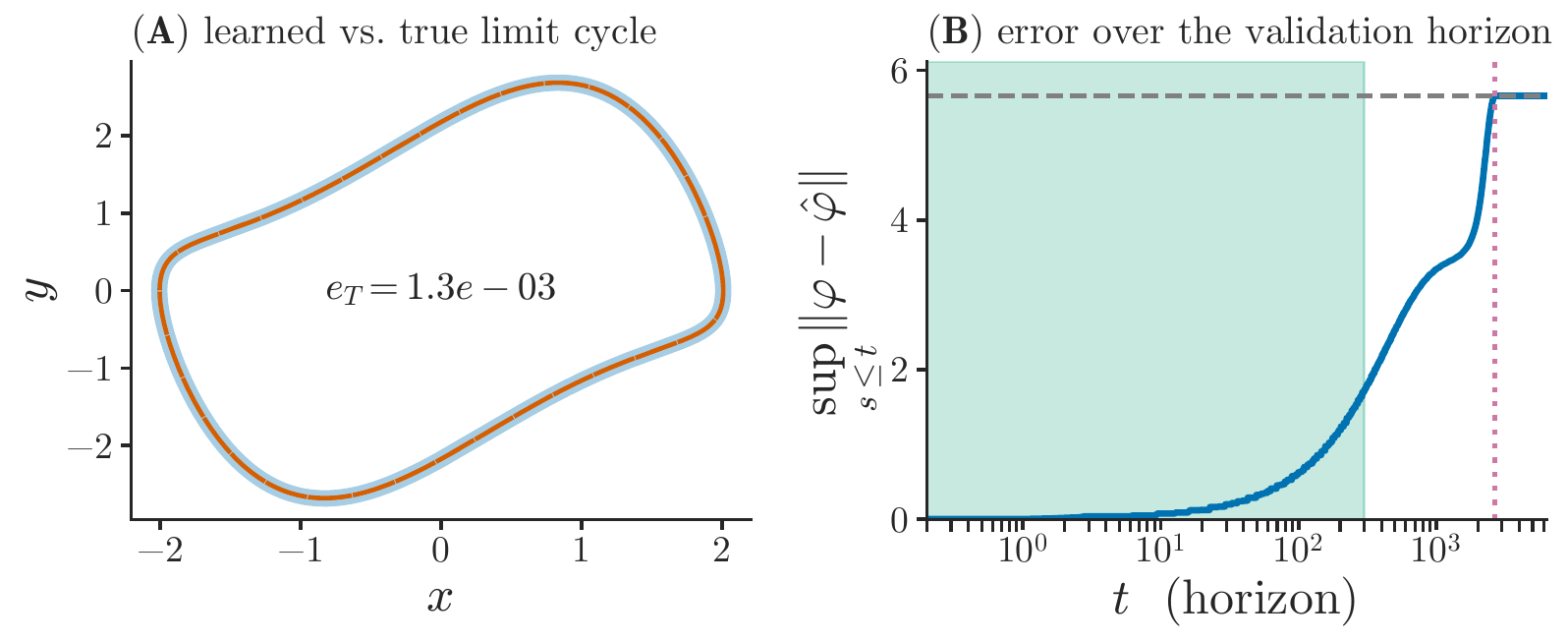}
\caption{P-type on the Van der Pol oscillator. Left: the learned limit cycle (dashed) coincides in shape with the true one (solid band) almost exactly; the period error is $|T - \hat T| = 1.25 \times 10^{-3}$. Right: the running supremum of the trajectory error over a horizon of length $t$ climbs to the cycle diameter $D$ (dashed) only near the detection horizon $t^\ast \approx 2657$ (dotted); a $45$-cycle validation window (shaded, $t \le 300$) captures only a fraction of it, so finite-horizon validation underestimates the drift.}
\label{fig:exp_vdp}
\end{figure}

\textbf{D-type (discretization): ring attractor.} For $\dot z = (1 - \|z\|^2) z$ (the unit circle is a continuum of equilibria), the ring persists (normal contraction $\approx -2 \times 10^{-4}$, so the flow is $\varepsilon$-$\delta$ close), but the continuum is discretized into $10$ isolated fixed points ($5$ sinks alternating with $5$ saddles), and the non-wandering set is never homeomorphic to $S^1$. The time-averaged $L^2$ error stays small ($8 \times 10^{-2}$) despite the changed topology, which is the tiling regime of Theorem~\ref{thm:uap_ca} (Figure~\ref{fig:exp_ca}A,B, top row). Figure~\ref{fig:exp_bpd} in the main text shows the B-type fit and its $\delta(\varepsilon)$ curve.

\textbf{Decomposition: two coaxial limit cycles.} To separate the phase mode from the basin mode, we fit a Neural ODE to a system with two stable limit cycles (radii $1$ and $2$) separated by an unstable cycle at $r = 1.5$ (Figure~\ref{fig:exp_2lc}). Without period correction, phase drift fills both basins (P-type), and $\delta(\varepsilon) \approx 1$ up to $\varepsilon \approx 0.5$. The exact period correction of Theorem~\ref{thm:uap_lc} collapses the bad set to a thin B-type collar at the unstable cycle $r = 1.5$, and $\delta(\varepsilon)$ drops to $0.026$, $0.015$, $0.004$ at $\varepsilon = 0.05, 0.10, 0.30$: phase drift accounts for $\approx 99\%$ of the $\varepsilon$-exceeding initial conditions, the basin collar the rest.

\begin{figure}[t]
\centering
\includegraphics[width=0.92\textwidth]{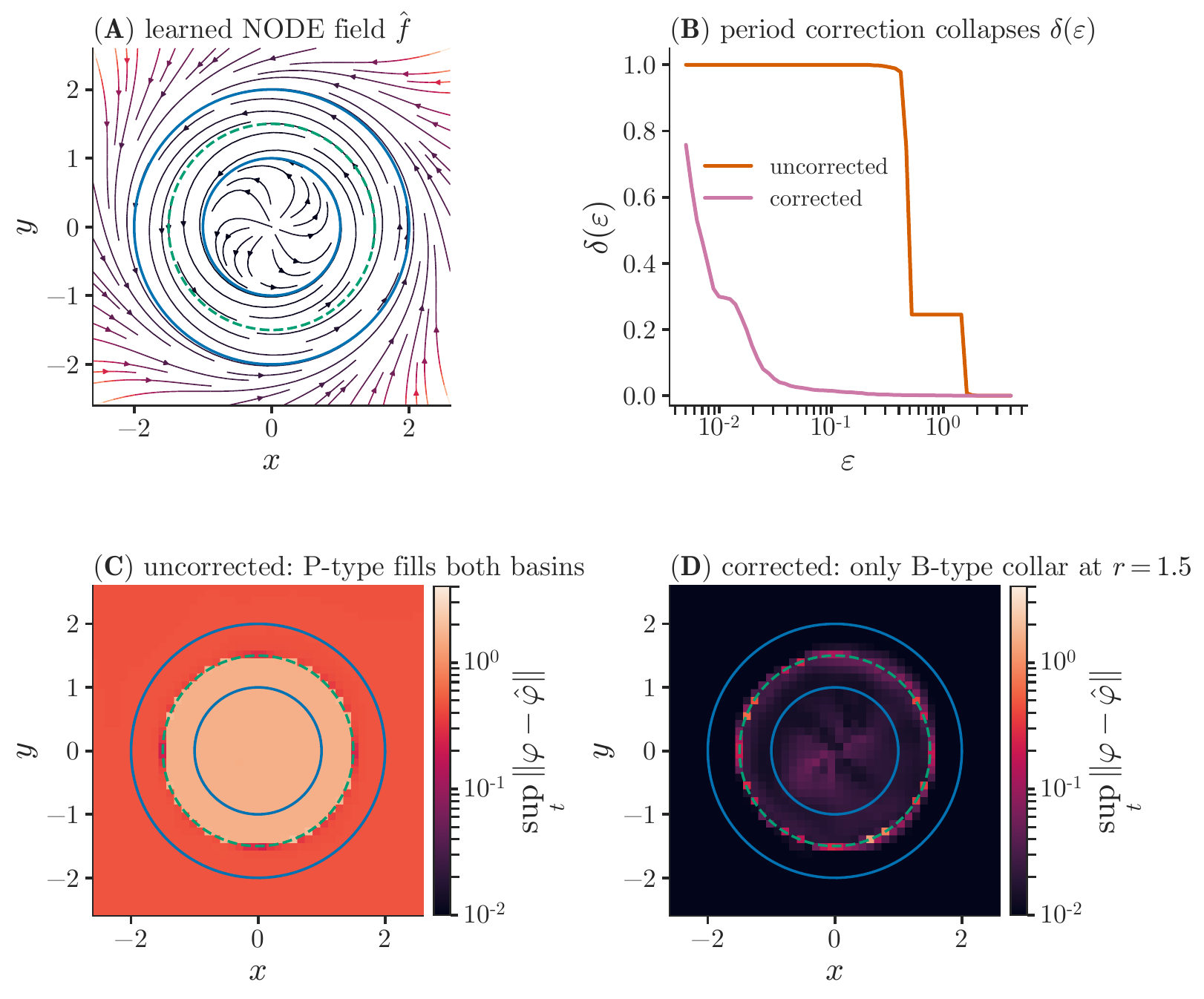}
\caption{Isolating the phase mode with two coaxial limit cycles (stable at $r = 1, 2$; unstable at $r = 1.5$).
\textbf{(A)} The learned Neural ODE field recovers both cycles.
\textbf{(B)} The exact period correction collapses $\delta(\varepsilon)$ from $\approx 1$ (uncorrected, phase-dominated) to the basin floor (corrected).
\textbf{(C)} Uncorrected, phase drift (P-type) fills both basins.
\textbf{(D)} Corrected, only a thin B-type collar remains at the unstable cycle $r = 1.5$.}
\label{fig:exp_2lc}
\end{figure}

\subsection{$\mathcal{F}_{\mathrm{FP}}$ holds across $C^1$-dense architectures}\label{app:exp_fp}

Theorem~\ref{thm:uap_fp} assumes only the $C^1$ universal approximation property (Definition~\ref{def:hypothesis_class}), and no specific architecture. We verify this on a three-dimensional target with $2^3 = 8$ hyperbolic fixed points, $\dot x_i = x_i - x_i^3$ for $i = 1, 2, 3$, whose attractors are the corners of the cube (Figure~\ref{fig:exp_fp}). Three $C^1$-dense hypothesis classes, a smooth-activation MLP, a $C^1$ RNN field $\dot x = -x + W \tanh(V x + b)$, and a radial-basis-function field, each recover all eight attractors with basin agreement above $99.8\%$ and $\hat\delta(0.3) \le 6 \times 10^{-3}$. The guarantee follows from $C^1$ density and holds for every such parameterization, as Theorem~\ref{thm:uap_fp} states.

\begin{figure}[t]
\centering
\includegraphics[width=\textwidth]{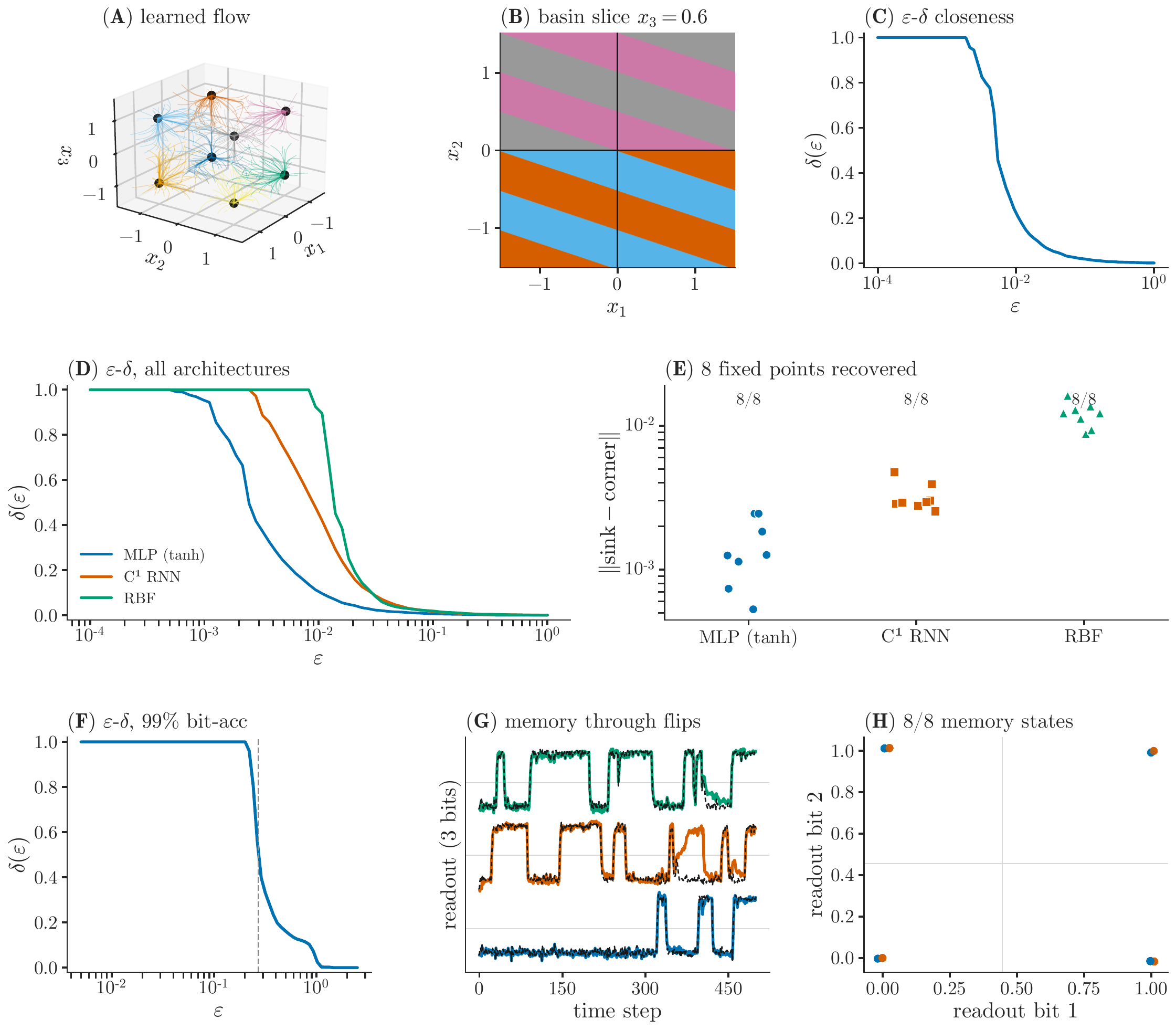}
\caption{$\mathcal{F}_{\mathrm{FP}}$ across architectures and on a real trained network.
\textbf{(A--C)} A $3$D eight-fixed-point target $\dot x_i = x_i - x_i^3$: trajectories from many initial conditions settle onto the eight cube-corner attractors, colored by basin (A); a basin slice (B); and the $\varepsilon$-$\delta$ curve (C).
\textbf{(D, E)} Three $C^1$-dense hypothesis classes, a smooth-activation MLP, a $C^1$ RNN field, and a radial-basis-function field, each reach $\varepsilon$-$\delta$ closeness (D) and recover all $8$ fixed points (E), with $\hat\delta(0.3) \le 6 \times 10^{-3}$.
\textbf{(F--H)} A Neural ODE $\varepsilon$-$\delta$ approximates a gated recurrent network trained on the $3$-bit flip-flop task, scored on the network's own readout: the $\varepsilon$-$\delta$ curve at $99.2\%$ per-bit accuracy (F); the three bits tracked through many memory flips, network (solid) versus Neural ODE (dashed) (G); and the Neural ODE's autonomous field, with stable fixed points at all $8$ memory states (H).}
\label{fig:exp_fp}
\end{figure}

\subsection{Real task-trained networks lie in the target classes}\label{app:exp_real}

The theorems concern targets in $\mathcal{F}_{\mathrm{FP}}$, $\mathcal{F}_{\mathrm{LC}}$, and $\mathcal{F}_{\mathrm{CA}}$. We now show that networks trained on standard tasks fall into these classes and are $\varepsilon$-$\delta$ approximable.

\textbf{Working memory as $\mathcal{F}_{\mathrm{FP}}$: the $3$-bit flip-flop.} We take a gated recurrent network trained on the $3$-bit flip-flop task~\citep{sussillo2013blackbox} from the Computation-through-Dynamics toolkit~\citep{versteeg2025computation, sedler2023expressive} (latent dimension $64$). Its linear readout reconstructs the task output with $R^2 = 0.9999$. FixedPointFinder~\citep{golub2018fixedpointfinder} with zero input recovers stable fixed points at all $8$ memory states (the cube corners), so the trained network is a member of $\mathcal{F}_{\mathrm{FP}}$, a real instance of the multistable target class (Figure~\ref{fig:exp_bpd}C, main text; the linear readout has $R^2 = 0.9999$).

\textbf{A Neural ODE approximates the trained network.} Writing $F$ for the network's gated recurrent cell, the discrete update $h_{t+1} = F(u_t, h_t)$ is exactly forward Euler (step $\Delta t = 1$) of the ODE $\dot x = g^\ast(x, u)$ with $g^\ast(x, u) = F(u, x) - x$; the trained network is thus the $\Delta t = 1$ Euler discretization of that ODE, not the ODE itself. Setting $u = 0$ gives the autonomous field, whose eight hyperbolic fixed points are the memory states, and this is the object $\mathcal{F}_{\mathrm{FP}}$ addresses. We fit a separate MLP field $\hat f_\theta \approx g^\ast$ (a different architecture, carrying its own nonzero approximation error) to the residual field over a tube around the trajectories, fine-tuned through the rollout, and ask whether it $\varepsilon$-$\delta$ approximates the network. Autonomously ($u = 0$), the approximation preserves the whole fixed-point structure over the infinite horizon: rolling out the network and the Neural ODE from the same states, they settle on the same memory state for $99.8\%$ of initial conditions with $\hat\delta(0.3) = 6 \times 10^{-3}$, unchanged out to $5000$ steps, and all eight memory states are inherited (Figure~\ref{fig:exp_fp}H). Under the task inputs the Neural ODE tracks the network's own (imperfect) readout at $99.2\%$ per-bit accuracy (Figure~\ref{fig:exp_fp}F,G); there the trajectory sup-error is larger ($\hat\delta(0.3) = 0.32$, median $0.25$) because a bit flip is a nearly discontinuous one-step transition, so a continuous flow has a brief transient mismatch there. This is the discrete-to-continuous analogue of the fast-transient region, where $\varepsilon$-$\delta$ closeness holds even though $C^1$ closeness does not.

\textbf{Continuous attractors as $\mathcal{F}_{\mathrm{CA}}$: head-direction integration.} A $64$-unit rectified-tanh recurrent network trained to integrate angular velocity on $S^1$~\citep{Sagodi2024a} holds heading on a ring attractor, the canonical continuous-attractor model of the head-direction circuit, whose ring topology has been recovered from neural population recordings~\citep{chaudhuri2019attractor}. A finite network cannot hold a true continuum; it instead realizes a structurally stable approximation of it~\citep{Sagodi2024a}, a normally attracting ring carrying isolated hyperbolic fixed points. These are the zeros of the along-ring drift (the autonomous field projected onto the ring tangent), which as the sign changes of a scalar on $S^1$ alternate sinks and saddles in equal number; the network we use tiles the ring into $12$ sinks alternating with $12$ saddles, all clearly hyperbolic (median leading eigenvalue $-0.20$), and a dense rollout of held headings converges onto exactly these discrete sinks. Because the fixed points are hyperbolic they are structurally stable, so any $C^1$ field sufficiently close to the network's must inherit them. We verify this by fitting an independent Neural ODE $\hat g(x) = -x + \mathrm{MLP}(x)$ (a continuous-time recurrent form, whose $-x$ leak guarantees global stability) to the network's autonomous field in a neighbourhood of the ring: its fixed points coincide with the network's ($22$ of the $24$ within $0.15$~rad; Figure~\ref{fig:exp_ca}C), and its held-heading drift $\delta(\varepsilon)$ matches the network's (Figure~\ref{fig:exp_ca}D). This is exactly the setting of Theorem~\ref{thm:uap_ca}: a finite model discretizes the continuum into isolated headings, the D-type mechanism of head-direction drift, and the discrete fixed-point structure transfers to the approximating flow.

\begin{figure}[t]
\centering
\includegraphics[width=\textwidth]{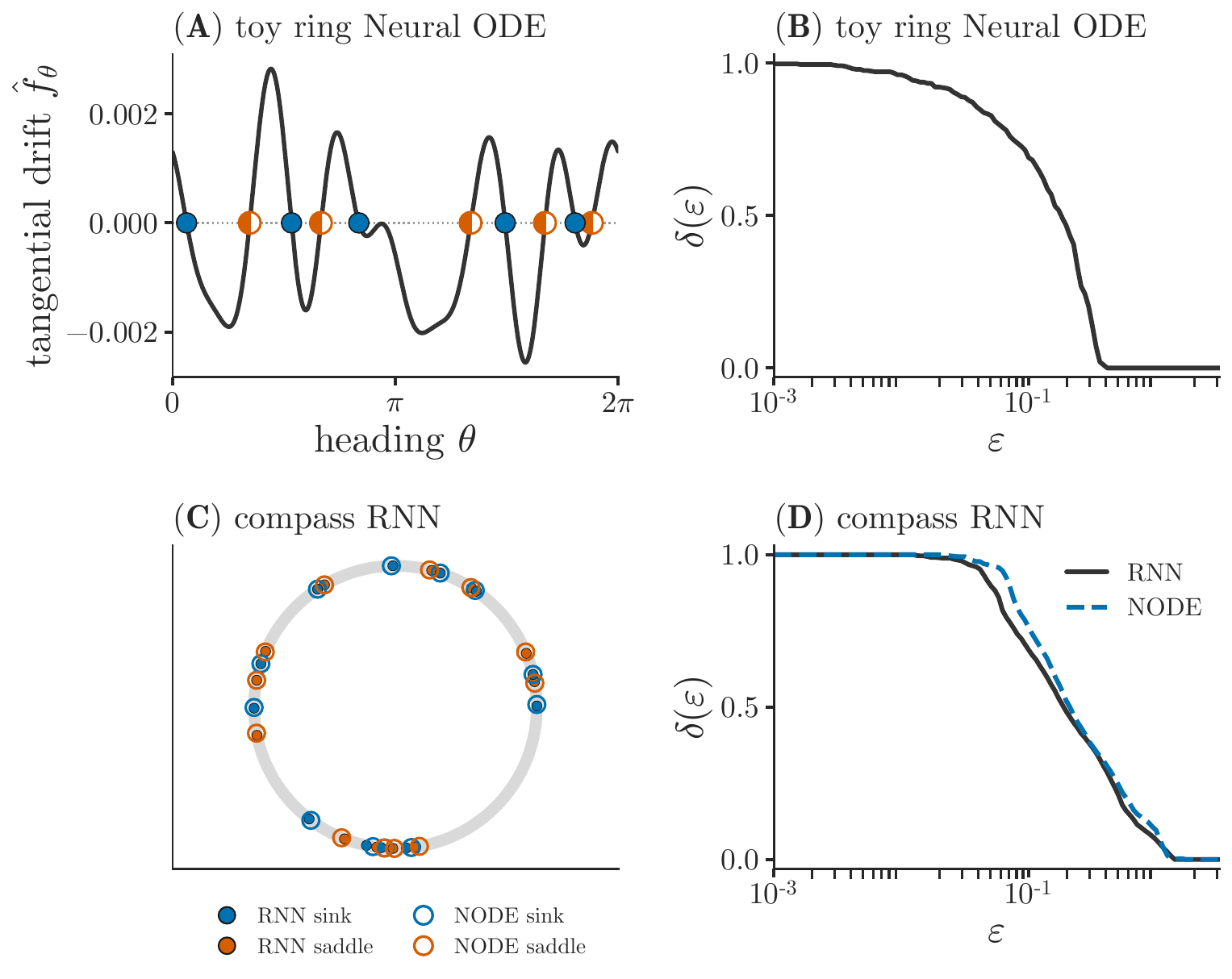}
\caption{Continuous attractors are discretized (D-type), and the discretization transfers to a Neural ODE. Top: a Neural ODE fit to the analytic ring attractor $\dot z = (1-\|z\|^2)z$. Bottom: a $64$-unit rectified-tanh network trained to integrate angular velocity~\citep{Sagodi2024a}, and an independent Neural ODE fitted to it.
\textbf{(A)} The toy ring's along-ring tangential drift; its zeros are isolated fixed points (filled circle: sink; half-filled: saddle), so the continuum is tiled by discrete equilibria, $5$ sinks alternating with $5$ saddles.
\textbf{(B)} $\varepsilon$-$\delta$ closeness of the toy Neural ODE to the analytic continuum: $\delta(\varepsilon)$ is the fraction of held headings whose trajectory error exceeds $\varepsilon$.
\textbf{(C)} Fixed points of the trained network on the ring, in canonical (readout) coordinates: $12$ sinks (blue) alternating with $12$ saddles (orange), all hyperbolic. Filled markers are the network's fixed points, open markers those of the fitted Neural ODE; they coincide ($22$ of $24$), the structurally-stable transfer of the discrete fixed-point set.
\textbf{(D)} Held-heading drift $\delta(\varepsilon)$ for the network (solid) and the Neural ODE (dashed): both discretize the continuum, and the Neural ODE reproduces the network's drift.}
\label{fig:exp_ca}
\end{figure}

\subsection{Budget and width scaling}\label{app:exp_budget}

Finally, we check the budget cascade of the proofs quantitatively. On a one-parameter bistable family we fit the field accuracy $\eta_0$ required to reach a target $(\varepsilon, \delta)$ to a power law $\eta_0 \propto \varepsilon^a \delta^b$. Across saddle rates spanning $L / \mu_u \in [1, 18]$ we measure $a = 0.94$ and $b = 1.02$ on average, so $\eta_0 = \Theta(\varepsilon \delta)$: the required vector-field accuracy is the symmetric product of the tolerance and the reliability, consistent with the linear basin-error scaling of Remark~\ref{rem:basin_scaling}. For the planar targets the trajectory error is linear in the field error through the origin (slope $1.18$ for Duffing), matching $\eta_{\mathrm{traj}} = \lambda \varepsilon / K$ of Lemma~\ref{lem:traj_bound}, and the approximation-limited field error scales as $\eta(W) \sim C W^{-p}$ with $p \in [1.1, 1.35]$, giving a width requirement $W \sim (\varepsilon \delta)^{-1/p}$ for these finite-width fields. These are composed sufficiency bounds and involve the problem constants $L, \lambda, \mu_u$, so they bound the required width from above; a matching lower bound is left open.

%==============================================================================
\section{Extended Literature Review}\label{app:literature}

This appendix provides a comprehensive review of universal approximation results for dynamical systems, establishing the context and novelty of our contributions.

\subsection{Approximation on Finite Time Intervals}

Universal approximation theorems for dynamical systems primarily address two regimes: finite-time simulation via high-dimensional embedding and direct trajectory approximation in the native state space.

\paragraph{Foundational results.}
Seminal works by~\citet{li1992approximation},~\citet{funahashi1993approximation} and ~\citet{doya1993universality} established that RNNs approximate general dynamical systems by embedding dynamics into high-dimensional hidden states.
This lineage was extended to the approximation of continuous functionals and nonlinear operators by~\citet{chen1993uap} and~\citet{chen1995uap}, broadening the scope of universality to map input signals to outputs.
This lineage of results extends back to homogeneous state-affine systems, which were shown to possess similar universality properties in the discrete-time setting for compact time intervals~\citep{fliess1980vers, sontag1979realization, sontag1979polynomial}.
Crucially, these guarantees are \emph{strictly limited to bounded time intervals}, leaving long-term asymptotic behavior unconstrained.

\paragraph{Control-theoretic universality.}
A complementary route to universality goes through geometric control theory: viewing a deep residual network or Neural ODE as a control system, its flow can be steered to approximate an arbitrary target map precisely when the constituent vector fields are bracket-generating, which is the nonlinear controllability criterion of~\citet{hermann1977nonlinear}.
This yields universal approximation and interpolation guarantees for residual networks and Neural ODEs~\citep{tabuada2022universal, cuchiero2020deep, ruiz2023neural, agrachev2022control}, and it is what underwrites the $C^1$ universal approximation property we assume of the hypothesis class (Definition~\ref{def:hypothesis_class}).
As with the embedding results above, however, these guarantees are finite-time: they control the flow map at a terminal time (for classification, interpolation, or transport), not the asymptotic behavior as $t \to \infty$ that our $\varepsilon$-$\delta$ conditions constrain.
The dual notion of observability~\citep{hermann1977nonlinear, sussmann1978single} governs whether such a vector field can in turn be identified from partial observations, a prerequisite for state-space interpretability that is complementary to, and outside the scope of, our approximation results.

\paragraph{Gr\"onwall-based bounds.}
Most finite-time results rely on Gr\"onwall's inequality to control trajectory differences, incurring exponential degradation of accuracy over time~\citep{sontag1992neural, funahashi1993approximation, chow2000modeling, li2005approximation}.
For time horizon $T$, if $\|f - \hat{f}\|_{C^0} < \eta$, then $\|\varphi(t) - \hat{\varphi}(t)\| \leq \eta(e^{Lt} - 1)/L$, which grows unboundedly as $t \to \infty$.

\paragraph{Discrete-time sequences.}
RNNs act as universal approximators for various temporal mappings, though often constrained to finite-time horizons. 
They can uniformly approximate the dynamics of continuous finite-memory systems on compact domains \citep{hammer2000approximation} and state-space trajectories over closed discrete-time intervals \citep{jin1995universal, patan2008approximation}. 
More recent work has explicitly established the finite-time universality of discrete-time RNNs \citep{aguiar2023universal} and linear recurrences with nonlinear projections \citep{orvieto2023universality}, alongside their ability to map continuous past-dependent sequence-to-sequence functions \citep{song2023minimal}.
Furthermore, CNNs have emerged as effective alternatives for modeling sequences with finite memory \citep{bai2018empirical, jiang2021approximation}.

\paragraph{Continuous-time trajectories.}
The universal approximation properties of static neural networks \citep{funahashi1989approximate} initially established them as powerful tools for modeling nonlinear systems \citep{narendra1990identification, warwick1992neural}, often employing multilayer perceptrons or radial basis functions \citep{chen1992adaptive, chen1992neural, choi1996constructive, tan1995efficient, garces2012strategies}. 
Expanding this to dynamic contexts, the capabilities of continuous-time recurrent networks for autonomous systems were demonstrated by \citet{funahashi1993approximation} and extended by \citet{kimura1998learning} and \citet{chow2000modeling}. 
For non-autonomous systems, any trajectory can be approximately realized on finite intervals \citep{nakamura2009approximation}, a property also held by Neural ODEs for input-driven systems \citep{li2022deep, zakwan2023universal}.
Furthermore, even networks with randomly initialized, fixed weights can achieve universal approximation of continuous functions and dynamical systems over compact intervals by learning only the bias parameters \citep{williams2024expressivity}.
Recent Neural ODE variants add frequency-domain~\citep{ashraf2026frequency} and energy-behavior-preserving~\citep{celledoni2026uepi} parameterizations, apply universal differential equations~\citep{rackauckas2020universal} to neural population data~\citep{elgazzar2024universal}, and enable data-driven phase control of learned limit-cycle oscillators~\citep{yawata2025data}.

For non-autonomous systems, where the dynamics explicitly depend on external inputs, it has been shown that the states of the output units of a continuous-time RNN, can approximate the solution of ODEs on compact time intervals \citep{garces2012strategies} and  linear, continuous and regular functionals \citep{li2021approximationencdec}.
A subset of an RNN's units can approximate any smooth ($C^1$) dynamical system (input-affine) with arbitrary precision for finite time \citep{kambhampati2000approximation}. 
Neural oscillators have the universal approximation property on compact time intervals (approximation of causal and continuous operators)  \citep{lanthaler2023neuraloscillators, huang2025upper}.
Any finite trajectory of an $n$-dimensional continuous dynamical system can be approximated by the internal state of the hidden units and $n$ output units of a Liquid time-constant (LTC) network~\citep{hasani2018liquid}, since extended with attractor-based memory for sequence modelling~\citep{kumar2025time}.
Finally, \citet{maass2007computational} extended these results to biological feedback models, proving that neural networks with dynamic synapses are universal approximators for dynamical systems, particularly in closed-loop feedback configurations.

\paragraph{Flow approximation.}
Flow-based approximations for finite time have been established for architectures including Recurrent High-Order Neural Networks~\citep{kosmatopoulos1995structural}.
Any continuous dynamical system can be approximately realized on finite intervals by an RNN~\citep{chow2000modeling}, extended to time-variant systems with fixed initial states~\citep{li2005approximation}.
More recently, \citet{tabuada2020universal} established that Neural ODEs are universal approximators for monotone analytic functions homotopic to the identity, proving that the flow of a single-layer system with time-varying parameters can approximate such mappings over compact sets. 
Despite these advances in representing complex mappings through flows, these results remain confined to finite-time horizons and do not address the long-term topological stability of multistable systems.

\subsection{Fading Memory Systems}\label{app:fading_memory}

For infinite time horizons, existing results predominantly require the \textbf{fading memory property} (FMP).

\begin{definition}[Fading Memory Property~{\citep{boyd1985fading}}]
An operator $H$ mapping input histories to outputs possesses the FMP if it is continuous with respect to a weighted norm that discounts the past: there exists a weighting function $w(z) \to 0$ as $z \to \infty$ such that closeness in the weighted past implies closeness in output.
\end{definition}

\paragraph{Topological constraint.}
The FMP imposes a severe restriction: systems with fading memory ``forget'' initial conditions, implying global asymptotic stability (monostability).
This \emph{inherently excludes multistable systems}, those with multiple fixed points or limit cycles where long-term behavior depends on initial conditions.
A simple bistable system has no fading memory: the final state depends entirely on the basin of attraction.

\paragraph{Reservoir computing.}
Echo State Networks (ESNs) were shown to be universal for fading memory systems by~\citet{jaeger2001echo}, formalized by~\citet{grigoryeva2018echo, grigoryeva2018universal}.
Liquid State Machines achieve universal computation for filters with FMP~\citep{maass2004computational}.
Various reservoir architectures approximate FMP systems: linear systems with polynomial readouts~\citep{boyd1985fading}, state-affine systems~\citep{gonon2020risk}, and Simple Cycle Reservoirs~\citep{li2024simple}.

\paragraph{Other stability conditions.} 
Various results establish universality for systems with decaying memory, including linear RNNs for regular functionals~\citep{li2021curse, li2022approximation}, State Space Models for exponentially decaying memory~\citep{wang2024state}, and fixed reservoirs with linear readouts~\citep{yasumoto2025universality}. 
Extending this to infinite time, \citet{hanson2020universal} and \citet{bishop2022universal, bishop2026recurrent} provide guarantees for deterministic and stochastic RNNs, respectively. However, these results strictly require incremental stability or exponential contraction to bound error propagation.
While effective for filtering, these conditions mathematically enforce a single global attractor, thereby precluding the history-dependent bifurcations and multistability essential for cognitive tasks.

\paragraph{FMP trivializes approximation.}
Every fading-memory system~\citep{sepulchre2021fading} can be uniformly approximated by linear state dynamics with a nonlinear readout~\citep{matthews1993approximating}.
Typically, the FMP is used to reduce the approximation problem to one over a finite, bounded index set, and then appeal to the density of fully connected neural network to obtain approximation  \citep{gonon2021fading}.
This constraint is so restrictive that \citet{boyd1985fading} proved FMP is the necessary and sufficient condition for a system to be uniformly approximated by a \textbf{finite Volterra series}~\citep{volterra1887sopra, franz2006wiener}.
The FMP implies ``asymptotic independence'' of the state from initial conditions~\citep{chua1976qualitative, manjunath2020stability, miller2018stable}, making multistability impossible.
Our work addresses this fundamental gap.

This result was extended to include approximation of systems with an input  \citep{manjunath2013echo}.
More recently, these results were rigorously formalized by \citet{grigoryeva2018echo,grigoryeva2018universal}, who proved that ESNs are universal uniform approximators for discrete-time fading memory filters, subject to uniformly bounded inputs.
In the domain of spiking neural networks, a parallel result holds: Liquid State Machines (LSMs) achieve universal computational power for all time-invariant filters that exhibit the FMP \citep{maass2004computational}.

Generalizing these results, we observe that universality is not unique to standard reservoir architectures.
If the domain of the functional $H$ is restricted to a space of uniformly bounded sequences with the fading memory property, various families of state-space transformations can approximate it uniformly.
These include linear systems with polynomial or neural network readouts \citep{coleman1968general,boyd1985fading,grigoryeva2018universal,gonon2019reservoir}, state-affine systems with linear readouts \citep{grigoryeva2018universal,gonon2020risk}, and the previously mentioned echo state networks \citep{grigoryeva2018echo,gonon2019reservoir,gonon2021fading,gonon2023approximation}.
More recently, Simple Cycle Reservoirs were also shown to be universal for this set of systems \citep{li2024simple}. 
Departing from these state-space approaches, a new method has been developed to approximate FMP systems using a kernel representation of the model \citep{huo2024kernel}.

\citet{dehghani2018universal} introduced the recurrently-stacked Universal Transformer, and \citet{yun2019transformers} showed that transformer models are universal approximators of continuous permutation-equivariant sequence-to-sequence functions with compact support.
Independently, \citet{bai2019deq} proposed the Deep Equilibrium Model (DEQ), an implicit-depth architecture based on root-finding for steady states of weight-tied deep networks; DEQs can approximate any sequence whose underlying dynamics converge to a fixed point.
The DEQ, which can be seen as an implicit-depth model similar to neural ordinary differential equations, was later extended to capture systems with arbitrary invariant sets; however, its universal approximation property has not yet been established \citep{konishi2023stable}.
Finally, assuming a global attractor property, infinite time approximation can be guaranteed~\citep{yi2023nmode,mei2024controlsynth}.

\subsection{Approximation of Diffeomorphisms and Flows}

\paragraph{Flow maps.}
Neural ODEs and invertible neural networks are universal approximators for diffeomorphisms~\citep{huang2018neural, jaini2019sum, teshima2020uap, teshima2020coupling, ishikawa2023uap, massaroli2020nodes,kuehn2023embedding}.
For dynamical systems, this guarantees only finite-time approximation: the flow $\varphi_t$ is a diffeomorphism for each fixed $t$, but approximating $\varphi_T$ does not control behavior for $t > T$.

\paragraph{Topological conjugacy.}
ESNs can be trained to have topologically conjugate dynamics to structurally stable systems~\citep{hart2020embedding}.
For NODEs, this follows immediately from structural stability (Theorem~\ref{thm:palis_smale}): $C^1$-close vector fields yield topologically equivalent flows.
Our contribution is making this observation precise with quantitative $\varepsilon$-$\delta$ bounds.
Finally, we would like to mention the possibility to construct dynamics with Cohen-Grossberg networks for arbitrary sets of fixed point and limit cycles through embedding of dynamics  \citep{cohen1992construction}.
Finally, models can be constructed to realize attractors \citep{danca2024hopfield}, though general infinite-horizon guarantees for these results are lacking.

\subsection{Why Multistability Requires New Theory}

\paragraph{Fundamental obstructions.}
Multistable systems exhibit two failure modes absent in FMP systems:
\begin{enumerate}[leftmargin=*]
\item \textbf{B-type error}: Near separatrices, arbitrarily small perturbations cause trajectories to converge to different attractors (Figure~\ref{fig:separatrices}).
\item \textbf{P-type error}: For limit cycles, any period mismatch causes unbounded phase drift as $t \to \infty$.
\end{enumerate}

\paragraph{Structural stability as the key.}
Morse-Smale systems are the natural target class: they are structurally stable (small perturbations preserve qualitative dynamics) yet allow multistability.
This paper provides universal approximation results exploiting this property.

\paragraph{Outside the Morse-Smale class.}
Chaotic attractors (C-type error) require fundamentally different metrics, for example, invariant measures or attractor reconstruction~\citep{dellnitz1999approximation, hart2024attractor}, and remain outside our framework.

\subsection{Computational Universality}

For completeness, we note that RNNs are Turing universal~\citep{moore1990unpredictability, pollack1991induction, siegelmann1992computational, kilian1996universality, moore1998finite, siegelmann1994analog}.
This computational universality mirrors that of smooth ODEs~\citep{branicky1995universal}, Hopfield nets~\citep{sima2003continuous}, and finite automata simulations~\citep{indyk1995optimal, kremer1995computational, sperduti1997computational, korsky2019computational}.
However, computational universality is orthogonal to approximation: it concerns symbolic computation, not trajectory tracking.
Our results address the latter.

\begin{landscape}
\subsection{Overview}
\begin{table}[ht]
    \centering
    \caption{Comprehensive Landscape of Universal Approximation for Dynamical Systems}
    \label{tab:uap_landscape}
    \small 
    % We expand to 1.6\textwidth to use the full landscape page width.
    % We define 4 X-columns. The sum of coefficients (1.6 + 0.8 + 0.8 + 0.8) = 4.0.
    \begin{tabularx}{1.6\textwidth}{@{}l >{\hsize=1.3\hsize}X l >{\hsize=0.8\hsize}X >{\hsize=0.8\hsize}X >{\hsize=0.99\hsize}X@{}}
        \toprule
        \textbf{Category} & \textbf{Citations} & \textbf{Hor.} & \textbf{Metric / Guarantee} & \textbf{Stability / Constraint} & \textbf{Asymptotic Property} \\ 
        \midrule
        Computation & \citep{moore1990unpredictability, pollack1991induction, siegelmann1992computational, siegelmann1994analog, kilian1996universality, moore1998finite, branicky1995universal, sima2003continuous, indyk1995optimal, kremer1995computational, sperduti1997computational, korsky2019computational} & Inf. & Turing universality & Analog computation & Recursive equivalence \\ 
        \addlinespace
        RNN & \citep{li1992approximation, funahashi1993approximation, doya1993universality} & Fin. & $L^\infty$ Trajectory tracking & Compact state space & Local existence \\ 
        \addlinespace
        Operator & \citep{chen1993uap, chen1995uap, narendra1990identification, warwick1992neural, lanthaler2023neuraloscillators, huang2025upper, chen1992adaptive, chen1992neural, choi1996constructive, tan1995efficient, garces2012strategies, kambhampati2000approximation, hasani2018liquid} & Fin. & Operator norm & Continuous signals & Signal mapping \\ 
        \addlinespace
        Fin. seq. & \citep{jin1995universal, patan2008approximation, hammer2000approximation, aguiar2023universal, orvieto2023universality, song2023minimal, bai2018empirical, jiang2021approximation, yun2019transformers, dehghani2018universal} & Fin. & Uniform convergence & Finite time window & Temporal mapping \\ 
        \addlinespace
        Grönwall & \citep{sontag1992neural, funahashi1993approximation, chow2000modeling, li2005approximation} & Fin. & Exp. error bound & Lipschitz continuity & Vacuous ($t \to \infty$) \\ 
        \addlinespace
        Flow & \citep{tabuada2020universal, li2022deep, zakwan2023universal, nakamura2009approximation, psaltis1988multilayered, li1989control, parisini1998neural} & Fin. & Trajectory realization &  & Flow approximation \\ 
        \addlinespace
        Diffeo. & \citep{huang2018neural, jaini2019sum, teshima2020uap, teshima2020coupling, ishikawa2023uap, massaroli2020nodes, kuehn2023embedding} & Fin. & Diffeo. mapping & Homotopy to identity & Coordinate warping \\ 
        \addlinespace
        FMP & \citep{boyd1985fading, jaeger2001echo, grigoryeva2018echo, grigoryeva2018universal, sepulchre2021fading, matthews1993approximating, gonon2021fading, maass2004computational, maass2007computational, gonon2019reservoir, gonon2020risk, gonon2023approximation, li2024simple, yi2023nmode} & Inf. & Decaying history & Unique global eq. & Global monostability \\ 
        \addlinespace
        Contractive & \citep{bishop2022universal, bishop2026recurrent, bai2019deq, konishi2023stable, li2021curse, li2022approximation, wang2024state} & Inf. & Stationary $L^p$ error & Exp. $p$-contraction & Unique measure \\ 
        \addlinespace
        Incr. stab. & \citep{hanson2020universal, chua1976qualitative, manjunath2020stability, miller2018stable} & Inf. & Uniform tracking & Asymptotic incr. stab. & Trajectory sync. \\ 
        \addlinespace
        Top. equiv. & \citep{hart2020embedding, cohen1992construction} & Inf. & Qualitative tracking & Structural stability & Topological conjugacy \\ 
        \addlinespace
        Formal Sim. & \citep{girard2008approximate, girard2009hierarchical, pola2004bisimulation, vanderschaft2004bisimulation, zamani2014symbolic} & Inf. & Bisimulation distance & Metric / Symbolic & Simulation equivalence \\ 
        \midrule
        \textbf{This Work} & \textbf{Thrm.~\ref{thm:uap_fp}, \ref{thm:uap_lc}, \ref{thm:uap_ca}} & \textbf{Inf.} & \textbf{$\varepsilon$-$\delta$ closeness} & \textbf{Morse-Smale / NHIM} & \textbf{Multistable attractors} \\ 
        \bottomrule
    \end{tabularx}
\end{table}
\end{landscape}

% NeurIPS paper checklist (make_arxiv.sh drops these two lines for the arXiv build)

\end{document}